\documentclass[11pt]{article}

\usepackage{latexsym}
\usepackage{amssymb}
\usepackage{amsthm}
\usepackage{amscd}
\usepackage{amsmath}
\usepackage{mathrsfs}
\usepackage{graphicx}
\usepackage{hyperref}
\usepackage{shuffle}
\usepackage{mathtools}
\usepackage[heads=vee]{diagrams} 
\usepackage[bbgreekl]{mathbbol}
\usepackage{ytableau}
\usepackage{bm}

\usepackage[all]{xy}
\input xy \xyoption{frame}
\xyoption{dvips}

\usepackage[colorinlistoftodos]{todonotes}
\usetikzlibrary{chains,scopes}

\DeclareSymbolFontAlphabet{\mathbb}{AMSb}
\DeclareSymbolFontAlphabet{\mathbbl}{bbold}

\theoremstyle{definition}
\newtheorem* {theorem*}{Theorem}
\newtheorem* {conjecture*}{Conjecture}
\newtheorem{theorem}{Theorem}[section]

\theoremstyle{definition}

\newtheorem* {example*}{Example}

\newtheorem{lemma}[theorem]{Lemma}
\theoremstyle{definition}
\newtheorem{definition}[theorem]{Definition}
\theoremstyle{definition}

\newtheorem{conjecture}[theorem]{Conjecture}
\newtheorem{proposition}[theorem]{Proposition}
\newtheorem{corollary}[theorem]{Corollary}

\newtheorem *{remark}{Remark}
\theoremstyle{definition}
\newtheorem {example}[theorem]{Example}
\theoremstyle{definition}

\theoremstyle{definition}

\theoremstyle{definition}

\xyoption{dvips}

\def\wh{\widehat}
\def\({\left(}
\def\){\right)}

\newcommand{\cO}{\mathcal{O}}

\newcommand{\cZ}{\mathcal{Z}}

\def\NN{\mathbb{N}}

\def\ZZ{\mathbb{Z}}

\def\ch{\mathrm{ch}}
\def\spanning{\textnormal{-span}}

\def\barr{\begin{array}}
\def\earr{\end{array}}
\def\ba{\begin{aligned}}
\def\ea{\end{aligned}}
\def\be{\begin{equation}}
\def\ee{\end{equation}}

\def\qquand{\qquad\text{and}\qquad}
\def\quand{\quad\text{and}\quad}

\def\cH{\mathcal H}

\def\id{\mathrm{id}}
\def\PP{\mathbb{P}}

\def\ben{\begin{enumerate}}
\def\een{\end{enumerate}}

\newcommand{\Mon}[1]{\mathscr{F}(#1)}

\def\Des{\mathrm{Des}}

\def\c{\textbf{c}}

\newcommand{\Sp}{\textsf{Sp}}

\renewcommand{\dim}{\operatorname{dim}}

\renewcommand{\r}[1]{\textcolor{red}{#1}}

\def\L{\underline L}

\def\arcstart{\ \xy<0cm,-.15cm>\xymatrix@R=.1cm@C=.3cm }
\newcommand{\arcstartc}[1]{\ \xy<0cm,-.15cm>\xymatrix@R=.1cm@C=#1cm}

\def\t{\mathbf{t}}
\def\r{\mathbf{r}}

\def\sC{\mathscr{C}}

\def\Fock{\overline{\mathcal{K}}}
\def\F{\mathcal{F}}

\def\bfSigma{\mathbf{K}}
\def\sSigma{\mathbb{K}}
\def\kk{\mathbbl{k}}

\def\bfW{\textbf{W}}
\def\bfI{\textbf{I}}

\def\PR{\textsf{PR}}
\def\KPR{\textsf{KPR}}

\def\Words{\mathbb{W}}

\def\cG{\mathcal{G}}

\def\FB{\textsf{FB}}

\def\Mon{\textsf{Mon}}
\def\Comon{\textsf{Comon}}
\def\Bimon{\textsf{Bimon}}

\def\kVec{\textsf{Vec}_\kk}
\def\whkVec{\wh{\textsf{Vec}}_\kk}
\def\kGrVec{\textsf{GrVec}_\kk}

\def\L{\mathbb{L}}

\def\One{\mathbf{1}}

\def\antipode{{\tt S}}

\def\Bimon{\textsf{Bimon}}

\def\htimes{\mathbin{\hat\otimes}}

\newcommand{\Sym}{\textsf{Sym}}
\newcommand{\WQSym}{\textsf{WQSym}}
\newcommand{\FQSym}{\textsf{FQSym}}

\def\QSym{\textsf{QSym}}
\def\whQSym{\wh{\textsf{Q}}\textsf{Sym}}
\def\whSym{\wh{\textsf{S}}\textsf{ym}}
\def\Sym{\textsf{Sym}}
\def\NSym{\textsf{NSym}}
\def\whNSym{\wh{\textsf{N}}\textsf{Sym}}

\def\zetaq{\zeta_{\QSym}}

\def\r{{\tt r}}
\def\c{{\tt c}}
\def\t{{\tt t}}

\def\cV{\mathscr{V}}
\def\cU{\mathscr{U}}
\def\cH{\mathscr{H}}
\def\B{\mathscr{B}}

\def\ExtBim{\Mon(\Comon^\FB)}
\def\FDExtBim{\Mon({\Comon}^{\FB}_{\emph{fin-dim}})}

\def\P{\textsf{P}}
\def\Packed{\mathbb{W}_\P}
\def\bfP{{\textbf{W}}_\P}
\def\hP{\hat{\textbf{W}}_\P}

\def\Reduced{\textsf{R}}

\def\bfE{\mathcal{E}}

\def\bF{\mathscr{F}}
\def\bG{\mathscr{G}}

\def\sW{\mathscr{W}}
\def\sS{\mathscr{K}}
\def\XX{\mathbb{X}}

\def\cZ{\text{Z}}
\def\Peak{\mathrm{Peak}}
\def\Valley{\mathrm{Val}}
\def\OQSym{\cO\QSym}
\def\OSym{\cO\Sym}

\def\ch{\operatorname{char}}
\def\sort{\mathrm{sort}}
\def\osim{\mathbin{\overset{\circ}\sim}}
\def\rpk{\flat}
\def\mMR{\mathfrak{m}\text{MR}}
\def\Zetaq{\cZ_{\QSym}}
\def\EQSym{\bfE\whQSym}

\def\st{\operatorname{fl}}

\usepackage{fullpage}

\numberwithin{equation}{section}
\allowdisplaybreaks[1]
\UseCrayolaColors

\begin{document}
\title{Linear compactness and combinatorial bialgebras}
\author{
Eric Marberg
\\ Department of Mathematics \\  Hong Kong University of Science and Technology \\ {\tt eric.marberg@gmail.com}
}

\date{}

\maketitle

\begin{abstract}
We present an expository overview of the monoidal structures in the category of linearly compact vector spaces.
Bimonoids in this category are the
natural duals of infinite-dimensional bialgebras.
We classify the relations on words whose equivalence classes generate linearly compact bialgebras
under shifted shuffling and deconcatenation.
We also extend some of the theory of combinatorial Hopf algebras
to bialgebras that are not connected or of finite graded dimension.
Finally, we discuss several examples of quasi-symmetric functions, not necessarily of bounded degree,
that may be constructed via terminal properties of combinatorial bialgebras. 
\end{abstract}

\setcounter{tocdepth}{2}
\tableofcontents

\section{Introduction}

The graded dual $\bfW_\P$ of the Hopf algebra of \emph{word quasi-symmetric functions}
has a basis given by the set of \emph{packed words}, i.e.,
finite sequences $w=w_1w_2\cdots w_n$ with $\{w_1,w_2,\dots,w_n\} = \{1,2,\dots,m\}$ for some $m \geq 0$.
The product for this Hopf algebra is a shifted shuffling operation, while the coproduct is a variant of deconcatenation;
for the precise definitions, skip to Section~\ref{word-sect}.

A fruitful method of constructing Hopf algebras of interest in combinatorics is to 
choose 
an equivalence relation $\sim$ on packed words and then form
the subspace $\bfSigma_\P^{(\sim)} \subset \bfW_\P$
spanned by the sums over each $\sim$-equivalence class $\kappa_E := \sum_{w \in E} w$.
A long list of well-known Hopf algebras can be realized as a subalgebra of $\bfW_\P$ in this way:
for example,
the \emph{noncommutative symmetric functions} $\NSym$ \cite{GKLLRT}, 
the \emph{Poirier-Reutenauer algebra} $\PR$ \cite{PoirierReutenauer},
the \emph{$K$-theoretic Poirier-Reutenauer algebra} $\KPR$ \cite{PylPat},
the \emph{small multi-Malvenuto-Reutenauer Hopf algebra} $\mMR$ \cite{LamPyl},
the \emph{Loday-Ronco algebra} \cite{AguiarSottile2,LR1},
and the \emph{Baxter Hopf algebra} \cite{Giraudo}. Similar Hopf algebra constructions involving
equivalences on (signed) words and permutations have been explored in \cite{ChatelPilaud,Pilaud,PilaudPons,Reading},
among other places.

The subspace $\bfSigma_\P^{(\sim)} \subset \bfW_\P$ is not necessarily a sub-bialgebra,
 and one of the aims of this paper is to describe precisely when this occurs.
The Hopf algebra $\bfW_\P$ is a quotient of a larger bialgebra $\bfW$ with a basis given by 
arbitrary words.
We will also consider the problem of classifying the word relations 
 that span sub-bialgebras  $\bfSigma^{(\sim)}\subset \bfW$ in a similar manner.
 
 For homogeneous relations, 
 versions of
 these problems have been studied in a few places previously, e.g., \cite{Giraudo,Hiver07,NovelliThibon,Priez}.
 Less has been written about the cases when $\sim$ is allowed to relate words of different lengths.
For inhomogeneous relations of this kind,
various complications arise
when one tries to interpret 
  $\bfSigma_\P^{(\sim)}$  as an algebra or a coalgebra.
  To start, such relations may have equivalence classes with infinitely many elements,
  in which case $\bfSigma_\P^{(\sim)}$ contains infinite linear combinations of packed words
  so is not technically a subspace of $\bfW_\P$. One can still try to evaluate the product and coproduct 
  of $\bfW_\P$ on elements of $\bfSigma_\P^{(\sim)}$ when this happens. However, products may result in infinite linear combinations of 
  the basis elements $\kappa_E$, and even if these infinite sums are adjoined to $\bfSigma_\P^{(\sim)}$,  
  coproducts may have too many terms to belong to $\bfSigma_\P^{(\sim)}\otimes \bfSigma_\P^{(\sim)}$.
  
Nevertheless, some interesting ``Hopf algebras'' that can be identified with $\bfSigma_\P^{(\sim)}$
  when $\sim$ is inhomogeneous have appeared in the literature \cite{HKPWZZ,LamPyl,Pat,PylPat}.
  A secondary, expository goal of this paper is to describe explicitly the 
monoidal category containing such objects, which in general is not the usual category of bialgebras over a field.
This point is often glossed over in the relevant combinatorial literature, though authors 
tend to indicate correctly that its resolution is topological in nature.

In detail, to make sense of ``sub-bialgebras'' of $\bfW_\P$ ``spanned'' by inhomogeneous word relations,
one should first consider the larger vector space $\hat\bfW_\P$ consisting of arbitrary 
(rather than just finite) linear combinations of packed words.
This object is naturally viewed as a \emph{linearly compact} topological space.
The full subcategory of such spaces, within the category of all topological vector spaces,
has a symmetric monoidal structure which leads to notions of 
\emph{linearly compact algebras, coalgebras}, and \emph{bialgebras}, of which $\hat\bfW_\P$ is an example.
In this language, our original classification problem becomes the question:
for which word relations $\sim$ is the subspace 
$\hat\bfSigma_\P^{(\sim)}$,
whose elements are the arbitrary linear combinations of the sums $\kappa_E$,
a linearly compact sub-bialgebra of $\hat\bfW_\P$?
  
After some preliminaries in Section~\ref{prelim-sect}, we review the main properties
of
linearly compact vector spaces in Section~\ref{cont-sect}.
This background material is semi-classical but perhaps not so widely known
in combinatorics.
Section~\ref{species-sect}
goes on to discuss some novel generalizations of the monoidal structures on $\bfW$ and $\bfW_\P$.
  In Section~\ref{wr-sect}, we answer the question in the previous paragraph.
Our general results about word relations recover a number of 
specific constructions of (linearly compact) Hopf algebras and bialgebras;
  we discuss some relevant examples in Section~\ref{example-sect}.
  
  One application of all this formalism 
  is to extend Aguiar, Bergeron, and Sottile's theory of combinatorial Hopf algebras from \cite{ABS}.
  Ignoring some technical details which will be clarified in Section~\ref{cb-sect},
  a \emph{combinatorial Hopf algebra} over a field $\kk$
  is a Hopf algebra $H$ with an algebra morphism $\zeta : H \to \kk$ called the \emph{character}.
A morphism $(H,\zeta) \to (H',\zeta')$ of combinatorial Hopf algebras 
is a Hopf algebra morphism $\phi :H \to H'$ with $\zeta = \zeta'\circ \phi$.
The Hopf algebra of \emph{quasi-symmetric functions} $\QSym$
with the homomorphism $\zetaq : \QSym \to \kk$
setting $x_1=1$ and $x_2=x_3=\dots=0$
is a fundamental example.

It is shown in \cite{ABS} that if $(H,\zeta)$ is a combinatorial Hopf algebra
in which $H$ is (1) graded, (2) connected, and (3) of finite graded dimension,
then there is a unique morphism $(H,\zeta) \to (\QSym,\zetaq)$.
This morphism supplies a uniform construction of many independent definitions
of quasi-symmetric generating functions attached to Hopf algebras.
In Section~\ref{cb-sect}, we prove two extensions of this result. The first (see Theorem~\ref{abs-thm}) removes assumptions (2) and (3), essentially just
by reframing the character of $H$ as an algebra morphism $\zeta : H \to \kk[t]$.
The second (see Theorem~\ref{cont-abs-thm}) lifts all of the assumptions (1), (2), and (3), at the cost of introducing some topological conditions
and replacing $\QSym$ by an appropriate completion.

These results  are not unexpected;
the authors mention in \cite[Remark 4.2]{ABS}
that assumption (3) may be dropped, and 
note work in preparation where this will be proved.
The relevant paper cited in \cite[Remark 4.2]{ABS} does not seem to have ever appeared in the literature, however.
We hope that our exposition fills this gap.

In Section~\ref{comwor-sect} we illustrate some more applications.
We discuss several examples of families of symmetric and quasi-symmetric functions, not necessarily of bounded degree,
that can be realized as the images of canonical morphisms from what we call \emph{(linearly compact) combinatorial bialgebras}.
For appropriate word relations, the space $\hat\bfSigma_\P^{(\sim)}$ is an object of this type and is therefore equipped
with a canonical morphism 
to a certain linearly compact ``completion'' of $\QSym$.
Our last results give a partial classification of the relations $\sim$ for which the image of this morphism consists entirely of symmetric functions.

\subsection*{Acknowledgements}

I am grateful to Frederick Tsz-Ho Fong, Zachary Hamaker, Amy Pang, and Brendan Pawlowski
for useful discussions.
I also thank the anonymous referees for their exceedingly thorough comments,
and for suggesting Example~\ref{referee-ex} as well as improvements and corrections to several proofs.

\section{Preliminaries}\label{prelim-sect}

Let $\ZZ \supset \NN \supset \PP$ denote the respective sets of all integers, all nonnegative integers,
and all positive integers.
For $m,n \in \NN$, define $[m,n] = \{ i \in \ZZ : m \leq i \leq n\}$ and $[n] = [1,n]$.

\subsection{Monoidal structures}\label{monoidal-sect}

Our reference for the background material in this section is \cite[Chapter 1]{AguiarMahajan}.
Suppose $\sC$ is a braided monoidal category with tensor product $\bullet$, unit object $I$,
and braiding $\beta$. 

\begin{definition}
A \emph{monoid} in $\sC$ is a triple $(A,\nabla,\iota)$ where $A \in \sC$ is an object
and $\nabla : A \bullet A \to A$ and $\iota : I \to A$ are morphisms (referred to as the \emph{product} and \emph{unit}) making these diagrams commute:
\be\label{assoc0-eq}
{\footnotesize
\begin{diagram}[small]
I\bullet A & \rTo^{\ \ \iota\bullet\id\ \ } & A\bullet A & \lTo^{\ \ \id\bullet \iota\ \ } & A \bullet I \\
 & \rdTo^{\cong} & \dTo^{\nabla} &  \ldTo^{\cong}\\
 & & A 
 \end{diagram}
\qquad\qquad
\begin{diagram}[small]
 A\bullet A \bullet A & \rTo^{\ \ \nabla\bullet \id\ \ }& A \bullet A\\
\dTo^{\id \bullet \nabla}  && \dTo_{\nabla} \\
A \bullet A & \rTo^{\nabla} & A 
\end{diagram}
}
\ee
\end{definition}
\begin{definition}
A \emph{comonoid} in $\sC$ is a triple $(A,\Delta,\epsilon)$ where $A \in \sC$ is an object
and $\Delta : A \to A \bullet A$ and $\epsilon : A \to I$ are morphisms (referred to as the \emph{coproduct} and \emph{counit}) making the
diagrams \eqref{assoc0-eq}, with $\nabla$ and $\iota$ replaced by $\Delta$ and $\epsilon$ and with the directions of all arrows reversed, commute.
\end{definition}

A monoid is \emph{commutative} 
if $\nabla\circ \beta = \nabla$.
A comonoid is \emph{cocommutative} if $\beta\circ \Delta = \Delta$.

\begin{definition}
A \emph{bimonoid} in $\sC$ is a tuple $(A,\nabla,\iota,\Delta,\epsilon)$ where
$(A,\nabla,\iota)$ is a monoid, $(A,\Delta,\epsilon)$ is a comonoid,
the composition $\epsilon\circ \iota$ is the identity morphism $I\to I$, and these diagrams commute:
\be\label{compat-eq}
{\footnotesize
 \begin{diagram}[small]
 A\bullet A&\rTo^{\nabla} & A & \rTo^{\Delta} & A \bullet A \\
\dTo^{\Delta \bullet \Delta}  & &&& \uTo_{\nabla \bullet \nabla} \\
A \bullet  A \bullet A \bullet A  && \rTo^{\id\bullet \beta \bullet \id} &&
 A \bullet A \bullet A  \bullet A
\end{diagram}
\qquad\qquad
\begin{diagram}[small]
I & \rTo^{\iota} & A
\\ 
\dTo^{\cong} && \dTo_{\Delta} \\
I\bullet I & \rTo^{\ \ \iota \bullet \iota\ \ } & A \bullet A 
\end{diagram}
\qquad\qquad
\begin{diagram}[small]
A \bullet A& \rTo^{\ \ \epsilon\bullet \epsilon\ \ } & I \bullet I
\\ 
\dTo^{\nabla} && \dTo_{\cong} \\
A & \rTo^{\epsilon} & I
\end{diagram}
}
\ee
\end{definition}

A morphism of (bi, co) monoids is a morphism  in $\sC$ that commutes with the relevant (co)unit and (co)product morphisms.
If $A$ is a monoid then $A \bullet A$ is a monoid with product $(\nabla \bullet \nabla) \circ (\id \bullet \beta \bullet \id)$ and unit $(\iota \bullet \iota) \circ (I \xrightarrow{\sim} I \bullet I)$.
If $A$ is a comonoid then $A\bullet A$ is naturally a comonoid in a similar way.
The diagrams \eqref{compat-eq} express that the coproduct and counit of a bimonoid are monoid morphisms,
and that the product and unit are comonoid morphisms.

We are exclusively interested in these definitions applied to a few related categories.
Let $\kk$ be a field and write $\kVec$ for the usual category of $\kk$-vector spaces with
linear maps as morphisms. This category is symmetric monoidal relative to the standard tensor product $\otimes =\otimes_\kk$
and braiding map $x\otimes y \mapsto y \otimes x$,
with unit object $\kk$.
Monoids, comonoids, and bimonoids in this category are the familiar notions of \emph{$\kk$-algebras}, \emph{$\kk$-coalgebras}, and \emph{$\kk$-bialgebras}.
In this context, the unit $\iota : \kk \to A$ is completely determined by $\iota(1) \in A$, which we refer to as the \emph{unit element}.

Assume that $\sC$ is $\kk$-linear so that
the morphisms between any two fixed objects in $\sC$ form a $\kk$-vector space.
Let $(H,\nabla,\iota,\Delta,\epsilon)$ be a bimonoid in $\sC$.
The \emph{convolution product} of two morphisms $f,g : H \to H$ 
is then $f* g = \nabla \circ (f\bullet g) \circ \Delta : H \to H$.
The operation $* $ is associative and makes the vector space of morphisms $H \to H$ into a $\kk$-algebra with unit element $\iota \circ \epsilon$,
referred to as the \emph{convolution algebra} of $H$.
The bimonoid $H$ is a \emph{Hopf monoid} if the identity morphism $\id : H \to H$ has a left and right inverse $\antipode : H \to H$ in the convolution algebra.
The morphism $\antipode$ is called the \emph{antipode} of $H$; if it exists, then 
$\antipode$ is the unique morphism $H \to H$ such that 
 $\nabla \circ (\id \bullet \antipode) \circ \Delta=\nabla \circ (\antipode \bullet \id) \circ \Delta=\iota\circ \epsilon$.
Hopf monoids in $\kVec$ are  \emph{Hopf algebras}.

%
%
%
%

\subsection{Graded vector spaces}\label{gr-sect}

If $I$ is a set and $V_i$ for $i \in I$ is a $\kk$-vector space, then
 $\bigoplus_{i \in I} V_i$ is the vector space of sums $\sum_{i \in I} v_i$ where $v_i \in V_i$ for $i \in I$
and $v_i=0$ for all but finitely many indices $i \in I$.
We interpret the direct product $\prod_{i \in I} V_i$ as the vector space of arbitrary formal sums $\sum_{i \in I} v_i$ with $v_i \in V_i$.
There is an obvious inclusion $\bigoplus_{i \in I} V_i \subset \prod_{i \in I} V_i$ which is equality if $I$ is finite.

A vector space $V$ is \emph{graded} if it has a direct sum decomposition $V = \bigoplus_{n\in \NN} V_n$.
A linear map $\phi : U \to V$ between graded vector spaces is \emph{graded} if it has 
the form $\phi = \bigoplus_{n \in \NN} \phi_n$ where each $\phi_n: U_n \to V_n$ is linear.
If $U = \prod_{n \in \NN} U_n$ and $V = \prod_{n \in \NN} V_n$
are direct products of vector spaces, then 
we also use the term \emph{graded} to refer to the linear maps
$\phi : U \to V$
of the form $\phi  = \prod_{n \in \NN} \phi_n$ where each $\phi_n : U_n \to V_n$ is linear.

An algebra $(V,\nabla,\iota)$ is \emph{graded} if $V$ is graded, 
 $\nabla(V_i\otimes V_j) \subset V_{i+j}$ for all $i,j \in \NN$, and $\iota(\kk) \subset V_0$.
Similarly, a coalgebra $(V,\Delta,\epsilon)$ is \emph{graded} if $V$ is graded,
$\Delta(V_n) \subset \bigoplus_{i+j = n} V_i \otimes V_j$ for all $n \in \NN$, and $\epsilon(V_n) = 0$ for $n \in \PP$.
A bialgebra is \emph{graded} if it is graded as both an algebra and a coalgebra.
These notions correspond to 
(co, bi) monoids in
the category $\kGrVec$ whose objects are graded $\kk$-vector spaces $V = \bigoplus_{n \in \NN} V_n$ and
 whose morphisms are graded linear maps,
 in which the tensor product of objects $U$ and $V$
 is the graded vector space $U \otimes V = \bigoplus_{n \in \NN} (U\otimes V)_n$
with  $(U\otimes V)_n = \bigoplus_{i+j = n} U_i \otimes V_j$.
The unit object in $\kGrVec$ 
is the field
$\kk$,  graded such that all elements have degree zero.

\subsection{Word bialgebras}\label{word-sect}

We review the definition of a particular graded bialgebra
which will serve as a running example in later sections.
Throughout, we use the term \emph{word} to mean a finite sequence of positive integers.
If $w =w_1w_2\cdots w_n$ is a word with $n$ letters
and $I=\{i_1 < i_2 < \dots < i_k\} \subset[n]$ is a subset of indices, then
we set $w|_I = w_{i_1}w_{i_2}\cdots w_{i_k}$.
The shuffle product of two words $u$ and $v$ of length $m$ and $n$
is the formal linear combination of words
\[u \shuffle v = \sum_{\substack{I\subset [m+n] \\ |I| = m}} \shuffle_I(u,v)\] where $\shuffle_I(u,v)$ is the unique $(m+n)$-letter word $w$ with $w|_I = u$ and $w|_{I^c} = v$.
Multiplicities may result in this expression; for example, $12\shuffle 21 = 2\cdot 1221  + 1212 + 2121 + 2 \cdot 2112$.

If $w = w_1w_2\cdots w_m$ is a word with $m>0$ letters,
then we set $\max(w)=\max \{ w_1,w_2,\dots,w_m\}$.
For the empty word $\emptyset$, we define $\max(\emptyset) = 0$.
Let $\Words_n$ for $n \in \NN$ be the set of pairs $[w,n]$ with $\max(w) \leq n$
and define $\Words = \bigcup_{n \in \NN} \Words_n$.
Let $\bfW_n = \kk \Words_n$ be the $\kk$-vector space with $\Words_n$ as a basis and define $\bfW=\bigoplus_{n \in \NN} \bfW_n$.

Denote the word formed by adding $n \in \NN$
to each letter of $w=w_1w_2\cdots w_m$ by 
 \[w\uparrow n=(w_1+n)(w_2+n)\cdots (w_m+n).\]
Given words $w^1,w^2,\dots, w^l$  with $\max(w^i) \leq n$ and $a_1,a_2,\dots,a_l \in \kk$,  
let $\left[ \sum_{i } a_i w^i, n \right] = \sum_{i} a_i [w^i,n] \in \bfW_n$.
Now define $\nabla_\shuffle : \bfW \otimes \bfW \to \bfW $ to be the linear map with
\be\label{nabla-eq}
\nabla_\shuffle([v,m]\otimes [w,n]) =  [v \shuffle (w\uparrow m), n+m] \in \bfW_{m+n}
\ee
for $[v,m] \in \Words_m$ and $[w,n] \in \Words_n$.
Since $v$ and $w\uparrow m$ are words with disjoint sets of letters, there are no multiplicities in the right expression;
for example,
$\nabla_\shuffle([12,3]\shuffle [2,2]) = [125,5] + [152,5] +  [512,5]$.
Next let  $\epsilon_\odot : \bfW \to \kk$ and $\Delta_\odot : \bfW \to \bfW \otimes \bfW$
be the linear maps with\be
\label{odot-maps}
 \epsilon_\odot([w,n]) = \begin{cases} 
1 &\text{if } w = \emptyset \\
0 &\text{otherwise}
\end{cases}
\quand
\Delta_{\odot}([w,n]) = \sum_{i=0}^m [w_1\cdots w_i, n] \otimes [w_{i+1}\cdots w_m,n]
\ee
for $[w,n] \in \Words_n$ with $w=w_1w_2\cdots w_m$.
Finally write $\iota_\shuffle $ for the linear map $\kk \to \bfW$ with 
$
\iota_\shuffle(1) = [\emptyset, 0].
$
We consider $\bfW$ to be a graded vector space in which $[w,n] \in \Words_n$ is homogeneous with degree $\ell(w)$, the length of the word $w$.
The following is \cite[Theorem 3.5]{M1}:

\begin{theorem}\label{w-thm}
$(\bfW, \nabla_\shuffle, \iota_\shuffle, \Delta_\odot, \epsilon_\odot)$ is a graded bialgebra, but not a Hopf algebra.
\end{theorem}

Let $w=w_1w_2\cdots w_n$ be a word. Suppose the set $S = \{w_1,w_2,\dots,w_n\}$ 
has $m$ distinct elements. 
If $\phi $ is the unique order-preserving bijection $S \to [m]$,
then
the \emph{flattened word} corresponding to $w$ is 
 $\st(w) = \phi(w_1)\phi(w_2)\cdots \phi(w_n)$.
 
A \emph{packed word} (also called a \emph{surjective word} \cite{Hazewinkel1}, \emph{Fubini word} \cite{Brends}, or \emph{initial word} \cite{PylPat}) is a word $w$ with $w = \st(w)$.
Define $\bfI_\P$ to be the subspace of $\bfW$ spanned by all differences $[v,m] - [w,n]$ 
where $[v,m],[w,n] \in \Words$ have $\st(v) = \st(w)$. The following is \cite[Proposition 3.7]{M1}:

\begin{proposition}\label{packed-prop}
The subspace $\bfI_\P$ is a homogeneous bi-ideal of $(\bfW, \nabla_\shuffle, \iota_\shuffle, \Delta_\odot, \epsilon_\odot)$.
The quotient bialgebra $\bfP = \bfW /\bfI_\P$ is a graded Hopf algebra.
\end{proposition}

The Hopf algebra $\bfP $ is 
the graded dual of the algebra of \emph{word quasi-symmetric functions} $\WQSym$  \cite{Maurice,NovelliThibon}.
Let $\Packed$ be the set of all packed words. 
If $[w,n] \in \Words$ and $w$ is a word with $m$ distinct letters then $v = \st(w)$
is the unique packed word such that $[w,n] + \bfI_\P = [v,m] + \bfI_\P$.
Identify $v \in \Packed$
with the coset $[v,m] + \bfI_\P$ so that we can view $\Packed$ as a basis for $\bfP$.
The unit element of $\bfP$ is then the empty packed word $\emptyset$, and the counit is the linear map $\epsilon_\odot : \bfP \to \kk$
with $\epsilon_\odot(\emptyset) = 1$ and $\epsilon_\odot(w) = 0$ for all  $\emptyset \neq w \in \Packed$.
%
For $u,v,w \in \Packed$ with $m = \max(u)$ and $n=\ell(w)$, 
\be\label{packed-prod-eq} \nabla_\shuffle(u \otimes v) = u\shuffle (v\uparrow m)
\qquand
\Delta_\odot(w) = \sum_{i=0}^n \st(w_1\cdots w_i) \otimes \st(w_{i+1}\cdots w_n).
\ee
The subspace of $\bfP$ spanned by the 
words in $\Packed$ that have no repeated letters
is a Hopf subalgebra. This is the well-known \emph{Malvenuto-Poirier-Reutenauer Hopf algebra} of permutations \cite{AguiarSottile,Mal2},
sometimes also called the Hopf algebra of \emph{free quasi-symmetric functions} $\FQSym$ \cite{NovelliThibon}.

\section{Linearly compact spaces}\label{cont-sect}

Let $U$ and $V$ be $\kk$-vector spaces.
Define $U^*$ to be the dual space of $U$, that is, the vector space of all $\kk$-linear maps $\lambda : U\to\kk$.
Given a linear map $\phi : U \to V$, define $\phi^*$
to be the linear map $V^* \to U^*$ with $\phi^*(\lambda) = \lambda \circ \phi$.
This makes $*$ into a contravariant functor $\kVec \to \kVec$.

We would like to be able to consider ``sub-bialgebras'' of $\bfW$ generated by certain infinite linear combinations of basis elements in $\Words$.
Such linear combinations are not well-defined in $\bfW$
but are naturally interpreted 
as elements of $\bfW^*$.
Therefore, we need a way of transferring the monoidal structures on the vector space $\bfW$ to its dual.

The full dual of an infinite-dimensional $\kk$-algebra is not naturally a $\kk$-coalgebra; see \cite[\S3.5]{Cartier}.
On the other hand, neither the standard form of graded duality nor the more general notion of restricted duality (see \cite[\S3.5]{Cartier})
 suffices for our application,
since $\bfW$ does not have finite graded dimension
and since the restricted dual will not permit infinite linear combinations.

The solution to these obstructions is to give the dual space a topology and consider monoidal structures
in the category of topological vector spaces rather than $\kVec$.
The topology in question is known as the \emph{linearly compact topology}, whose properties we quickly review.
Much of the background material in this section appears in \cite[Chapter 1]{Dieudonne}, so we omit some proofs.

A bilinear form $\langle\cdot,\cdot\rangle : U \times V \to \kk$ 
is \emph{nondegenerate} if $v \mapsto \langle \cdot,v\rangle$
is a bijection $V \to U^*$.
For example, the \emph{tautological form}
$\langle u, \lambda \rangle := \lambda(u)$ is a nondegenerate
bilinear form $U \times U^* \to \kk$.
The bilinear form $\langle a,b\rangle := ab$ is likewise a nondegenerate pairing $\kk \times \kk \to \kk$.

\begin{lemma}\label{sumprod-lem}
Suppose $\langle\cdot,\cdot\rangle : U \times V \to \kk$  is a nondegenerate bilinear form.
If there is a direct sum decomposition $U = \bigoplus_{i \in I} U_i$
then $V = \prod_{i \in I} V_i$ where $V_i = \{ v \in V : \langle u,v\rangle=0\text{ if }u \in U_j\text{ for } i\neq j\}$.
\end{lemma}

\begin{proof}
Identify  $\sum_{i \in I} v_i \in  \prod_{i \in I} V_i$ with the unique $v\in V$
satisfying $\langle u,v\rangle = \langle u,v_i\rangle$ for $i \in I$ and $u \in U_i$
to get an inclusion $\prod_{i \in I} V_i \hookrightarrow V$.
For $v \in V$, the linear map $U \to \kk$ with $u \mapsto \langle u,v\rangle$ for $u \in U_i$
and $u\mapsto 0$ for $u \in \bigoplus_{i\neq j} U_j$ has the form $u \mapsto \langle u,v_i\rangle$ for some $v_i \in V_i$,
and $v = \sum_{i \in I} v_i$. 
\end{proof}

Suppose $\langle\cdot,\cdot\rangle : U \times V \to \kk$ is a nondegenerate bilinear form
and $\{ u_i : i \in I\} $ is a basis for $U$.
For each $i \in I$, there exists a unique $v_i \in V$ with $\langle u_j, v_i \rangle = \delta_{ij}$ for all $j \in I$.
As $U = \bigoplus_{i \in I} \kk u_i$,
Lemma~\ref{sumprod-lem} implies that $V = \prod_{i \in I} \kk v_i$.
Thus each $v \in V$ can be uniquely expressed as the (potentially infinite) sum
$v=\sum_{i \in I} \langle u_i,v\rangle v_i $.
Following \cite{Dieudonne}, we call $\{v_i : i \in I\}$ a \emph{pseudobasis} for $V$;
this is sometimes also referred to as a \emph{continuous basis} (e.g., in \cite[\S3]{Pat}).

View each subspace $\kk v_i$ as a discrete topological space and 
give $V= \prod_{i \in I} \kk v_i$ the corresponding product topology;
this is the \emph{linearly compact topology} on $V$, also sometimes called the \emph{pseudocompact topology}.
This topology depends on the form $\langle\cdot,\cdot\rangle$ but not on the choice of basis for $U$.
Any finite intersection of 
 sets of the form $\left\{ \sum_{i \in I} c_i v_i  \in V : c_j \in C\right\}$ for fixed choices of
$C \subset  \kk$ and $j \in I$ is open in the linearly compact topology,
and every open subset of $V$ 
can be expressed as a union of these intersections. In other words, a basis for the linearly compact topology  consists of the sets \be
\label{basis-open}
\left\{ \sum_{i \in I} c_i v_i  \in V : c_i \in \kk\text{ for all $i \in I$ and }c_{i_1} \in C_1,\ c_{i_2} \in C_2,\ \dots,\ c_{i_p} \in C_p \right\}
\ee
for any finite list of indices $i_1,i_2,\dots,i_p \in I$ and any nonempty subsets $C_1,C_2,\dots,C_p \subset \kk$.
If $V$ is finite-dimensional,
then the linearly compact topology is discrete.

\begin{definition}
A \emph{linearly compact $\kk$-vector space} is a $\kk$-vector space $V$ equipped with
the linearly compact topology induced by a nondegenerate bilinear form $U \times V \to \kk$ for some $\kk$-vector space $U$.
Let $\whkVec$ denote the full subcategory of the category of topological $\kk$-vector spaces whose objects are linearly compact vector spaces.
\end{definition}

As noted in \cite{Dieudonne},
a topological vector space $V$ belongs to $\whkVec$ if and only if its topology 
is Hausdorff and linear (i.e., the open affine subspaces form a basis) and 
any family of closed affine subspaces with the finite intersection property has nonempty intersection.
The category  $\whkVec$ is closed under arbitrary direct products and finite direct sums, and contains 
the category of finite-dimensional vector spaces as a full subcategory.

A morphism between linearly compact vector spaces is a linear map that is continuous in the linearly compact topology.
We can be more explicit about which linear maps are continuous.
Suppose $V,W \in \whkVec$ have pseudobases $\{ v_i : i \in I\}$ and $\{w_j : j \in J\}$.
Let  $\psi : V \to W$ be a linear map and 
define $\psi_{ij} \in \kk$ 
to be the coefficient such that $\psi(v_i) = \sum_{j \in J} \psi_{ij} w_j$ for all $i \in I$.

\begin{lemma}\label{continuous-lem}
The map $\psi : V \to W$ is continuous in the linearly compact topology
if and only if
$\{ i \in I : \psi_{ij}\neq 0\}$ is finite for each $j \in J$
and 
$\psi\(\sum_{i \in I} c_i v_i \) = \sum_{j \in J} \(\sum_{i \in I} c_i \psi_{ij}\) w_j$  for any $c_i \in \kk$.
\end{lemma}

In other words, $\psi$ is continuous  when 
$\sum_{i \in I} c_i \psi(v_i)$ is always defined and equal to $\psi\(\sum_{i \in I} c_i v_i \)$.
It is an instructive exercise to work through the proof of this basic lemma.

\begin{proof}
If the given properties hold then the inverse image of $\left\{ \sum_{i \in J} c_i w_i  \in W : c_j \in C \right\}$
under $\psi$ is a union of finite intersections of analogous sets in $V$ and is therefore open.
It follows in this case that the inverse image of any open subset of $W$ under $\psi$ is open, so $\psi$ is continuous.

Conversely, assume $\psi$ is continuous. Let $j \in J$. We first check that $\{ i \in I : \psi_{ij}\neq 0\}$ is finite. 
Consider the open subset $S = \{ \sum_{k \in J} c_k w_k \in W : c_j = 0\}$.
The inverse image $\psi^{-1}(S)$ is open since $\psi$ is continuous and nonempty since $0 \in S$.
Therefore $\psi^{-1}(S)$ contains an open subset of the form \eqref{basis-open}.
Let $i_1,i_2,\dots,i_p \in I$ be the finite list of indices corresponding to this subset. Then for any 
$g \in \psi^{-1}(S)$ and $i \in I \setminus \{i_1,i_2,\dots,i_p\}$ we have $g+v_i \in \psi^{-1}(S)$,
so $\psi(g) \in S$ and $\psi(g+v_i) \in S$, whence by linearity $\psi(v_i)= \sum_{k \in J} \psi_{ik} w_k \in S$.
But this says precisely that if $i \in I \setminus \{i_1,i_2,\dots,i_p\}$ then $\psi_{ij} = 0$,
so $\{ i \in I : \psi_{ij}\neq 0\} $ is a subset of the finite set $ \{i_1,i_2,\dots,i_p\}$.

%
The map $\phi : V \to W$ defined by $\phi\(\sum_{i \in I} c_i v_i \) = \sum_{j \in J} \(\sum_{i \in I} c_i \psi_{ij}\) w_j$  
is thus well-defined and linear, and also continuous by the first paragraph of the proof.
Since $\psi - \phi$ is then linear and continuous, to deduce that $\psi = \phi$, it suffices to show that the only continuous linear map
$V \to W$ with $v_i \mapsto 0$ for all $i \in I$ is zero. This holds as the (open) inverse 
image of the open set $W -\{0\}$ under such a map 
does not contain any finite linear combination of pseudobasis elements $\{v_i : i \in I\}$, 
and therefore does not contain any set of the form \eqref{basis-open}, 
so must be empty.
\end{proof}

Suppose we have nondegenerate bilinear forms $\langle\cdot,\cdot\rangle_i : U_i \times V_i \to \kk$ for $i\in\{1,2\}$.
If $\phi : U_2 \to U_1$ is linear, then there exists a unique linear map
$\phi^\perp : V_1 \to V_2$ such that $\langle \phi(u_2),v_1\rangle_1 = \langle u_2,\phi^\perp(v_1)\rangle_2$ for all $u_2 \in U_2$ and $v_1\in V_1$.
If $V_i = U_i^*$ and $\langle\cdot,\cdot\rangle_i$ is the tautological form,
then $\phi^* = \phi^\perp$.

\begin{corollary}
In the preceding setup, a linear map $\psi : V_1 \to V_2$ is continuous in the linearly compact topology 
if and only if $\psi = \phi^\perp$
for some linear map $\phi : U_2 \to U_1$.
\end{corollary}

The set of continuous linear maps $V \to W$ between linearly compact vector spaces is therefore a $\kk$-vector space.
Let $V^\vee$ be the vector space of continuous linear maps $V\to\kk$ for $V \in \whkVec$.
This vector space is sometimes called the \emph{continuous dual} of $V$ (for example, in \cite[\S7.4]{LamPyl}). 

\begin{corollary}
Suppose $\langle\cdot,\cdot\rangle : U \times V \to\kk$ is a nondegenerate bilinear form.
If $\{u_i : i \in I\}$ is a basis for $U$,
then the functions $\langle u_i,\cdot \rangle : V \to \kk$ for $i \in I$ are a basis for $V^\vee$.
\end{corollary}

If $\psi : V \to W$ is a continuous linear map
then $\psi^* : W^* \to V^*$ restricts to a map $W^\vee \to V^\vee$, which we denote $\psi^\vee$.
The operation $\vee$ is then a contravariant functor $\whkVec\to \kVec$.
The preceding corollary implies that $U \in \kVec$ is naturally isomorphic to $(U^*)^\vee$ as a vector space
and that $V \in \whkVec$ is naturally isomorphic to $(V^\vee)^*$ as a topological vector space.
Thus, 
if $V \in \whkVec$ then the tautological pairing $V^\vee \times V \to \kk$
is nondegenerate and the linearly compact topology induced by this form recovers the topology on $V$.
We can summarize these observations as follows:

\begin{proposition} The functors $* : \kVec \to \whkVec$ and $\vee : \whkVec \to\kVec$ are dualities of categories.
\end{proposition}

Define the \emph{completion} of a $\kk$-vector space $U$ 
with respect to a given basis $\{ u_i : i \in I\}$ to be the vector space 
$\hat U = \prod_{i \in I} \kk u_i$ with the product topology, where each subspace $\kk u_i$ is  discrete.
In other words, $\hat U$ is the linearly compact $\kk$-vector space with $\{ u_i :i \in I\}$ as a pseudobasis.
Of course, if $U$ is finite-dimensional then $U = \hat U$.
The bilinear form $\langle\cdot,\cdot\rangle : U\times U \to \kk$ with $\langle u_i,u_j\rangle = \delta_{ij}$ 
 extends to a nondegenerate  bilinear form $U \times \hat U \to \kk$.
The space $\hat U$ is distinguished from $U^*$ in 
having a fixed inclusion $U \subset \hat U$.
Relative to this inclusion,  $U$ is a dense subset of $\hat U$, 
which explains why $\hat U$ is referred to as a completion.

The category $\whkVec$ has the following monoidal structure.
For objects $V,W,V',W' \in \whkVec$
and morphisms $\phi : V\to V'$
and $\psi : W \to W'$,
define \[V\htimes W = (V^\vee \otimes W^\vee)^*
\qquand
\phi \htimes \psi = (\phi^\vee \otimes \psi^\vee)^*.\]
The object $V \htimes W$ is a linearly compact vector space and the linear map $\phi \htimes \psi$ is continuous in the linearly compact topology.
There is a canonical inclusion 
$V \otimes W \hookrightarrow V \htimes W$ given by the linear map
identifying $v \otimes w$ for $v \in V$ and $w \in W$ with the linear function that has $\lambda \otimes \mu \mapsto \lambda(v)\mu(w)$ for $\lambda \in V^\vee$ and $\mu \in W^\vee$.
Relative to this inclusion, $V\otimes W$ is a dense subset of the linearly compact space $V\htimes W$,
and for this reason one calls $\htimes$ the \emph{completed tensor product}.
If $V$ and $W$ have pseudobases $\{v_i : i \in I\}$ and $\{w_j : j \in J\}$,
then the image of the set 
 $\{ v_i \otimes w_j  : (i,j) \in I\times J\} \subset V\otimes W$ in $V\htimes W$
is a pseudobasis.
We usually identify $V \otimes W$ with its image in $V\htimes W$ without comment.

%
%

Let $\beta $ be the isomorphism $V \otimes W \xrightarrow{\sim} W\otimes V$ induced by $x\otimes y \mapsto y \otimes x$.
This map uniquely extends to an isomorphism 
$\hat\beta : V\htimes W \to W\htimes V$ for all $V,W \in \whkVec$.
Recall that $\kk$ is a linearly compact vector space with the discrete topology. 

\begin{proposition}
The category $\whkVec$ is symmetric monoidal 
relative to the completed tensor product $\htimes$, braiding map $\hat\beta$, and unit object $\kk$.
\end{proposition}

\begin{proof}
Checking this proposition is a routine exercise from the axioms \cite[Chapter 1]{AguiarMahajan}.
One may simply transfer all arguments in the proof that $\kVec$ is symmetric monoidal 
to $\whkVec$ by duality.
\end{proof}

Since $\whkVec$ is symmetric monoidal, we have corresponding notions of 
(co, bi, Hopf) monoids in this category.
We refer to monoids, comonoids, bimonoids, and Hopf monoids in $\whkVec$ respectively as
\emph{linearly compact algebras}, \emph{coalgebras}, \emph{bialgebras}, and \emph{Hopf algebras}.
A structure of this type consists explicitly of a linearly compact vector space $V \in \whkVec$ along with continuous linear maps $V\htimes V \to V$,
$\kk \to V$, $V\to V\htimes V$, and $V\to \kk$ satisfying the conditions in Section~\ref{monoidal-sect}.

Alternatively, one can define linearly compact (co, bi, Hopf) algebras in $\whkVec$ entirely in terms of 
(co, bi, Hopf) algebras by duality.
Let $U \in \kVec$ and $V \in \whkVec$ and let $\langle\cdot,\cdot\rangle : U \times V \to \kk$ be a nondegenerate bilinear form.
Define
$\langle u_1\otimes u_2, v_1\otimes v_2\rangle = \langle u_1,v_1\rangle\langle u_2,v_2\rangle$
for $u_i \in U$ and $v_i \in V$
and extend by continuity and linearity to define a nondegenerate bilinear form $(U\otimes U)\times (V\htimes V) \to \kk$
that is continuous in the second coordinate.
Also let $\langle a,b\rangle=ab$ for $a,b \in \kk$.

Now suppose 
$\nabla : U \otimes U \to U$, $\iota : \kk \to U$, $\Delta : U \to U \otimes U$, and $\epsilon : U\to\kk$
are linear maps 
and
$\hat \nabla : V \htimes V \to V$, $\hat \iota : \kk \to V$, $\hat \Delta : V \to V \htimes V$, and $\hat\epsilon : V\to\kk$
are continuous linear maps
such that 
\[
\langle \nabla(u_1\otimes u_2), v\rangle = \langle u_1\otimes u_2,\hat\Delta(v)\rangle
\qquand
\langle \iota(a),v\rangle = \langle a, \hat\epsilon(v)\rangle\]
for all $u_1,u_2 \in U$, $v \in V$, and $a \in \kk$
and
\[
\langle \Delta(u), v_1\otimes v_2\rangle = \langle u, \hat\nabla(v_1\otimes v_2)\rangle
\qquand
\langle \epsilon(u),b\rangle = \langle u, \hat\iota(b)\rangle
\]
for all $u \in U$, $v_1,v_2 \in V$, and $b \in \kk$.
Either map in each of the pairs $(\nabla,\hat \Delta)$, $(\iota,\hat \epsilon)$, $(\Delta,\hat \nabla)$ and $(\epsilon,\hat\iota)$
then
uniquely determines the other.

In this setup, $(U,\nabla,\iota)$ is an algebra if and only if $(V,\hat\Delta,\hat\epsilon)$ is a linearly compact  coalgebra;
 $(U,\Delta,\epsilon)$ is a coalgebra if and only if $(V,\hat\nabla,\hat\iota)$ is a linearly compact  algebra;
and $(U,\nabla,\iota,\Delta,\epsilon)$ is a bialgebra (respectively, Hopf algebra) if and only if $(V,\hat\nabla,\hat\iota,\hat\Delta,\hat\epsilon)$
is a linearly compact  bialgebra (respectively, Hopf algebra). 
In these cases, we say that the monoidal structure on $V$ is the \emph{(algebraic) dual} of the structure on $U$ via the form $\langle\cdot,\cdot\rangle$.

This perspective indicates how to give a linearly compact (co, bi, Hopf) algebra structure to the completed tensor product or direct sum of two
linearly compact (co, bi, Hopf) algebras.
For example, suppose $U_1$ and $U_2$ are algebras and $V_i$ is the linearly compact coalgebra dual to $U_i$.
Then $U_1 \otimes U_2$ and $U_1 \oplus U_2$ are both naturally algebras, and 
we can identify $V_1 \htimes V_2$ with the dual of $U_1 \otimes U_2$
and $V_1\oplus V_2$ with the dual of $U_1 \oplus U_2$ in order to interpret both objects
as linearly compact coalgebras.
A similar statement holds if we assume each $U_i$ is a coalgebra, bialgebra, or Hopf algebra so 
that each $V_i$ is a linearly compact algebra, bialgebra, or Hopf algebra, respectively.

\begin{example}\label{kk[x]-ex}
Let $\kk[x] = \bigoplus_{n \in \NN} \kk x^n$ and $\kk[[x]] = \prod_{n \in \NN} \kk x^n$ denote the $\kk$-algebras of polynomials and formal power series in $x$.
The bilinear form $ \kk[x] \times \kk[[x]] \to \kk$
with $\langle x^m, \sum_{n \in \NN} c_n x^n \rangle = c_m$ is nondegenerate,
and restricts to a nondegenerate form $\kk[x] \times \kk[x] \to \kk$. 

The space $\kk[x]$ is a graded Hopf algebra whose coproduct, counit, and antipode are the algebra morphisms with $\Delta(x) = 1 \otimes x + x\otimes 1$,
$\epsilon(x) = 0$, and  $\antipode(x) = -x$.
The space $\kk[[x]]$ is a linearly compact Hopf algebra whose coproduct, counit, and antipode are the linearly compact algebra morphisms with the same formulas.

The Hopf algebra $\kk[x]$ is its own graded dual via the form $\langle\cdot,\cdot\rangle$, but $\kk[[x]]$ is its algebraic dual.
The completed tensor product $\kk[[x]]\htimes \kk[[x]]$ is isomorphic to the vector space of formal power series $\kk[[x,y]]$ in two commuting variables.
\end{example}

\begin{example}\label{ext-ex}
Any graded (co, bi, Hopf) algebra of \emph{finite graded dimension} (that is, whose homogeneous components are each finite-dimensional)
extends to a linearly compact  (co, bi, Hopf) algebra.
In detail, suppose $V = \bigoplus_{n\in \NN} V_n$ is a graded $\kk$-vector space where each $V_n$ is finite-dimensional.
Let $\hat V = \prod_{n \in \NN} V_n$ 
and give this space the product topology in which each subspace $V_n$ is discrete.
Then $\hat V$ is  a linearly compact vector space
and any graded linear map $V \otimes V \to V$ or $ \kk \to V$ or $ V \to V \otimes V$ or $ V \to \kk$
extends uniquely to a continuous linear map 
$\hat V \htimes \hat V \to \hat V$ or $\kk \to \hat V$ or $\hat V \to \hat V \htimes \hat V$ or $\hat V \to \kk$, respectively.
If $V$ has the structure of a graded (bi, co, Hopf) algebra, then these extensions make $\hat V$ into
a linearly compact (bi, co, Hopf) algebra; the relevant structure on $\hat V$ is isomorphic to the algebraic dual of the graded dual of $V$.
\end{example}

\begin{remark}
Linearly compact (bi, co, Hopf) algebras have appeared in a few places previously in the literature, usually without being so named.
For example, the ``bialgebras'' $\hat \Gamma$ and $\hat \Lambda$ in \cite[\S9]{Buch} are linearly compact bialgebras.
Likewise, the ``Hopf algebras'' $\mathfrak{m}\mathrm{Sym}$, $\mathfrak{m}\mathrm{QSym}$, and $\mathfrak{m}\mathrm{MR}$
introduced in \cite{LamPyl} and further studied in \cite{Pat} are all linearly compact Hopf algebras.
\end{remark}


Recall that $\Words$ is the set of pairs $[w,n]$ where $n \in \NN$ and $w$ is a word with letters in $\{1,2,\dots,n\}$, and
 $\bfW = \kk\Words$.
Define $\hat\bfW$ to be 
the completion of $\bfW$ with respect to the basis $\Words$.
For $\sigma \in \hat\bfW$ and $[w,n] \in \Words$,
let $\sigma(w,n)\in \kk$ denote the coefficient  such that $\sigma = \sum_{[w,n] \in \Words} \sigma(w,n) [w,n]$.
The associated nondegenerate bilinear form  $\langle\cdot,\cdot\rangle : \bfW \times \hat\bfW \to \kk$
is then 
\be
\label{line-form} \langle \sigma,\tau \rangle = \sum_{[w,n] \in \Words} \sigma(w,n)\tau(w,n) 
\qquad\text{for $\sigma \in \bfW$ and $\tau \in \hat\bfW$.}
\ee 
Define $\nabla_\odot : \hat\bfW \htimes \hat\bfW \to \hat\bfW$ 
to be the continuous linear map
with 
\be\label{nabla-odot-eq}
\nabla_\odot([v,m] \otimes [w,n]) = \begin{cases}  [vw, m] &\text{if }m=n \\ 0&\text{otherwise}\end{cases}
\ee
for $[v,m],[w,n] \in\Words$.
Define
$\Delta_\shuffle : \hat\bfW \to \hat\bfW\htimes \hat\bfW$
to be the continuous linear map
with
\be\label{delta-shuffle-eq}
\Delta_\shuffle([w,n]) = \sum_{m=0}^{n} \left[w \cap \{1,2,\dots,m\},m\right] \otimes \left[(w\downarrow m) \cap \{1,2,\dots,n-m\} ,n-m\right]
\ee
where if $p=\ell(w)$ then
$w\downarrow m = (w_1-m)(w_2-m)\dots (w_{p}-m)$
and where 
 $w \cap S$ denotes the subword of $w$ formed by omitting all letters not in $S$.
Define $\epsilon_\shuffle : \hat\bfW \to \kk$ and $\iota_\odot : \kk \to \hat\bfW$ to be the continous linear maps with 
\be
\label{co-units-eq}
 \epsilon_\shuffle([w,n]) = \begin{cases} 1 &\text{if }n = 0
\\
0&\text{otherwise}
\end{cases}
\qquand 
\iota_\odot(1) = \sum_{n \in \NN} [\emptyset,n].
\ee

\begin{theorem}\label{w-cont-thm}
$(\hat\bfW, \nabla_\odot, \iota_\odot, \Delta_\shuffle,\epsilon_\shuffle)$
is a linearly compact bialgebra.
\end{theorem}

\begin{proof}
It is  a straightforward exercise
to check that $(\hat\bfW, \nabla_\odot, \iota_\odot, \Delta_\shuffle,\epsilon_\shuffle)$ is the algebraic dual
of the bialgebra $(\bfW,\nabla_\shuffle,\iota_\shuffle,\Delta_\odot,\epsilon_\odot)$ via the bilinear form \eqref{line-form}.
\end{proof}

%
%
%

Define $\hP$ to be the completion of the vector space of packed words $\bfP$ with respect to the basis $\Packed$.
The natural pairing $\bfP \times \hP \to \kk$ gives $\hP$
the structure of a linearly compact  Hopf algebra dual to $(\bfP,\nabla_\shuffle,\iota_\shuffle,\Delta_\odot,\epsilon_\odot)$,
which one can realize as a sub-bialgebra of $(\hat\bfW, \nabla_\odot, \iota_\odot, \Delta_\shuffle,\epsilon_\shuffle)$.
This object is not of much relevance to our discussion, however.

On the other hand, 
since $\bfP$ has finite graded dimension when graded by word length,
 the maps $\nabla_\shuffle$, $\iota_\shuffle$, $\Delta_\odot$ and $\epsilon_\odot$ from \eqref{packed-prod-eq}
have continuous linear extensions to maps between $\hP$, $\hP\htimes\hP$, and $\kk$ as appropriate, 
and the following holds in view of Example~\ref{ext-ex}:

\begin{proposition}\label{packed-hat-prop}
$(\hP,\nabla_\shuffle,\iota_\shuffle,\Delta_\odot,\epsilon_\odot)$ is a linearly compact Hopf algebra.
\end{proposition}

Let $\hat\bfW_n$ be the completion of $\bfW_n$ with respect to $\Words_n$.
Each subspace $\bfW_n$ for $n \in \NN$ is a sub-coalgebra of $(\bfW,\Delta_\odot, \epsilon_\odot)$ of finite graded dimension,
so  $\Delta_{\odot}$ and $\epsilon_\odot$ extend
to continuous linear maps $\hat\bfW_n \to \hat\bfW_n \htimes \hat\bfW_n$ and $\hat\bfW_n \to \kk$,
and the following similarly holds:

\begin{proposition}\label{wn-prop}
For each $n \in \NN$, $(\hat\bfW_n,\Delta_\odot,\epsilon_\odot)$ is a linearly compact  coalgebra.
\end{proposition}

Since $\bfW$ does not have finite graded dimension,
the bialgebra structure $(\bfW, \nabla_\shuffle, \iota_\shuffle, \Delta_\odot, \epsilon_\odot)$
does not extend to $\hat\bfW$.
In particular, the counit $\epsilon_\odot$ 
cannot be evaluated
at $\sum_{n \in \NN} [\emptyset,n] \in \hat\bfW$.
Nevertheless, there is a sense in which $\nabla_\shuffle$ and $\Delta_\odot$ can be interpreted as compatible morphisms
$\hat\bfW \htimes \hat\bfW \to \hat\bfW$  and $\hat\bfW \to \hat\bfW \htimes\hat\bfW$. This is the main theme of the next section.

\section{Species coalgebroids}\label{species-sect}

Let $\Mon(\sC)$, $\Comon(\sC)$, and $\Bimon(\sC)$ be the categories of monoids, comonoids, and bimonoids in a symmetric monoidal category $\sC$.
Let $\FB$ denote the category of finite sets with bijections as morphisms.
A \emph{$\sC$-species} is a functor $ \FB \to \sC$.
Such functors form a category, denoted $\sC$-$\Sp$, with natural transformations as morphisms.
For more background on species, see \cite[Chapter 8]{AguiarMahajan}.

When $\bF$ is a $\sC$-species and $S$ is a finite set and $\sigma : S\to T$ is a bijection,
we write $\bF[S]$ for the corresponding object in $\sC$ and $\bF[\sigma]$ for the corresponding morphism $\bF[S] \to \bF[T]$, which is necessarily an isomorphism.
When $\eta  : \bF \to \bG$ is a natural transformation and $S$ is a finite set, we write $\eta_S$ for the corresponding morphism $\bF[S] \to \bG[S]$.
We refer to $\bF[S]$ and $\eta_S$ as the \emph{$S$-component} of $\bF$ and $\eta$.
If $S$ is clear from context and $x \in \bF[S]$, then we may write $\eta(x)$ instead of $\eta_S(x) $ for the corresponding element of $\bG[S]$.
A \emph{subspecies} of a $\sC$-species $\bG$ is a $\sC$-species $\bF$ with $\bF[S] \subset \bG[S]$ for all finite sets $S$
and $\bF[\sigma] = \bG[\sigma]|_{\bF[S]}$ for all bijections $\sigma : S \to T$.
We write $\bF \subset \bG$ to indicate that $\bF$ is a subspecies of $\bG$.

With these conventions,
a \emph{linearly compact coalgebra species} is a functor $\cV : \FB \to \Comon(\whkVec)$.
Suppose $\cU$, $\cU'$, $\cV$, and $\cV'$ are linearly compact coalgebra species and $\alpha : \cU \to \cU'$ and $\beta : \cV \to \cV'$ are natural transformations.
Define $\cU \cdot \cV : \FB \to \Comon(\whkVec)$ and $\alpha\cdot \beta : \cU\cdot \cV \to \cU'\cdot \cV'$ by
\be\label{cauchy-eq} (\cU \cdot \cV)[I] = \bigoplus_{S\sqcup T = I} \cU[S]\htimes \cV[T]
\qquand
 (\alpha \cdot \beta)_I = \bigoplus_{S\sqcup T=I} \alpha_S\htimes \beta_T
\ee
for each finite set $I$, where the sums are over all $2^{|I|}$ ways of writing $I$ as a union of two disjoint sets.
Define $(\cU \cdot \cV)[\sigma] : (\cU\cdot \cV)[I] \to (\cU\cdot \cV)[J]$ similarly when $\sigma :I \to J$ is a bijection. 
The category of linearly compact coalgebra species
 is symmetric monoidal 
with respect to this operation, called the \emph{Cauchy product} in \cite{AguiarMahajan},
with unit object given by the species $\One : \NN \to \Comon(\whkVec)$ that has $\One[\varnothing] = \kk$ and $\One[S] = 0$ for all nonempty finite sets $S$.
When $\nabla :   \cV \cdot \cV \to \cV$ is a natural transformation and $I =S\sqcup T$,
  we write $\nabla_{ST} :   \cV[S] \htimes \cV[T] \to \cV[I]$ for the composition of $\nabla_{I} :  (\cV \cdot\cV)[I]  \to \cV[I]$ with the inclusion
 $\cV[S] \htimes \cV[T] \to (\cV \cdot\cV)[I] $.

\begin{definition}\label{extended-def}
A \emph{species coalgebroid} is a monoid in the category of linearly compact coalgebra species.
Explicitly, suppose $\cV : \FB \to \Comon(\whkVec)$ is a functor. Write $\Delta_I$ and $\epsilon_I$ for the coproduct and counit of $\cV[I]$ and
let $\Delta =(\Delta_I)$ and $\epsilon = (\epsilon_I)$ denote the corresponding families of linear maps.
Suppose $\nabla :  \cV \cdot \cV \to \cV$ and $\iota : \One \to \cV$ are natural transformations.
Then $(\cV,\nabla,\iota,\Delta,\epsilon)$ is a species coalgebroid if and only if the following conditions hold:
\ben
\item[(a)] For all pairwise disjoint finite sets  $S$, $T$, and $U$, the following diagrams commute:
 \[
{\footnotesize
\begin{diagram}[small]
 \cV[S]\htimes \cV[T] \htimes \cV[U] & \rTo^{\ \ \nabla_{ST}\htimes \id\ \ }& \cV[S\sqcup T] \htimes \cV[U]\\
\dTo^{\id \htimes \nabla_{TU}}  && \dTo_{\nabla_{S\sqcup T,U}} \\
\cV[S] \htimes \cV[T\sqcup U] & \rTo^{\nabla_{S,T\sqcup U}} & \cV[S \sqcup T \sqcup U]
\end{diagram}
}
\]
\smallskip
\[
{\footnotesize
\begin{diagram}[small]
\kk\htimes \cV[S] & \rTo^{\ \ \iota_\varnothing\htimes\id\ \ } & \cV[\varnothing]\htimes \cV[S]  \\
 & \rdTo^{\cong} & \dTo^{\nabla_{\varnothing S}}  \\
 & & \cV[S]
 \end{diagram}
 \qquad\qquad
 \begin{diagram}[small]
 \cV[S]\htimes \cV[\varnothing] & \lTo^{\ \ \id\htimes \iota_\varnothing\ \ } & \cV[S] \htimes \kk
  \\ \dTo^{\nabla_{S\varnothing}} &  \ldTo^{\cong}\\
 \cV[S]
 \end{diagram}
}
\]

\item[(b)] For all disjoint finite sets $S$ and $T$, the following diagrams commute:
\[
{\footnotesize
\begin{diagram}[small]
\kk &   &  \rTo^{\id} & &  \kk \\
&  \rdTo^{\iota_\varnothing} & & \ruTo^{\epsilon_\varnothing} \\
 & & \cV[\varnothing]
\end{diagram}
\qquad\qquad
\begin{diagram}[small]
\kk & \rTo^{\iota_\varnothing} & \cV[\varnothing]
\\ 
\dTo^{\cong} && \dTo_{\Delta_{\varnothing}} \\
\kk\htimes \kk & \rTo^{\ \ \iota_\varnothing \htimes \iota_\varnothing\ \ } & \cV[\varnothing] \htimes \cV[\varnothing]
\end{diagram}
\qquad\qquad
\begin{diagram}[small]
\cV[S] \htimes \cV[T]& \rTo^{\ \ \epsilon_S\htimes \epsilon_T\ \ } & \kk \htimes \kk
\\ 
\dTo^{\nabla_{ST}} && \dTo_{\cong} \\
\cV[S\sqcup T] & \rTo^{\epsilon_{S\sqcup T}} & \kk
\end{diagram}
}
\]

\item[(c)] For all disjoint finite sets $S$ and $T$, the following diagram commutes:
\[
{\footnotesize
 \begin{diagram}[small]
 \cV[S]\htimes \cV[T]&\rTo^{\nabla_{ST}\ \ \ \ } & \cV[S\sqcup T] & \rTo^{\ \ \ \ \Delta_{S\sqcup T}} & \cV[S\sqcup T] \htimes \cV[S\sqcup T] \\
\dTo^{\Delta_{S} \htimes \Delta_{T}}  & &&& \uTo_{\nabla_{ST} \htimes \nabla_{ST}} \\
\cV[S] \htimes  \cV[S] \htimes \cV[T] \htimes \cV[T]  && \rTo^{\id\htimes \hat\beta \htimes \id} &&
 \cV[S] \htimes \cV[T] \htimes \cV[S]  \htimes \cV[T]
\end{diagram}
}
\]
\een
We refer to $\nabla : \cV \cdot \cV \to \cV$ and $\iota : \One \to \cV$ as the \emph{product} and \emph{unit} of $\cV$,
and to the families of maps $\Delta =(\Delta_I)$ and $\epsilon = (\epsilon_I)$ as the \emph{coproduct} and \emph{counit} of $\cV$.
Species coalgebroids form a category, which we denote by $\ExtBim$, whose morphisms are the natural transformations 
between $\Comon(\whkVec)$-species that commute with the product and unit morphisms.
\end{definition}

If
 $(\cV,\nabla,\iota,\Delta,\epsilon) \in \ExtBim$ is a species coalgebroid,
then a subspecies $\cH\subset \cV$ is a \emph{sub-coalgebroid}
when  
$\Delta_S(\cH[S]) \subset \cH[S] \htimes \cH[S]$  for each finite set $S$
and the morphisms
 $\nabla$ and $\iota$ restrict to 
natural transformations $\cH \cdot \cH \to \cH$ and $\One \to \cH$.
When these conditions hold,
we have
$(\cH,\nabla,\iota,\Delta,\epsilon) \in \ExtBim$.

\begin{remark}
If needed, one can introduce a sequence of definitions dual to those above.
The natural dual of a linearly compact coalgebra species is an \emph{algebra species}, i.e., a functor $ \FB \to \Mon(\kVec)$.
Such functors form a symmetric monoidal category with unit object $\One$, relative to the Cauchy product defined just as in \eqref{cauchy-eq} 
but with the
completed tensor product $\htimes$ replaced by
$\otimes$.
The natural dual of a species coalgebroid is then a comonoid in the category of algebra species.
\end{remark}

Species coalgebroids generalize linearly compact bialgebras
since the latter are monoids in the category of linearly compact coalgebras.
We highlight three functors to or from $\ExtBim$:
\ben
\item[(i)] There is  a natural ``forgetful'' functor
\be
\F : \ExtBim \to \Bimon(\whkVec)
\ee
with
$\F(\B) = (\cV[\varnothing], \nabla_{\varnothing,\varnothing},\iota_\varnothing, \Delta_\varnothing, \epsilon_\varnothing)$
 for each $\B = (\cV,\nabla,\iota,\Delta,\epsilon) \in \ExtBim$
 and with  $\F(\eta) = \eta_\varnothing$
for each morphism $\eta : \B \to \B'$ in $\ExtBim$.

\item[(ii)] For $V \in \kVec$, let $\bfE(V) : \FB \to \kVec$ be the species with
$\bfE(V)[S] = V$
and 
$\bfE(V)[\sigma] = \id_V$ for all  finite sets $S$ and bijections $\sigma : S \to T$.
For any linear map $\phi : V\to V'$,
let $\bfE(\phi) : \bfE(V) \to \bfE(V')$
be the natural transformation with $\bfE(\phi)_S = \phi$ for all finite sets $S$.
This gives a functor
\be
\label{nn-eq}
\bfE : \kVec \to \kVec\text{-}\Sp.
\ee
If
$B = (V,\mu,i,\delta,e) $ is a linearly compact bialgebra, then 
define 
$
\bfE(B) :=  (\bfE(V),\nabla,\iota,\Delta,\epsilon)
$
to be the species coalgebroid
in which $\nabla_{ST} = \mu$ , $\iota_I = i$, $\Delta_I = \delta$, and $\epsilon_I = e$ for all disjoint finite sets $S$, $T$, and $I$.
This makes $\bfE$ into a functor $\Bimon(\whkVec) \to \ExtBim$.

\item[(iii)] Suppose $\B = (\cV,\nabla,\iota,\Delta,\epsilon) \in \ExtBim$ is \emph{finite-dimensional} 
in that
 $\dim_\kk \cV[S] < \infty$ for all finite sets $S$.
For each $n \in \NN$, the symmetric group $S_n$ 
acts as a group of coalgebra automorphisms on $\cV[n] := \cV[\{1,2,\dots,n\}]$
via $x \mapsto \cV[\sigma](x)$ for $x \in \cV[n]$ and $\sigma \in S_n$.
The subspace $\bfI_n\subset \cV[n]$ spanned by all differences $x - \cV[\sigma](x)$
for $x \in \cV[n]$ and $\sigma \in S_n$
is a coideal
and we denote the corresponding quotient coalgebra by  $\cV[n]_{S_n} = \cV[n] / \bfI_n.$
Reuse $\Delta_n$ and $\epsilon_n$ to denote the coproduct and counit of $\cV[n]_{S_n}$.
Consider the compositition
\[   ({\cV \cdot \cV})[n]  =  \bigoplus_{S\sqcup T=[n]} \cV[S]\otimes \cV[T]  \to \bigoplus_{i+j=n} \cV[i]\otimes \cV[j] \to   \bigoplus_{i+j=n} \cV[i]_{S_i} \otimes \cV[j]_{S_j}\]
where the second arrow is the natural quotient map and the first arrow is the direct sum  $ \bigoplus_{S\sqcup T=[n]}\cV[\sigma_S] \otimes \cV[\sigma_T]$
with $\sigma_S$ denoting the order-preserving bijection $S \to [|S|]$ and $\sigma_T$ defined likewise.
As explained in \cite[\S15.1.1]{AguiarMahajan} (see in particular the proof of \cite[Proposition 15.2]{AguiarMahajan}),
this map
descends to an isomorphism   
\[(\cV\cdot\cV)[n]_{S_n} \xrightarrow{\sim} \bigoplus_{i+j=n} \cV[i]_{S_i} \otimes \cV[j]_{S_j}.\]
The $[n]$-component of $\nabla$ descends to a linear map 
\[\nabla_n : (\cV \cdot \cV)[n]_{S_n} \to \cV[n]_{S_n}.\]
The space
$V=\bigoplus_{n \in \NN} \cV[n]_{S_n}$ is a $\kk$-bialgebra with 
product
\[
V \otimes V = \bigoplus_{n \in \NN}  \bigoplus_{i+j=n} \cV[i]_{S_i} \otimes \cV[j]_{S_j} \xrightarrow{\sim} \bigoplus_{n \in\NN} (\cV \cdot \cV)[n]_{S_n} \xrightarrow{\bigoplus_{n \in \NN} \nabla_n} V,\]
and  coproduct
\[
V \xrightarrow{\bigoplus_{n \in \NN} \Delta_n} \bigoplus_{n \in \NN} (\cV[n]_{S_n} \otimes \cV[n]_{S_n}) \hookrightarrow V\otimes V,
\]
along with unit  $ \bigoplus_{n \in \NN} \iota_{[n]} = \iota_\varnothing$ and counit $\bigoplus_{n \in \NN} \epsilon_n$.
Let $\Fock(\B)$ denote this bialgebra.
When
$\eta : \B \to \B'$ is a morphism between finite-dimensional coalgebroids,
the direct sum $ \bigoplus_{n \in \NN} \eta_{[n]}$ descends to a map
$ \Fock(\B) \to \Fock(\B')$, denoted $\Fock(\eta)$.
This makes $\Fock$ into a functor
\be\label{sigma-eq} \Fock : \FDExtBim \to \Bimon(\kVec)\ee
where $\FDExtBim$ is the full subcategory of finite-dimensional species coalgebroids. 
The functor $\Fock$ is similar to the \emph{bosonic Fock functor} defined 
in \cite[Chapter 15]{AguiarMahajan}.
\een

We conclude this section by constructing what will be our fundamental example of Definition~\ref{extended-def}.
Fix a set $S$ of size $n$.
For each bijection $\lambda: [n] \to S$, let
$\Words_\lambda$ be the set of pairs $[w,\lambda]$ where $w$ is a word with $\max(w) \leq n$.
Define $\hat\bfW_\lambda$ to be the linearly compact $\kk$-vector space with $\Words_\lambda$ as a pseudobasis.
Write $\L[S]$ for the set of bijections $[n] \to S$ and let
\[\sW[S] = \bigoplus_{\lambda \in \L[S] }\hat\bfW_\lambda \in \whkVec.
\]
For each bijection $\sigma : S \to T$,
define $\sW[\sigma]$ to be the continuous linear map $\sW[S] \to \sW[T]$ with 
\be\label{spsigma-eq}
\sW[\sigma]([w,\lambda]) = [w, \sigma \circ \lambda]\qquad\text{for }[w,\lambda] \in \Words_\lambda.
\ee
These definitions make 
$ \sW $
into a functor $\FB \to \whkVec$.

Identify $[w,n] \in \Words_n$ with the element  $[w,\lambda] \in \Words_{\lambda}$ 
where $\lambda$ is the identity map $[n] \to [n]$
and in this way view $\hat\bfW_n$ as a subspace of $\sW[n] := \sW[\{1,2,\dots,n\}]$.
We extend 
$\Delta_{\odot} : \hat\bfW_n \to \hat\bfW_n \htimes \hat\bfW_n$ and $\epsilon_\odot : \hat\bfW_n \to \kk$ from \eqref{odot-maps}
to continuous linear maps
$\Delta_\odot : \sW[S] \to \sW[S]\htimes \sW[S]
$
and
$
\epsilon_\odot: \sW[S] \to \kk$ by requiring that
for each subspace $\hat\bfW_\lambda \subset \sW[S]$ we have
\be\label{spdelta-eq}
\Delta_\odot|_{\hat\bfW_\lambda} = 
( \sW[\lambda]\htimes \sW[\lambda]) \circ \Delta_\odot \circ \sW[\lambda^{-1}]|_{\hat\bfW_\lambda}
\quand
\epsilon_{\odot}|_{\hat\bfW_\lambda} = \epsilon_\odot
\circ
\sW[\lambda^{-1}]|_{\hat\bfW_\lambda}.
\ee
This means that if $[w,\lambda] \in  \Words_{\lambda}$ where $w=w_1w_2\cdots w_m$ has $m$ letters, then
\[ \Delta_{\odot} ( [w,\lambda]) = \sum_{i=0}^m [w_1\cdots w_i,\lambda] \otimes  [w_{i+1}\cdots w_m,\lambda]
\quand \epsilon_{\odot} ( [w,\lambda]) = \begin{cases} 1 & \text{if $m=0$} \\ 
0&\text{if $m>0$}.
\end{cases}\]
By Proposition~\ref{wn-prop}, $\sW$ defines a linearly compact coalgebra species $\FB \to \Comon(\whkVec)$.

Given 
disjoint finite sets $S$ and $T$ with $n=|S|$ and $m=|T|$ and 
bijections $(\lambda,\mu) \in \L[S] \times \L[T] $,
let $\lambda \oplus \mu$ denote the bijection $  [n+m] \to S\sqcup T$ with $i \mapsto \lambda(i)$ for $i \in [n]$ and $n+j \mapsto  \mu(j)$ for $j \in [m]$.
Write $\nabla_\shuffle : \sW \cdot \sW \to \sW$
for the natural transformation whose $I$-component $(\sW\cdot \sW)[I] \to \sW[I]$ is the direct sum, over all disjoint decompositions $I = S\sqcup T$ and $(\lambda,\mu) \in \L[S] \times \L[T] $,
of the maps
\be\label{spnabla-eq} \sW[\lambda \oplus \mu] \circ \nabla_\shuffle \circ \(\sW[\lambda^{-1}] \htimes \sW[\mu^{-1}]\) : \hat\bfW_\lambda \htimes \hat\bfW_\mu \to \hat\bfW_{\lambda\oplus \mu} \ee
with $\nabla_\shuffle : \hat\bfW_n \htimes \hat\bfW_m \to \hat\bfW_{n+m}$ as in \eqref{nabla-eq}.
This means that if $[v,\lambda] \in  \Words_{\lambda}$ and $[w,\mu] \in  \Words_{\mu}$ then
\[ \nabla_\shuffle([v,\lambda],[w,\mu]) = [v \shuffle (w\uparrow m), \lambda\oplus \mu]\]
where $m$ is the size of the domain of $\lambda$.
Finally, let $\iota_\shuffle : \One \to \sW$ be the natural transformation whose nontrivial component is the linear map $\One[\varnothing]=\kk \to \sW[\varnothing]$
with $1 \mapsto  [\emptyset,\id_\varnothing]$.

\begin{remark}
We can describe the maps  \eqref{spdelta-eq} and \eqref{spnabla-eq} more concretely.
Let $S$ be a finite set of size $n$.
An \emph{$S$-word} is a finite sequence $a=a_1a_2\cdots a_l$ with $a_i \in S$.
Given a bijection $\lambda : [n] \to S$,
define $(a,\lambda) = [w,\lambda] \in \Words_\lambda$ where $w=w_1w_2\cdots w_l$ is the word with 
$\lambda(w) := \lambda(w_1)\lambda(w_2)\cdots \lambda(w_l) =a$. 
Equation \eqref{spsigma-eq} is then
$ \sW[\sigma]((a,\lambda)) = (\sigma(a), \sigma\circ \lambda)
$
and the formulas in \eqref{spdelta-eq} become
\[
\Delta_\odot((a,\lambda)) = \sum_{i=0}^l (a_1\cdots a_i,\lambda) \otimes (a_{i+1}\dots a_l,\lambda)
\quand
\epsilon_{\odot}((a,\lambda)) = \begin{cases} 1 & \text{if } a = \emptyset \\ 0&\text{otherwise.}\end{cases}
\]
If
$b$ is a $T$-word where $S\cap T = \varnothing$ and $\mu : [m] \to T$ is a bijection,
so that $(b,\mu) \in \Words_\mu$,
then \eqref{spnabla-eq} is the continuous linear map with
$
\nabla_\shuffle((a,\lambda)\otimes (b,\mu)) = (a \shuffle b, \lambda \oplus \mu)$
where we define $( c_1 w^1 + \dots + c_k w^k, \lambda) = c_1 (w^1,\lambda) + \dots + c_k (w^k,\lambda)$.
In this way, the product can be defined using the ordinary shuffle operation instead of the shifted shuffle in \eqref{nabla-eq}.
\end{remark}

With slight abuse of notation,
we reuse the symbols $\Delta_\odot$ and $\epsilon_\odot$ to denote the families of maps
$\sW[S] \xrightarrow{\Delta_\odot} \sW[S] \htimes \sW[S]$ and $\sW[S] \xrightarrow{\epsilon_\odot} \kk$
for all finite sets $S$.
The following then holds:

\begin{theorem}\label{co-ext-bia-prop}
$(\sW, \nabla_\shuffle,\iota_\shuffle,\Delta_\odot,\epsilon_\odot)$
is a species coalgebroid.
\end{theorem}

\begin{proof}
Modify the diagrams in Definition~\ref{extended-def} by replacing $\cV[\emptyset]$, $\cV[S]$, $\cV[T]$, and $\cV[U]$ by $\hat\bfW_0$, $\hat\bfW_{|S|}$, $\hat\bfW_{|T|}$, and $\hat\bfW_{|U|}$.
It suffices to show that these modified diagrams each commute.
Since all arrows in the diagrams are continuous linear maps, this follows by Theorem~\ref{w-thm}.
\end{proof}

\section{Word relations}\label{wr-sect}

Here, we characterize the relations on words 
that generate sub-objects of the bialgebra $\bfW$, the linearly compact Hopf algebra $\hat\bfW_\P$, or the species coalgebroid $\sW$. 
Our starting point is the following:

\begin{definition}
A \emph{word relation} is an equivalence relation $\sim$ on words
with the property that $v\sim w$ only if $v$ and $w$ share the same set of letters, not necessarily with the same multiplicities.
\end{definition}

\subsection{Algebraic relations}

Recall that 
 $w\uparrow m$ and $w\downarrow m$ are formed from $w$ by adding and subtracting $m$ to each letter.

\begin{definition}\label{alg-def} A word relation $\sim$ is \emph{algebraic} if
for all words $v$ and $w$, the following holds:
\ben

\item[(a)] If $v'$, $w'$ are words with $v\sim v'$ and $w\sim w'$, then $v  w\sim v'w'$.

\item[(b)] If $v\sim w$ and $I = \{m+1,m+2,\dots,n\}$ for $m,n\in \NN$, then $(v \cap I) \downarrow m \sim (w\cap I) \downarrow m$.

\een
\end{definition}

Condition (a) states that $\sim$ is a \emph{congruence} on the free monoid on $\PP$,
and is equivalent to requiring that $vxw \sim vyw$ whenever $v,w,x,y$ are words with $x\sim y$.
A typical example of an algebraic word relation is \emph{$K$-Knuth equivalence} \cite[Definition 5.3]{BuchSamuel},
 the strongest congruence with
$bac \sim bca$,
$acb \sim cab$,
$aba \sim bab$,
and
$a\sim aa$
for all integers $a<b<c$. For this relation, Definition~\ref{alg-def}(b) can be checked directly; see also Proposition~\ref{g-prop}.

Fix a word relation $\sim$ and suppose $v$ and $w$ are words.
We note two basic facts:

\begin{lemma}\label{bas1-lem}
If $\sim$ is algebraic and $v\sim w$, then $v\cap [n] \sim w\cap[n]$ for all $n \in \NN$.
\end{lemma}

\begin{proof}
Take $m=0$ in condition (b) in Definition~\ref{alg-def}.
\end{proof}

\begin{lemma}\label{bas2-lem}
If $\sim$ is algebraic and $v\uparrow m \sim w\uparrow m$ for some $m \in \NN$,
then $v\sim w$.
\end{lemma}

\begin{proof}
If $\tilde v :=v\uparrow m \sim w\uparrow m =:  \tilde w$,
then $v = (\tilde v \cap I) \downarrow m \sim (\tilde w \cap I) \downarrow m=w$ for $I = m + \PP$.
\end{proof}

Given a set $E$ of words with letters in $[n]$ and a bijection $\lambda : [n] \to S$,
define
\be\label{kappa-eq} \kappa_E^\lambda = \sum_{w \in E} [w,\lambda] \in \hat\bfW_\lambda \subset \sW[S].\ee
For each finite set $S$ of size $n \in \NN$,
let $\sSigma_S^{({\sim})}$ be
the set of elements of the form $\kappa_E^\lambda$
where  $E$ is a $\sim$-equivalence class of words with letters in $[n]$ 
and $\lambda$ is a bijection $[n] \to S$.
Let $\sS^{({\sim})}[S]$ be the linearly compact $\kk$-vector space with $\sSigma_S^{({\sim})}$ as a 
pseudobasis. The linearly compact topology on this space is the same as the subspace 
topology induced by $\sW[S]$.
Continuous maps to or from $\sW[S]$ therefore remain continuous when restricted to $\sS^{({\sim})}[S]$.
It follows that
\be\label{sS-eq} \sS^{(\sim)} : \FB \to \whkVec\ee
defines a subspecies of $\sW$.

\begin{theorem} \label{alg-thm}
Suppose $\sim$ is a word relation.
Then the species $\sS^{(\sim)} : \FB \to \whkVec$ is a sub-coalgebroid of
$(\sW,\nabla_\shuffle,\iota_\shuffle, \Delta_\odot,\epsilon_\odot)$
if and only if $\sim$ is algebraic.
\end{theorem}

\begin{proof}
The definition of a word relation implies that the empty word $\emptyset$ belongs to its own $\sim$-equivalence class,
so the element $[\emptyset, \id_\varnothing]\in\sW[\varnothing]$ also belongs to $\sS^{(\sim)}[\varnothing]$.
This observation shows that $\iota_\shuffle$ always restricts to a natural transformation $\One \to \sS^{(\sim)}$.
By the comments after Definition~\ref{extended-def},
$\sS^{(\sim)}$ is a sub-coalgebroid of $\sW$
if and only if   
$\Delta_\odot(\sS^{(\sim)}[S]) \subset \sS^{(\sim)}[S] \htimes \sS^{(\sim)}[S]$  for each finite set $S$
and 
 $\nabla_\shuffle$  restricts to a
natural transformation $\sS^{(\sim)} \cdot \sS^{(\sim)} \to \sS^{(\sim)}$.

Condition (a) in Definition~\ref{alg-def}  holds
if and only if $\Delta_\odot(\kappa^\lambda_E) 
\in\sS^{({\sim})}[S] \htimes \sS^{({\sim})}[S]$ for each bijection $\lambda : [n] \to S$
and basis element $\kappa_E^\lambda \in \sSigma_S^{({\sim})}$.
Condition (b) in Definition~\ref{alg-def}
holds if and only if for all words $v,w$ with $v\sim w$ and all integers $n \in \NN$, 
we have both $v\cap [n] \sim w\cap [n]$ and $(v\cap I) \downarrow n \sim (w\cap I)\downarrow n$ for $I = n + \PP$.
By taking $E$ and $F$ to be the $\sim$-equivalence classes
of $v\cap [n]$ and $(v\cap I) \downarrow n$,
one checks that this property is necessary and sufficient to have $\nabla_\shuffle(\kappa_E^\lambda \otimes \kappa_F^\mu)\in\sS^{(\sim)}[S\sqcup T]$ for 
all disjoint finite sets $S$ and $T$ and basis elements
$\kappa_E^\lambda \in \sSigma^{(\sim)}_S$ and $\kappa_F^\mu \in \sSigma^{(\sim)}_T$.
This suffices to show that $\sS^{(\sim)}$ is a sub-coalgebroid if and only if $\sim$ is algebraic.
\end{proof}

Continue to let $\sim$ be a word relation.
For $n \in \NN$,
write $\kappa_E^n$ in place of $\kappa_E^{\lambda}$
when $\lambda$ is the identity map $[n] \to [n]$, and let 
 $\sSigma_n^{({\sim})} = \sSigma_{[n]}^{({\sim})} \cap \hat\bfW_n$ be the set of elements
$\kappa_E^n$ where $E$ ranges over all $\sim$-equivalence classes of words with letters in $[n]$.
Define
\be\label{bfSigma-eq}
\bfSigma_n^{({\sim})} = \kk \sSigma_n^{({\sim})}
\qquand\bfSigma^{({\sim})} = \bigoplus_{n \in \NN} \bfSigma_n^{({\sim})}.
\ee
The vector space $\bfSigma^{({\sim})}$ is a subspace of $\hat\bfW$ but 
is considered to be a discrete topological space.
We say that $\sim$ is of \emph{finite-type} if for each $n \in \NN$, 
the space $\bfSigma_n^{(\sim)}$ is finite-dimensional, or equivalently 
if the set of words with letters in $[n]$ decomposes as a union of  finitely many $\sim$-equivalence classes.

\begin{corollary} \label{alg-cor1}
If $\sim$ is algebraic and of finite-type then
$(\bfSigma^{(\sim)},\nabla_\shuffle,\iota_\shuffle, \Delta_\odot,\epsilon_\odot) \in \Bimon(\kVec)$.
\end{corollary}

\begin{proof}
If $\sim$ is algebraic and of finite-type, then the species coalgebroid $(\sS^{(\sim)} ,\nabla_\shuffle,\iota_\shuffle, \Delta_\odot,\epsilon_\odot)$
is finite-dimensional and its image under the functor \eqref{sigma-eq} is isomorphic to $(\bfSigma^{(\sim)},\nabla_\shuffle,\iota_\shuffle, \Delta_\odot,\epsilon_\odot)$.
\end{proof}

The relation $\sim$ is \emph{homogeneous} if $v\sim w$ implies that $v$ and $w$ have the same length.
When this holds, each equivalence class in $ \Words_n$ is finite so $\bfSigma^{(\sim)}\subset \bfW$,
and each $\kappa_E^n \in \sSigma_n^{(\sim)}$  is homogeneous.

\begin{theorem} \label{alg-cor}
Suppose $\sim$ is a homogeneous word relation.
The vector space $\bfSigma^{(\sim)}$ is a graded sub-bialgebra of $(\bfW,\nabla_\shuffle,\iota_\shuffle, \Delta_\odot,\epsilon_\odot) \in \Bimon(\kVec)$
if and only if $\sim$ is algebraic.
\end{theorem}

\begin{proof}
The argument 
is the same as in the proof of Theorem~\ref{alg-thm}, \emph{mutatis mutandis}.
\end{proof}

A word of minimal length in its $\sim$-equivalence class is \emph{reduced}.
A pair $[w,n] \in \Words_n$ is \emph{reduced} with respect to $\sim$
if $w$ is reduced.
Let $\Words_\Reduced^{({\sim})}$ be the set of reduced elements in $\Words = \bigsqcup_{n \in \NN} \Words_n$.
Define $\sSigma_\Reduced^{(\sim)} $ to be the set of elements of the form
$ \kappa^n_{E} \in \bfW$
where $n \in \NN$
and
$E$ is the (finite) subset of reduced elements in a single $\sim$-equivalence class of words with letters in $[n]$.
Finally, let 
\be\label{bfW-red-eq} \bfW_\Reduced^{({\sim})} =\kk \Words_\Reduced^{(\sim)}
\qquand
\bfSigma_\Reduced^{(\sim)} = \kk \sSigma_\Reduced^{(\sim)}.
\ee
If $\sim$ is homogeneous then $\bfW_\Reduced^{(\sim)} = \bfW$ and $\bfSigma_\Reduced^{(\sim)} = \bfSigma^{(\sim)}$.

\begin{proposition} \label{red-prop}
Suppose ${\sim}$ is an algebraic word relation.
Then $\bfSigma_\Reduced^{(\sim)}$ and $\bfW_\Reduced^{(\sim)}$ are
sub-bialgebras of $(\bfW,\nabla_\shuffle,\iota_\shuffle,\Delta_\odot,\epsilon_\odot)$.
\end{proposition}

\begin{proof}
Conditions (a) and (b) in Definition~\ref{alg-def} respectively imply that (1) if $v$ and $w$ are words such that $vw$ is reduced then $v$ and $w$ are reduced,
and (2) if $v$ and $w$ are reduced words with $\max(v) \leq m$ then every term in the sum $v \shuffle (w \uparrow m)$ is reduced.
One concludes that $ \bfW_\Reduced^{({\sim})} $ is a sub-bialgebra of $\bfW$.

Condition (a) in Definition~\ref{alg-def}
implies that 
if $E$ is the set of reduced elements in a single $\sim$-equivalence class of words with letters in $[n]$
then
$\Delta_\odot(\kappa^n_{E}) $ is a finite sum of tensors of the form $\kappa^n_{F} \otimes \kappa^n_{G}$
where $F$ and $G$ are also the sets of reduced elements in $\sim$-equivalence classes of words with letters in $[n]$.
Thus $\bfSigma_\Reduced^{(\sim)}$ is a sub-coalgebra of $\bfW_\Reduced^{(\sim)}$.
It follows similarly from
condition (b) in Definition~\ref{alg-def} that $\bfSigma_\Reduced^{(\sim)}$ is subalgebra of $\bfW_\Reduced^{(\sim)}$.
Thus $\bfSigma_\Reduced^{(\sim)}$ is a sub-bialgebra of $\bfW_\Reduced^{(\sim)}$.
\end{proof}

\subsection{$\P$-algebraic relations}

To adapt Theorem~\ref{alg-thm} to packed words,
a somewhat technical variation of Definition~\ref{alg-def} is needed.
If $u,v \in \Packed$ are two packed words,
then we say that $w \in \Packed$ is a \emph{$(u,v)$-destandardization}
if there are (not necessarily packed) words $\tilde u, \tilde v$ such that $w = \tilde u \tilde v$ and 
$u = \st(\tilde u)$ and $v = \st(\tilde v)$.
For example, $1234$, $1324$, and $1423$ are $(12,12)$-destandardizations, as is $1212$.

\begin{definition}\label{alg-def3}
A word relation $\sim$ is \emph{$\P$-algebraic}
if for all  $v, w \in \Words_\P$, the following holds:
\ben
\item[(a)] Let $v', w' \in \Words_\P$ with $v\sim v'$ and $w\sim w'$.
 In any $\sim$-equivalence class,
the numbers of $(v,w)$- and $(v',w')$-destandardizations 
are  equal if $\ch \kk =0$ or congruent modulo $p = \ch \kk > 0$.

\item[(b)] If $v\sim w$ and $I = \{m+1,m+2,\dots,n\}$ for $m,n\in \NN$, then $(v \cap I) \downarrow m \sim (w\cap I) \downarrow m$.
\een
Note that property (a) depends on the field $\kk$.
\end{definition}


The set of packed words $\Packed$ is a union of equivalence classes under any word relation.
Let $\sSigma_\P^{(\sim)}$ be the set of sums $\kappa_E := \sum_{w \in E} w \in \hat\bfW_\P$
where $E$ is a $\sim$-equivalence class in $\Packed$.
Define $\bfSigma_\P^{({\sim})}=\kk\sSigma_\P^{(\sim)}$ 
and let $\hat\bfSigma_\P^{({\sim})} \subset \hat\bfW_\P$ be the completion of $\bfSigma_\P^{({\sim})}$
with respect to $\sSigma_\P^{(\sim)}$.

\begin{theorem}\label{packed-alg-thm}
Suppose $\sim$ is a word relation.
Then $\hat\bfSigma_\P^{(\sim)}$ is a linearly compact Hopf subalgebra of $(\hP,\nabla_\shuffle,\iota_\shuffle, \Delta_\odot,\epsilon_\odot)$
if and only if $\sim$ is $\P$-algebraic.
If $\sim$ is homogeneous, then 
$\bfSigma_\P^{(\sim)}$ is a graded Hopf subalgebra of $(\bfP,\nabla_\shuffle,\iota_\shuffle, \Delta_\odot,\epsilon_\odot)$
if and only if $\sim$ is $\P$-algebraic.
\end{theorem}

The part of the theorem asserting that $\bfSigma_\P^{(\sim)}$ is a Hopf algebra when $\sim$ is homogeneous and $\P$-algebraic
is formally similar to
  \cite[Theorem 31]{Hiver07} and \cite[Theorem 2.1]{NovelliThibon}, 
  though neither of these results is a special case of our statement, or vice versa.

\begin{proof}
We first prove the weaker version of the theorem where both instances of ``Hopf subalgebra'' are replaced by ``sub-bialgebra.''
Suppose $v$ and $w$ are packed words and $E \subset \Words_\P$ is a $\sim$-equivalence class.
The coefficient of $v\otimes w$ in $\Delta_\odot(\kappa_E)$
is the number of $(v,w)$-destandardizations in $E$, modulo $p$ if $p=\ch \kk>0$.
We have $\Delta_\odot(\kappa_E) \in \hat\bfSigma_\P^{(\sim)} \htimes \hat\bfSigma_\P^{(\sim)}$
if and only if this coefficient is the same as the corresponding coefficient of $v'\otimes w'$
for any packed words $v',w'$ with $v\sim v'$ and $w\sim w'$.
It follows that $\hat\bfSigma_\P^{(\sim)}$ is a linearly compact sub-coalgebra of $\hat\bfW_\P$
if and only if condition (a) in Definition~\ref{alg-def3} holds.

One has 
 $\nabla_\shuffle(\kappa_E \otimes \kappa_F)\in\hat\bfSigma_\P^{(\sim)}$ for 
all basis elements
$\kappa_E,\kappa_F \in \sSigma^{(\sim)}_\P$
if and only if condition (b) in Definition~\ref{alg-def3} holds by the same reasoning as 
in the proof of Theorem~\ref{alg-thm}.
We conclude that $\hat\bfSigma_\P^{(\sim)}$ is a 
linearly compact sub-bialgebra of $\hat\bfW_\P$ if and only if $\sim$ is $\P$-algebraic.
If $\sim$ is homogeneous
then, in view of Example~\ref{ext-ex}, $\bfSigma_\P^{(\sim)}$ is a graded sub-bialgebra of $\bfW_\P$ 
if and only if 
$\hat\bfSigma_\P^{(\sim)}$ is a 
linearly compact sub-bialgebra of $\hat\bfW_\P$.

To upgrade these conclusions to what is stated in the theorem, 
we first observe that if $\sim$ is homogeneous and $\P$-algebraic then $\bfSigma_\P^{(\sim)}$ is a bialgebra that is graded and connected,
and all such bialgebras are Hopf algebras \cite[Proposition 1.4.16]{GrinbergReiner}.

Next assume $\sim$ is $\P$-algebraic but not necessarily homogeneous.  Then $\hat\bfSigma_\P^{(\sim)}$ is the dual of a bialgebra $\mathbf{H}^{(\sim)}$ with a basis
consisting of all $\sim$-equivalence classes of packed words. Given a packed word $w$, let $\overline w$ denote its $\sim$-equivalence class.
The product in $\mathbf{H}^{(\sim)}$ of the equivalence classes of two packed words $v$ and $w$ is $\nabla(\overline v \otimes \overline w) = \sum_u \overline u$ where the sum is over the finite set of packed words $u$ that are $(v,w)$-destandardizations.
Similarly, the coproduct in  $\mathbf{H}^{(\sim)}$ of the $\sim$-equivalence class of a packed word $w$ with $n =\max(w)$
is $\Delta(\overline w) = \sum_{m=0}^n \overline{w \cap \{1,2,\dots, m\}} \otimes  \overline{(w\downarrow m) \cap \{1,2,\dots, n-m\}}$. Let $\mathbf{H}^{(\sim)}_n\subset \mathbf{H}^{(\sim)}$ be the subspace spanned by all $\sim$-equivalence classes  containing a packed word of length $\leq n$, so that if $w$ is a packed word of length $n$ then $\overline{w} \in\mathbf{H}^{(\sim)}_n$.
Then we have a filtratation
\[\mathbf{H}^{(\sim)}_0 \subset \mathbf{H}^{(\sim)}_1 \subset \mathbf{H}^{(\sim)}_2 \subset \dots \subset \bigcup_{n \in \NN} \mathbf{H}^{(\sim)}_n  =\mathbf{H}^{(\sim)}\]
and the bialgebra $\mathbf{H}^{(\sim)}$ is both \emph{filtered} in the sense that
 \[\nabla\(\mathbf{H}^{(\sim)}_p \otimes \mathbf{H}^{(\sim)}_q\) \subset \mathbf{H}^{(\sim)}_{p+q}
 \quand \Delta\(\mathbf{H}^{(\sim)}_n\) \subset \sum_{p+q =n} \mathbf{H}^{(\sim)}_p \otimes \mathbf{H}^{(\sim)}_q\]
as well as \emph{connected} in the sense that $\dim \mathbf{H}^{(\sim)}_0 = 1$.

Any
connected filtered bialgebra has an antipode given by Takeuchi's formula  (see \cite[Proposition 1.4.24 and Remark 1.4.25]{GrinbergReiner} or \cite[Corollary II.3.2]{Manchon}).
Hence, if $\sim$ is $\P$-algebraic, then $\hat\bfSigma_\P^{(\sim)}$ is a linearly compact Hopf algebra since it is the dual of a Hopf algebra.
More precisely, 
to see that $\hat\bfSigma_\P^{(\sim)}$ is not just a linearly compact Hopf algebra but a linearly compact  Hopf subalgebra of $(\hP,\nabla_\shuffle,\iota_\shuffle, \Delta_\odot,\epsilon_\odot)$, observe that the latter is just $\hat\bfSigma_\P^{(=)}$ and is therefore the dual of $\mathbf{H}^{(=)}$, 
where $=$ is the usual equality relation interpreted as the ($\P$-algebraic) word relation whose equivalence classes all have size one.
But 
 $\mathbf{H}^{(\sim)}$ is evidently a quotient of  $\mathbf{H}^{(=)}$, so under duality 
$\hat\bfSigma_\P^{(\sim)}$ becomes a  linearly compact Hopf subalgebra of $\hat\bfSigma_\P^{(=)}=(\hP,\nabla_\shuffle,\iota_\shuffle, \Delta_\odot,\epsilon_\odot)$.
\end{proof}

\begin{corollary} \label{p-alg-cor1}
If $\sim$ is $\P$-algebraic and of finite-type then
$(\bfSigma_\P^{(\sim)},\nabla_\shuffle,\iota_\shuffle, \Delta_\odot,\epsilon_\odot) \in \Bimon(\kVec)$.
\end{corollary}

This bialgebra is not necessarily graded so may fail to be a Hopf algebra; see Example~\ref{kknuth-ex}.

\begin{proof}
Assume $\sim$ is $\P$-algebraic and of finite-type.
All products and coproducts of basis elements in $\sSigma_{\P}^{(\sim)}$
are finite sums of (tensors of) other basis elements, so belong to $\bfSigma_\P^{(\sim)}$ or $\bfSigma_\P^{(\sim)} \otimes \bfSigma_\P^{(\sim)}$.
The unit element $\emptyset  \in\hat\bfSigma_\P^{(\sim)}$ is also 
in $\bfSigma_\P^{(\sim)}$,
so $(\bfSigma_\P^{(\sim)},\nabla_\shuffle,\iota_\shuffle, \Delta_\odot,\epsilon_\odot)$
is a bialgebra.
\end{proof}

An algebraic word relation is not necessarily $\P$-algebraic, or vice versa.
The following is a natural sufficient condition
for this to occur.

\begin{lemma}\label{suff-lem}
Let $\sim$ be an algebraic word relation.
Assume that whenever $v$ and $w$ are words with the same set of letters
and $\st(v) \sim \st(w)$, it holds that $v \sim w$.
Then $\sim$ is $\P$-algebraic.
\end{lemma}

\begin{proof}
Suppose $v,w,v',w' \in \Words_\P$ and $v\sim v'$ and $w\sim w'$.
For any word $\tilde v$ with $\st(\tilde v) = v$,
there exists a unique word $\tilde v'$ that has the same set of letters as $\tilde v$ and satisfies $\st(\tilde v') = v'$,
and for this word we have $\tilde v \sim \tilde v'$.
Given a word $\tilde w$ with $\st(\tilde w) = w$, define $\tilde w'$ analogously.
The map $\tilde v \tilde w \mapsto \tilde v' \tilde w'$
is then a bijection between
the sets of $(v,w)$- and $(v',w')$-destandardizations 
in any $\sim$-equivalence class, so $\sim$ is $\P$-algebraic.
\end{proof}

\subsection{Uniformly algebraic relations}\label{uniform-sect}

Problematically, we do not know of any efficient method to check whether
an arbitrary word relation satisfies condition (a) in Definition~\ref{alg-def3},
or to generate relations that have this property.
It is therefore useful in practice to consider the following less general type of relation:

\begin{definition}\label{alg-def2}
A word relation $\sim$ is \emph{uniformly algebraic}
if for all words $v,w$, the following holds:
\ben
\item[(a)] If $v',w'$ are words with $v\sim v'$ and $w\sim w'$, then $vw\sim v'w'$.
\item[(b)] If $v\sim w$ and $I \subset \PP$ is an interval (i.e., a set of consecutive integers), then $v\cap I \sim w\cap I$.
\item[(c)] If $v\sim w$ then $\phi(v) \sim \phi(w)$ for any order-preserving injection 
$\phi : [\min(v),\max(v)] \to \PP.$
\een
\end{definition}

Condition (b)  is the property referred to in \cite[\S3.1.2]{Giraudo}, \cite[\S4.3]{Hiver07}, and \cite[Definition 4]{Priez}
as  \emph{compatibility with restriction to alphabet intervals}.
Condition (c) is a weaker form of  
the property referred to in  \cite{Giraudo,Hiver07} as
\emph{compatibility with (de)standardization}.

\begin{lemma}\label{uni-suff-lem}
An algebraic word relation $\sim$ is uniformly algebraic
if and only if $\phi(v) \sim \phi(w)$ whenever $v,w$ are words with
$v\sim w$ and $\phi : [\max(v)] \to \PP$ is an order-preserving injection.
\end{lemma}

\begin{proof}
The given property is a special case of condition (c) in Definition~\ref{alg-def2}, so is certainly necessary.
Let $\sim$ be an algebraic word relation with this property.
If $v \sim w$ and $I=\{m+1,m+2,\dots,n\}$ then $(v\cap I) \downarrow m \sim (w \cap I) \downarrow m$,
and applying the map $\phi : i \mapsto i +m$ to both sides gives
$v\cap I \sim w\cap I$.
Condition (c) in Definition~\ref{alg-def2} holds in view of Lemma~\ref{bas2-lem}.
\end{proof}



\begin{corollary}
A uniformly algebraic word relation is both algebraic and $\P$-algebraic.
\end{corollary}

\begin{proof}
Suppose $\sim$ is uniformly algebraic.
Conditions (b) and (c) in Definition~\ref{alg-def2} together imply condition (b) in Definition~\ref{alg-def}.
Moreover, it follows that 
if $v$ and $w$ are words with the same set of letters
and $\st(v) \sim \st(w)$, then $v \sim w$.
By Lemma~\ref{suff-lem}, $\sim$ is therefore algebraic and $\P$-algebraic.
\end{proof}

Finally, we note a simple way of generating (uniformly) algebraic word relations.

\begin{proposition}\label{g-prop}
Let $\cG$ be a 
set of unordered pairs of words.
Assume that  $v$ and $w$ have the same set of letters if $\{v,w\} \in \cG$,
and if $I$ is an interval then $v\cap I = w\cap I$ or $\{ (v\cap I) \downarrow m,(w\cap I)\downarrow m\} \in \cG$ for some $0 \leq m < \min (I)$.
The reflexive, transitive closure of the relation $\sim$
with \[a(v\downarrow m)b\sim a(w\downarrow m)b\] for all words $a$ and $b$, pairs $\{v,w\} \in \cG$, and integers $0 \leq m < \min(v) =\min(w)$
is then an algebraic word relation.
If it holds that $\{\phi(v),\phi(w)\} \in \cG$ whenever $\{v,w\} \in \cG$ and $\phi : \PP \to \PP$ is an order-preserving injective map,
then $\sim$ is uniformly algebraic.
\end{proposition}

We refer to $\sim$ as the \emph{strongest algebraic word relation} with $v\sim w$ for $\{v,w\} \in \cG$.

\begin{proof}
Condition (a) in Definition~\ref{alg-def}
holds if and only if one has $avb\sim awb$ whenever $a,b,v,w$ are words with $v \sim w$,
which is evidently the case here.
To check condition (b) in Definition~\ref{alg-def},
let $I = \{m+1,m+2,\dots,n\}$ be an interval in $\PP$, fix a pair $\{v,w\} \in \cG$, and let $0\leq k < \min(v) = \min(w)$.
It suffices to show that $\tilde v := ((v \downarrow k) \cap I) \downarrow m \sim ((w \downarrow k) \cap I) \downarrow m =: \tilde w$.
Since $\tilde v = (v \cap J) \downarrow (m+k)$ and 
$\tilde w = (w \cap J) \downarrow (m+k)$
for $J = k + I$,
and since we know that either $v \cap J = w\cap J$ or 
$\{ (v\cap J) \downarrow l,(w\cap J)\downarrow l\} \in \cG$ for an integer $0\leq l \leq m+k$,
the desired conclusion follows.

Now assume that $\{\phi(v),\phi(w)\} \in \cG$ whenever $\{v,w\} \in \cG$ and $\phi : \PP \to \PP$ is an order-preserving injective map.
To show that $\sim$ is uniformly algebraic, it suffices by Lemma~\ref{uni-suff-lem}
to check that
 $\phi(x) \sim \phi(y)$
whenever 
$x$ and $y$ are words with
$x\sim y$ and $\phi : \PP \to \PP$ is an order-preserving injection.
It is enough to show this when $x=a(v\downarrow m)b$ and $y = a (w\downarrow m) b$
for some $\{v,w\} \in \cG$, where $0 \leq m < \min(v) =\min(w)$ and where $a$ and $b$ are arbitrary words.
Observe that
 $\phi(v\downarrow m) = \psi(v) \downarrow m$ and $\phi(w\downarrow m) = \psi(w) \downarrow m$
 where $\psi : \PP \to \PP$ is the map with
 \[ \psi(i) = \begin{cases} i & \text{if }i \leq m \\ \phi(i -m) + m &\text{if }i > m\end{cases}\]
for $i \in \PP$. This map is an order-preserving injection, 
so we have $\{ \psi(v),\psi(w)\} \in \cG$ by hypothesis, and  it also holds that $0 \leq m < \min(\psi(v)) =\min(\psi(w))$.
Thus 
\[\phi(x) = \phi(a) (\psi(v) \downarrow m) \phi(b) \sim \phi(a) (\psi(w) \downarrow m) \phi(b) = \phi(y)\]
holds by the definition of $\sim$,
as desired.
\end{proof}

The following example is instructive when comparing the definitions in this section.
Let $(W,S)$ be a Coxeter system with length function $\ell : W \to \NN$.
There exists a unique associative product $\circ : W \times W \to W$
with the property that $s\circ s = s $ for $s \in S$ and $v \circ w = vw$ if $v,w \in W$ have $\ell(vw) = \ell(v) + \ell(w)$.
One way to derive this claim is to set $a_s=1$ and $b_s=0$ in \cite[Theorem 7.1]{Humphreys}
and then notice that $\{T_w : w \in W\}$ is a monoid under multiplication; alternatively, see 
the discussion in \cite[\S3.10]{RichSpring}.
The resulting monoid $(W,\circ)$ is often called the \emph{$0$-Hecke monoid} or \emph{Richardson-Springer monoid}.
Suppose $S = \{s_1,s_2,s_3,\dots\}$ is countably infinite, and let $\osim$ be 
the equivalence relation  on words with
\[ i_1i_2\cdots i_m \osim j_1j_2\cdots j_n
\quad\text{if and only if}\quad
s_{i_1} \circ s_{i_2} \circ \cdots \circ s_{i_m} = s_{j_1} \circ s_{j_2} \circ \cdots \circ s_{j_n}.\]
Let $m(i,j) \in \PP \sqcup \{\infty\}$ denote the order of $s_is_j\in W$.
Then $m$ can be any map $\PP\times \PP \to \PP \sqcup \{\infty\}$
with $m(i,j) = m(j,i)$ for all $i,j$ and
$m(i,j) = 1$ if and only if $i=j$.
The description of the monoid $(W,\circ)$ by generators and relations in  \cite[\S3.10]{RichSpring} shows
that $\osim$ is the strongest equivalence relation 
that has $vxw\osim vyw$ whenever $x\osim y$
and that has
$a \osim aa $ and $  ababa\cdots  \osim  babab\cdots$ (both sides with $m(a,b)$ terms)
for all $a,b \in \PP$.
In particular, $\osim$ is a word relation.

\begin{lemma}\label{braid-lem}
Let $a,b,n \in \PP$ and set $v = ababa\cdots$  and $w=babab\cdots$ where both words have length $n$.
Then $v\osim w$ if and only if $m(a,b) \leq n$.
\end{lemma}

\begin{proof}
It is clear that $v$ and $w$ are not equivalent under $\osim$ when $ m(a,b) > n$ and that $v\osim w$ when $m(a,b) = n$.
If $m(a,b) < n$ then
by induction $v =aw' \osim av' \osim v' \osim w'  \osim bw' \osim bv'= w$
for the words $v ' =ababa \cdots $ ($n-1$ letters) and $w' = babab\cdots$ ($n-1$ letters).
\end{proof}

\begin{proposition} \label{osim-prop1}
The relation $\osim$ is algebraic if and only if $m(i,j) \leq m(i+1,j+1)$ for all $i,j \in \PP$,
and uniformly algebraic if and only if 
$m(i,j) \leq m(a,b)$ whenever $0<|a-b| \leq |i - j|$.
\end{proposition}

This means that if $\osim$ is uniformly algebraic then $m(i,j) = m(i+1,j+1)$ for all $i,j \in \PP$.

\begin{proof}
Combining Lemmas~\ref{bas2-lem}, \ref{uni-suff-lem}, and \ref{braid-lem} shows that the given conditions are necessary.
Condition (a) in Definition~\ref{alg-def} holds for $\osim$ by construction.

Assume $m(i,j) \leq m(i+1,j+1)$ for all $i,j \in \PP$
and let $I = [k+1,n]$ for some $k, n \in \NN$.
To check condition (b) in Definition~\ref{alg-def},
it suffices 
to show that if 
$v=  ababa\cdots  $ and $w=  babab\cdots$ 
for some $a,b \in \PP$,
 where both words have $m(a,b)$ letters,
then $(v\cap I) \downarrow k \osim (w \cap I) \downarrow k$.
This is clear when $I \cap \{a,b\} \neq \{a,b\}$
and holds when $\{a,b\} \subset I$ by Lemma~\ref{braid-lem}.
Thus $\osim$ is algebraic.
It follows by Lemmas~\ref{uni-suff-lem} and \ref{braid-lem} that the condition for
$\osim$ to be uniformly algebraic is also sufficient.
\end{proof}

A generator $s_i$ belongs to the center $Z(W)$ of $W$ if and only if $m(i,j) = 2$ for all $j \in \PP\setminus\{i\}$.
The group $W$ is abelian if and only if $W=Z(W)$, which occurs when $m(i,j)=2$ for all $i<j$.

\begin{proposition}\label{osim-prop2}
If $W$ is abelian, then 
 $\osim$ is uniformly algebraic and of finite-type.
If $W$ is non-abelian and $p \in \PP$ is minimal such that $s_{p}\notin Z(W)$,
then 
 $\osim$ is algebraic and of finite-type if and only if for some $q\in\PP$ 
it holds that $m(i,i+q) = 3$ and $m(i,j) = 2$ for all $p\leq i<j\neq i+q$.
If these conditions hold, then the word relation $\osim$ is uniformly algebraic when $p=q=1$
but not $\P$-algebraic over any field when $p>1$ or $q>1$.
\end{proposition}

We discuss the bialgebras $\bfSigma^{(\osim)}$ and $\bfSigma_\P^{(\osim)}$
when $\osim$ has these properties   in the next section.

\begin{proof}
The proof depends on the classification of finite Coxeter groups.
The $\osim$-equivalence classes in $\Words_n$ are in bijection with the elements of the parabolic subgroup
$ \langle s_1,s_2,\dots,s_n \rangle \subset W$,
so $\osim$ is of finite-type if and only if each of these subgroups is finite.
In the listed cases, each subgroup of this form is a finite direct product of finite symmetric groups, and is therefore finite.

If $W$ is abelian then both conditions in Proposition~\ref{osim-prop1} obviously hold, so $\osim$ is uniformly algebraic.
Assume $W$ is non-abelian and we have $m(i,i+q) = 3$ and $m(i,j) = 2$ for all $p\leq i<j\neq i+q$,
where $p \in \PP$ is minimal with $s_p \notin Z(W)$.
The first condition in Proposition~\ref{osim-prop1} is clear, and the second condition holds if and only if $p=q=1$.
Hence $\osim$ is uniformly algebraic when $W$ is non-abelian if and only if $p=q=1$.
Assume instead that $p>1$ or $q>1$. In this case we have $12 \osim 21$,
but the $\osim$-equivalence class of the 2-letter word $p(p+q)$ consists of all words of the form $pp\cdots p (p+q)(p+q)\cdots (p+q)$
and so contains exactly one $(12,\emptyset)$-destandardization and no $(21,\emptyset)$-destandardizations.
Thus condition (a) in Definition~\ref{alg-def3} fails
so $\osim$ is not $\P$-algebraic.

 Continue to assume $W$ is non-abelian and $p \in \PP$ is minimal with $s_p \notin Z(W)$.
Suppose $\osim$ is algebraic and of finite-type, 
so that $m(i,j) \leq m(i+1,j+1)$ for all $i,j \in \PP$.
We cannot have $m(i,j) > 3$ for any $i<j$ since then $3 < m(j,2j-i)$
and $\langle s_i, s_j, s_{2j-i}\rangle$ would be infinite.
Since $s_p\notin Z(W)$ but $\{s_1,s_2,\dots,s_{p-1}\} \subset Z(W)$, there exists a minimal $q \in \PP$ such that $m(p,p+q) = 3$.
Then  $m(i,i+q) = 3$ for all $i \geq p$. We cannot have $m(i,j) = 3$ for any $p\leq i<j \neq i+q$ as then we would also have $m(i+q,j+q)=3$
so the Coxeter graph of $(W,S)$
would contain a cycle and some $ \langle s_1,s_2,\dots,s_n \rangle$ would be infinite.
Hence $m(i,i+q) = 3$ and $m(i,j) = 2$ for all $p\leq i<j\neq i+q$.
\end{proof}

\section{Examples}\label{example-sect}

This section presents some further examples of word relations and related bialgebras.

\begin{example}\label{nsym-ex}
Define the \emph{commutation relation} on words to be the relation with $v \sim w$ if
$w$ is formed by rearranging the letters of $v$.
Both $\bfSigma^{(\sim)}\subset \bfW$ and $\bfSigma_\P^{(\sim)}\subset \bfW_\P$
are graded sub-bialgebras since $\sim$ is homogeneous and uniformly algebraic.
Recording multiplicities of the letters in each equivalence class identifies 
 $\sSigma_n^{(\sim)}$ with $\NN^n$.
Given $\alpha \in \NN^n$,
let $[[\alpha]] = \sum_w [w,n]  \in \sSigma_n^{(\sim)}$ where the sum is over all words $w$ with $\max(w) \leq n$ and with exactly  $\alpha_i$ letters equal to $i$.
The product and coproduct of $\bfSigma^{(\sim)}$
then have the formulas 
$\nabla_\shuffle([[\alpha]] \otimes [[\beta]]) = [[\alpha\beta]] $ where $\alpha\beta$ means concatenation
and
$\Delta_\odot([[\alpha]]) = \sum_{ \alpha = \alpha'+\alpha''} [[\alpha']] \otimes [[\alpha'']]$ where the sum is over $\alpha',\alpha'' \in \NN^n$.

Let $H_n \in \sSigma_\P^{(\sim)}$ denote the $n$-letter packed word $111\cdots 1$,
so that $H_0 = \emptyset$ is the unit element in $\bfSigma_\P^{(\sim)}$.
Each $H_n$ is homogeneous of degree $n$, 
and
the algebra structure on $\bfSigma_\P^{(\sim)}$ is just the polynomial algebra $\kk\langle H_1,H_2,\dots\rangle$
where $H_1,H_2,\dots$ are interpreted as non-commuting indeterminates.
The coproduct of $\bfSigma_\P^{(\sim)}$ satisfies $\Delta_\odot(H_n) = \sum_{i=0}^n H_i \otimes H_{n-i}$.
This graded Hopf algebra is commonly known as the 
 algebra of \emph{noncommutative symmetric functions} $\NSym$ \cite{GKLLRT} or \emph{Leibniz-Hopf algebra}.
\end{example}

\begin{example}\label{k-ex}
Define \emph{$K$-equivalence} to be the strongest algebraic word relation with
$a\sim aa$ for all $a \in \PP$.
This is the case of the relation $\osim$ described in the previous section when $(W,S)$ is a universal Coxeter, i.e., when $m(i,j) = \infty$ for all $i<j$.
 $K$-equivalence is therefore uniformly algebraic but neither homogeneous nor of finite-type.
One has $v \sim w$ if and only if $v$ and $w$ coincide after all adjacent repeated letters are combined.

Each equivalence class under $\sim$ contains a unique reduced word with no equal adjacent letters,
which we call a \emph{partial (small) multi-permutation}.
A \emph{(small) multi-permutation} is a partial multi-permutation that is also a packed word.
This notion of a multi-permutation 
is what is intended in \cite[Definition 4.1]{LamPyl},
which omits our condition about being a packed word 
(and so inadvertently gives the definition of a partial multi-permutation).

For a partial multi-permutation $w$ with $\max(w) \leq n$, define $[[w,n]] = \sum_{u \sim w} [u,n] \in \sSigma_n^{(\sim)}$.
Given an arbitrary list $w^1,w^2,\dots$ of distinct partial multi-permutations with letters in $[n]$ and coefficients $c_1,c_2,\dots \in \kk$, 
we abbreviate our notation by setting
\[[[ c_1w^1 + c_2w^2  +\dots ,n]] = c_1 [[w^1,n]] + c_2 [[w^2,n]]    + \dots \in \sS^{(\sim)}[n]\]
and \[[[ w^1\otimes w^2 ,n]] = [[w^1,n]] \otimes [[w^2,n]] \in \sS^{(\sim)}[n]\otimes \sS^{(\sim)}[n].\]
If $v$ and $w$ are partial multi-permutations with letters in $[m]$ and $[n]$, respectively,
then 
\[\nabla_\shuffle([[v,m]] \otimes [[w,n]]) = [[v \star (w \uparrow m), m+n]] \in \sS^{(\sim)}[m+n]\]
 where $\star$ is the \emph{multishuffle product} described by \cite[Proposition 3.1]{LamPyl},
while \[\Delta_\odot([[w,n]]) = [[ \blacktriangle w, n]] \in \sS^{(\sim)}[n] \otimes \sS^{(\sim)}[n]\]
where $\blacktriangle$ is the \emph{cuut coproduct} defined in \cite[\S3]{LamPyl}.
From \eqref{spdelta-eq} and \eqref{spnabla-eq}, these formulas completely determine the (co)product of
the species coalgebroid
$(\sS^{(\sim)},\nabla_\shuffle,\iota_\shuffle, \Delta_\odot,\epsilon_\odot)$.

The linearly compact Hopf algebra $(\hat\bfSigma_\P^{(\sim)},\nabla_\shuffle,\iota_\shuffle, \Delta_\odot,\epsilon_\odot)$
is what Lam and Pylyavskyy call the \emph{small multi-Malvenuto-Reutenauer bialgebra} $\mMR$ \cite[\S4]{LamPyl}.
Theorem~\ref{packed-alg-thm} for the special case of $K$-equivalence recovers  \cite[Theorem 4.2]{LamPyl}, 
which asserts somewhat imprecisely that ``$\mMR$ is a bialgebra'' 
(despite the fact that its coproduct
only makes sense as a map $\mMR \to \mMR \htimes \mMR$).
The linearly compact Hopf algebra $\mMR$ is the algebraic dual of what Lam and Pylyavskyy call the 
\emph{big multi-Malvenuto-Reutenauer Hopf algebra} $\mathfrak{M}\text{MR}$ \cite[\S7]{LamPyl}.
The assertion that $\mathfrak{M}\text{MR}$ has an antipode \cite[Proposition 7.8]{LamPyl}
follows from Theorem~\ref{packed-alg-thm} via this duality.
\end{example}

\begin{example}
Define the \emph{$K$-commutation relation}
to be the transitive closure $\sim$ of $K$-equivalence and the commutation relation.
 This is the weakest word relation,
 in the sense that any word relation is a subrelation of $\sim$.
As the relation $\sim$  is the special case of $\osim$ when $W$ is abelian,
it is uniformly algebraic, inhomogeneous, and of finite-type by Proposition~\ref{osim-prop2}.

 The $\sim$-equivalence classes in $\Words_n$ are in bijection with subsets $I \subset[n]$.
 All packed words $w$ with $\max(w) = n$ belong to the same $\sim$-equivalence class.
 If we let $\kappa_n \in \sSigma_\P^{(\sim)}$ denote the sum of these words,
 then $\kappa_n = \nabla_\shuffle^{(n-1)}(x\otimes x \otimes \cdots \otimes x)$
 for $x = \kappa_1 = 1 + 11 + 111 + \dots$.
Thus
$\bfSigma_\P^{(\sim)}$ coincides as an algebra with $\kk[x]$,
but its coproduct has $\Delta_\odot(x) = x\otimes 1 + x\otimes x + 1\otimes x$. 
This is the $q=1$ version of the \emph{univariate infiltration bialgebra}
discussed, for example, in \cite[\S2.3.3.4]{Hoang}.
\end{example}

\begin{example}\label{knuth-ex}
Define \emph{Knuth equivalence} to be the strongest algebraic word relation with
\[
bac \sim bca,
\qquad
acb \sim cab,
\qquad
aba \sim baa,\qquand bab\sim bba
\]
for all  $a<b<c$.
This relation is of ubiquitous significance in combinatorics.
Its equivalence classes are the sets of words with the same insertion tableau under the RSK correspondence.

Suppose $\lambda = (\lambda_1 \geq \lambda_2 \geq \dots \geq \lambda_m > 0)$ is an integer partition
and $w=w^1w^2\cdots w^m$ is the factorization of a word $w$ into maximal weakly increasing subwords.
Slightly abusing standard terminology, 
say that $w$ is a \emph{semistandard tableau} of shape $\lambda$
if $\ell(w^{i}) = \lambda_{m+1-i}$ for $i \in [m]$ and
$w_j^i > w_j^{i+1}$ whenever both sides are defined. For example, $\emptyset$, $645123$, $2211$, and  $655133$
are semistandard tableaux of the respective shapes $\emptyset$, $(3,2,1)$, $(2,2)$, and $(3,2,1)$.

Each Knuth equivalence class contains a unique semistandard tableau $T$.
When $\max(T) \leq n$,
write $[[T,n]] = \sum_{w \sim T} [w,n] \in \sSigma_n^{(\sim)}$.
Since $T \cap [n]$ is a semistandard tableau whenever $T$ is, 
it follows that the product and coproduct of $\bfSigma^{(\sim)}$
have the formulas
\be
\label{knuth-form}
\nabla_\shuffle([[U,m]]\otimes [[V,n]]) = \sum_T [[T,m+n]]
\quand
\Delta_\odot([[T,n]]) = \sum_{T\sim UV} [[U,n]] \otimes [[V,n]]
\ee
where the first sum is over semistandard tableaux $T$ with $T \cap [m] = U$ and $T\cap [m+1,\infty) \sim V \uparrow m$,
and the second sum is over pairs of semistandard tableaux $U$ and $V$ with $T \sim UV$.

Similar formulas for the (co)product of the Hopf algebra $\bfSigma_\P^{(\sim)}$ 
are noted in \cite{LeclercThibon,PylPat}. 
The subalgebra of $\bfSigma_\P^{(\sim)}$
spanned by Knuth equivalence classes of permutations
 is the \emph{Poirier-Reutenauer Hopf algebra} $\PR$
\cite{PoirierReutenauer}.
Theorem~\ref{packed-alg-thm} for Knuth equivalence recovers \cite[Theorem 3.1]{PoirierReutenauer}.
\end{example}

\begin{example}\label{kknuth-ex}
Recall that  \emph{$K$-Knuth equivalence} is the strongest algebraic word relation with
 \be\label{kk-eq}
bac \sim bca,
\qquad
acb \sim cab,
\qquad
aba \sim bab,
\qquand
a\sim aa\ee
for all integers $a<b<c$.
Proposition~\ref{g-prop} implies that this relation is uniformly algebraic.
Though less well-studied than its homogeneous analogue, $K$-Knuth equivalence 
appears to be an equally fundamental case of interest.
Its relationship with \emph{Hecke insertion} \cite{BKSTY} 
is parallel to that of Knuth equivalence with the RSK correspondence.

Again with minor abuse of standard terminology,
define an \emph{increasing tableau} 
to be a semistandard tableau with no equal adjacent letters,
i.e., in which every weakly increasing consecutive subword is strictly increasing.
For example, the words $\emptyset$ or $645123$ or $5612$ or $545234$ are all increasing tableaux
under our definition.

There are finitely many increasing tableaux with all letters in a given finite set \cite[Lemma 3.2]{PylPat}
and every $K$-Knuth equivalence class contains at least one increasing tableau \cite[Lemma 58]{GMPPRST}.
Thus, $K$-Knuth equivalence is of finite-type
and a somewhat improved way of indexing the elements of $\sSigma^{(\sim)}_n$
is to define $[[T,n]] = \sum_{w \sim T} [w,n]$ for each increasing tableau $T$ with $\max(T) \leq n$. 
The usefulness of this construction is limited, since it is not known how to easily detect when 
two increasing tableaux are $K$-Knuth equivalent.
There is
an algorithm to compute all $K$-Knuth classes of words with a given set of letters, however \cite{GMPPRST}.

It is an open problem to find an irredundant indexing set for $K$-Knuth equivalence classes,
with respect to which one can
describe explicitly the product and coproduct of the bialgebras $\bfSigma^{(\sim)}$ and $\bfSigma_\P^{(\sim)}$.
Patrias and Pylyavskyy \cite{PylPat} refer to the latter as the \emph{$K$-theoretic Poirier-Reutenauer bialgebra} $\KPR$.
 They note that $\KPR$ is not a Hopf algebra \cite[\S4]{PylPat}
 and give some (necessarily inexplicit) formulas for its product and coproduct; see \cite[Theorems 4.3, 4.5, 4.10, and 4.12]{PylPat}.
Theorem~\ref{packed-alg-thm} for $K$-Knuth equivalence recovers \cite[Theorem 4.15]{PylPat}.

As noted in \cite[Remark 5.10]{BuchSamuel} and \cite[\S4]{GMPPRST},
the set of reduced words in a $K$-Knuth equivalence class 
may fail to be spanned by the homogeneous relations $bac\sim bca$, $acb \sim cab$, and $aba\sim bab$ for $a<b<c$.
The graded sub-bialgebra
$\bfSigma_\Reduced^{(\sim)} \subset \bfSigma^{(\sim)}$ 
is thus in some sense not any easier to study.

We mention one other property of this relation.
Define \emph{weak $K$-Knuth equivalence} to be the word relation $\approx$ 
with $v \approx w$ if $v\sim w$ or if $v = v_1v_2v_3\cdots v_n$ and $w=v_2v_1 v_3 \cdots v_n$.
Let $w^\r$ be the word obtained by reversing $w$.
If $v\approx w$ then $v^\r v \sim w^\r w$ since $baab \sim bab \sim aba \sim abba$.
Buch and Samuel state the converse as \cite[Conjecture 7.10]{BuchSamuel}, which appears to be still unresolved:

\begin{conjecture}[Buch and Samuel \cite{BuchSamuel}]\label{bs-conj}
Two words $v$ and $w$ are weakly $K$-Knuth equivalent if and only if $v^\r v $ and $w^\r w$ are $K$-Knuth equivalent.
\end{conjecture}
\end{example}

\begin{example}\label{hecke-ex}
One avoids many pathologies of $K$-Knuth equivalence by considering the stronger relation of
\emph{Hecke equivalence}, which is
the strongest algebraic word relation $\sim$ with
\[
ac \sim ca,
\qquad
aba \sim bab,\qquand a\sim aa
\]
for all positive integers $a<b<c$, so that $13\sim 31$ but $12 \not \sim 21$ \cite[Definition 6.4]{BuchSamuel}.
As explained in Proposition~\ref{osim-prop2}, this is the only case of the relation $\osim$ 
that is uniformly algebraic and of finite-type for which the ambient Coxeter group is non-abelian.
Each set $\sSigma_n^{(\sim)}$ is in bijection with the symmetric group $S_{n+1}$,
which we view 
as the set of words of length $n+1$
containing each $i \in [n+1]$ as a letter exactly once.

Given $\pi \in S_{n+1}$, 
let $[[\pi]] = \sum_w [w,n] \in \sSigma_n^{(\sim)}$
 denote the sum over \emph{Hecke words} for $\pi$, i.e.,
 words $w=w_1w_2\cdots w_m$ 
 with $\pi = s_{w_1} \circ s_{w_2} \circ \cdots \circ s_{w_m}$
where $\circ$ is the product defined
in Section~\ref{uniform-sect} and $s_a=(a,a+1) \in S_{n+1}$. 
The coproduct of  $\bfSigma^{(\sim)}$ satisfies
$\Delta_{\odot}([[\pi]]) = \sum_{\pi = \pi' \circ \pi''} [[\pi']]\otimes [[\pi'']]
$ where the sum is over $\pi',\pi'' \in S_{n+1}$.
It is an open problem to describe the product $\nabla_\shuffle([[\pi']] \otimes [[\pi'']] )$.

We have a better understanding of the graded bialgebra of reduced classes
$ \bfSigma_\Reduced^{(\sim)}$ when $\sim$ is Hecke equivalence.
This bialgebra is the main topic of our complementary paper \cite{M1}, which derives a recursive formula for the 
product of any two basis elements in $\sSigma_\Reduced^{(\sim)}$.

For another point of comparison with $K$-Knuth equivalence,
define \emph{weak Hecke equivalence} to be the word relation $\approx$ 
with $v \approx w$ if $v\sim w$ or if $v = v_1v_2v_3\cdots v_n$ and $w=v_2v_1 v_3 \cdots v_n$.
The analogue of Conjecture~\ref{bs-conj} for Hecke equivalence is known to be true:

\begin{proposition}[{\cite[Theorem 6.4]{HMP2}}]
Two words $v$ and $w$ are weakly Hecke equivalent if and only if $v^\r v $ and $w^\r w$ are Hecke equivalent.
\end{proposition}

\end{example}

\begin{example}
Fix integers $p,q \in \PP$ and define $\approx$ to be the strongest algebraic word relation with
\ben
\item[(i)]  $a(a+q)a \approx (a+q)a(a+q)$ for all integers $a \geq p$,
\item[(ii)] $ab\approx ba$ for all positive integers $a<b$ with $a < p$ or $b \neq a+q$, and
\item[(iii)] $a \approx aa$ for all positive integers $a$.
\een
When $p=q=1$ this relation coincides with Hecke equivalence from Example~\ref{hecke-ex}.
If $\min\{p,q\}>1$ then $\approx$ 
corresponds to the cases of  $\osim$ in Proposition~\ref{osim-prop2} that are algebraic and of finite-type but not $\P$-algebraic.
In the latter situation $\bfSigma^{(\approx)}$ is a bialgebra,
but $\bfSigma_\P^{(\approx)}$ is not a sub-bialgebra of $\bfW_\P$.
If $p=1$ and we write
  $\sim$ for Hecke equivalence,
then the subspace $^{(q)}\bfSigma^{(\approx)} := \bigoplus_{r \in \NN} \bfSigma_{qr}^{(\approx)} \subset \bfSigma^{(\approx)}$
is a sub-bialgebra  with 
$^{(q)}\bfSigma^{(\approx)} \cong \bfSigma^{(\sim)} \otimes \bfSigma^{(\sim)} \otimes \cdots \otimes \bfSigma^{(\sim)}$ ($q$ factors).
\end{example}

\begin{example}\label{referee-ex}
Let $\sim$ be the transitive, reflexive closure of the relation that satisfies 
\[puvuq \sim puvq\quad\text{for all words $p$, $q$, $u$, and $v$.}\]
We refer to this relation as \emph{left-regular band (LRB) equivalence}. 
It is straightforward to check that LRB equivalence is a uniformly algebraic word relation of finite-type. 
The reduced words for this relation are the words with all distinct letters, often called \emph{injective words} or \emph{partial permutations}.
Distinct reduced words for LRB equivalence are never equivalent.
The quotient of the free monoid by $\sim$ is the \emph{free left-regular band} discussed, for example, in \cite[\S1.3]{Brown}.
The packed injective words are precisely the permutations of $[n]$ for all $n \in \NN$, which index a basis
for $\bfSigma_\P^{(\sim)}$. By considering this basis, 
it is easy to see that the algebra structures on $\bfSigma_\P^{(\sim)}$ and 
the Malvenuto-Poirier-Reutenauer Hopf algebra $\FQSym$  mentioned at the end of Section~\ref{word-sect}
are isomorphic. 
The coproduct for  $\bfSigma_\P^{(\sim)}$ is more complicated than for $\FQSym$, however, and is no longer graded.
\end{example}

Many other algebraic word relations 
appear in the literature;
 for example, 
the \emph{hypoplactic relation} (see \cite[Definition 4.16]{KrobThibon} or \cite[Definition 4.2]{Novelli98}),
\emph{sylvester equivalence} (see \cite[Definition 8]{HNT}),
\emph{hyposylvester equivalence} and \emph{metasylvester equivalence} (see \cite[\S3]{NovelliThibon}),
\emph{Baxter equivalence} (see \cite[Definition 3.1]{Giraudo}),
and the \emph{ta\"iga relation} (see \cite[Eq.\ (8)]{Priez})
are all uniformly algebraic, homogeneous word relations.
For sylvester and Baxter equivalence,
the associated Hopf algebra $\bfSigma_\P^{(\sim)}$ recovers the \emph{Loday-Ronco algebra} of planar binary trees \cite{AguiarSottile2,LR1}
and the \emph{Baxter Hopf algebra} of twin binary trees \cite{Giraudo}, respectively.

We mention one other miscellaneous example which will be of significance in Section~\ref{comwor-sect}.

\begin{example}\label{exotic-ex}
Define \emph{exotic Knuth equivalence} to be the strongest algebraic word relation with
\[
bac \sim bca,
\qquad
acb \sim cab,
\qquad
bba \sim bab \sim abb,
\qquand
xyzy\sim yzyx
\]
for all positive integers $a<b<c$ and $x\leq y < z$.
Proposition~\ref{g-prop}  implies that $\sim$ is homogeneous and uniformly algebraic.
This relation does not seem to have been studied previously.
A sensible invariant to consider is the 
sequence $(d_n)_{n= 0,1,2,\dots}$ giving the graded dimension of $\bfSigma_\P^{(\sim)}$,
i.e., in which $d_n$ counts the $\sim$-equivalence classes of packed words of length $n$.
This sequence starts as
\[ (d_n)_{n=0,1,2,\dots} = (1, 1, 3, 9, 31, 110, 412, 1597, 6465, 27021\dots)\]
but does not match any existing entry in \cite{OEIS}.
\end{example}

\section{Combinatorial bialgebras}\label{cb-sect}


A \emph{composition} $\alpha$ of $n \in \NN$, written $\alpha \vDash n$,
is a sequence of positive integers $\alpha=(\alpha_1,\alpha_2,\dots,\alpha_l)$
with $\alpha_1 + \alpha_2 + \dots + \alpha_l = n$. 
The nonzero numbers $\alpha_i$ are the \emph{parts} of the composition.
The unique composition of $n=0$ is the empty word $\emptyset$.
Let $\kk[[x_1,x_2,\dots]]$ be the algebra of formal power series with coefficients in $\kk$ in a countable set of commuting variables.
The \emph{monomial quasi-symmetric function} $M_\alpha$ indexed by a
composition $\alpha\vDash n $ with $l$ parts
is 
\[ M_\alpha = \sum_{i_1<i_2<\dots<i_l} x_{i_1}^{\alpha_1} x_{i_2}^{\alpha_2}\cdots x_{i_l}^{\alpha_l} 
\in \kk[[x_1,x_2,\dots]].\]
When $\alpha$ is the empty composition, set $M_\emptyset=1$.

For each $n \in \NN$, the set $\{ M_\alpha : \alpha \vDash n\}$
is a basis for a subspace  $\QSym_n\subset \kk[[x_1,x_2,\dots]]$.
The vector space of \emph{quasi-symmetric functions}
 $\QSym=\bigoplus_{n \in \NN} \QSym_n$
is a subalgebra of $\kk[[x_1,x_2,\dots]]$.
This algebra is a graded Hopf algebra 
whose coproduct is the linear map with
$\Delta(M_\alpha)  = \sum_{\alpha = \beta \gamma} M_\beta \otimes M_\gamma$
and whose counit is the linear map with
$\epsilon(M_\emptyset) = 1$ and $\epsilon(M_\alpha)=0$ for $\alpha \neq \emptyset$ \cite[\S3]{ABS}.

Each $\alpha \vDash n$
can be rearranged to form a partition  of $n$, denoted $ \sort(\alpha)$.
The \emph{monomial symmetric function} indexed by a partition $\lambda$
is $m_\lambda = \sum_{\sort(\alpha) = \lambda} M_\alpha.$
Write $\lambda \vdash n$ when $\lambda$ is a partition of $n$ and 
let $\Sym_n = \kk\spanning\{ m_\lambda : \lambda \vdash n\}$.
The subspace
$\Sym = \bigoplus_{n \in \NN} \Sym_n \subset \QSym$
is the familiar graded Hopf subalgebra of \emph{symmetric functions}.

Let $\NSym = \kk\langle H_1,H_2,\dots\rangle$ be the graded Hopf algebra
of \emph{noncommutative symmetric functions} described in 
Example~\ref{nsym-ex}, that is,   
the $\kk$-algebra of polynomials in non-commuting indeterminates $H_1,H_2,H_3,\dots$,
where $H_n$ has degree $n$ and the coproduct has $\Delta(H_n) = \sum_{i=0}^n H_i\otimes H_{n-i}$.
Given $\alpha \vDash n$ with $l$ parts, let 
$H_\alpha = H_{\alpha_1}H_{\alpha_2}\cdots H_{\alpha_l}$
and define $H_\emptyset = H_0 =1$.
%
 $\NSym$ is the graded dual of $\QSym$ via
the bilinear form $ \NSym \times \QSym \to \kk$  in which
$\{ H_\alpha \}$ and $\{M_\alpha\}$ are dual bases \cite[\S3]{ABS}.


If $\zeta : V \to \kk[t]$ is a map and $a \in \kk$, then let $\zeta|_{t=a} : V \to \kk$
be the map $v \mapsto \zeta(v)(a)$.


\begin{definition}\label{cco-def}
Suppose $(V,\Delta,\epsilon) \in \Comon(\kGrVec)$ is a graded coalgebra.
If $\zeta : V \to \kk[t]$ is a graded linear map
with  $ \zeta|_{t=0}=\epsilon$,
then  $(V,\Delta,\epsilon,\zeta)$ is a \emph{combinatorial coalgebra}.
\end{definition}
%
%

\begin{definition}\label{bco-def}
Suppose $(V,\nabla,\iota,\Delta,\epsilon) \in \Bimon(\kGrVec)$ is a graded bialgebra.
If $\zeta : V \to \kk[t]$ is a graded algebra morphism with  $ \zeta|_{t=0}=\epsilon$,
then $(V,\nabla,\iota,\Delta,\epsilon,\zeta) $ is a \emph{combinatorial bialgebra}.
\end{definition}

A \emph{combinatorial Hopf algebra} is a combinatorial bialgebra $(V,\nabla,\iota,\Delta,\epsilon,\zeta) $ 
in which $(V,\nabla,\iota,\Delta,\epsilon) $ is a Hopf algebra.
These definitions are minor generalizations of the notions of combinatorial coalgebras and Hopf algebras in \cite{ABS}, 
where it is required that $V$ have finite graded dimension and  $\dim V_0 = 1$. 

When the
structure maps are clear from context, we refer to just the pair $(V,\zeta)$ as a combinatorial coalgebra or bialgebra.
A morphism $\phi: (V,\zeta)\to(V',\zeta')$ 
 of combinatorial coalgebras or bialgebras
is a graded coalgebra or bialgebra morphism 
$\phi : V \to V'$ 
satisfying  $\zeta' = \zeta\circ \phi$.
The map 
 $\zeta$ is the \emph{character} of a combinatorial coalgebra or bialgebra $(V,\zeta)$.

\begin{remark}
 Specifying a graded linear map (respectively, algebra morphism) $V\to\kk[t]$
 is equivalent to defining a (multiplicative) linear map $V \to \kk$.
We define the character $\zeta$ to be a map $V \to \kk[t]$ since this extends more naturally to the linearly compact case.
This convention differs from  \cite{ABS,M1}, where the character of a combinatorial coalgebra is defined to be a linear map $V \to \kk$.
\end{remark}

\begin{example}
There is a graded algebra morphism $\zetaq:\QSym\to\kk[t]$ that has
$\zetaq(M_\emptyset) = 1$, 
$\zetaq(M_{(n)})  = t^n$ for each $n\geq 1$,
and $\zetaq(M_\alpha)  = 0$ for all other compositions $\alpha$.
One way to see that the graded linear map $\zetaq$ is an algebra morphism is to observe that it is the restriction of the algebra morphism 
$\kk[[x_1,x_2,\dots]] \to \kk[[t]]$ that sets $x_1=t$ and $x_n=0$ for all $n>1$.
The pair $(\QSym,\zetaq)$ is a combinatorial Hopf algebra.
\end{example}

Suppose $(V,\Delta,\epsilon,\zeta)$ is a combinatorial coalgebra. 
Define $\Delta^{(1)} = \Delta$ and for $m>2$ set
\[\Delta^{(m-1)} = \(\Delta^{(m-2)} \otimes \id\) \circ \Delta : V \to V^{\otimes m}.\]
Let $\zeta_\emptyset  := \zeta|_{t=0}= \epsilon$.
Given $\alpha \vDash n>0$,
let $\zeta_\alpha : V \to \kk$ be
the map
whose value at $v \in V$ is the coefficient of $t^{\alpha_1} \otimes t^{\alpha_2}\otimes \cdots \otimes t^{\alpha_m}$
in the image of $v$ under  the map
\[V \xrightarrow{\Delta^{(m-1)}} V^{\otimes m} 
\xrightarrow{ \zeta^{\otimes m}} \kk[t]^{\otimes m}.\]
Define $\psi : V \to \QSym$ 
by
\be\label{psi-def} \psi(v)= \sum_{\alpha} \zeta_\alpha(v) M_\alpha\qquad\text{for }v \in V\ee
where the sum is over all compositions.
This \emph{a priori} infinite sum belongs to $\QSym$
since if $v\in V_n$ is homogeneous of degree $n \in \NN$ then $\psi(v) = \sum_{\alpha\vDash n} \zeta_\alpha(v) M_\alpha$.
Thus, $\psi$ is a graded linear map.
The pair $(\QSym,\zetaq)$ is the terminal object in the category of combinatorial (co/bi)algebras:

\begin{theorem}[Aguiar, Bergeron, Sottile \cite{ABS}] \label{abs-thm}
Let $(V,\zeta)$ be a combinatorial coalgebra. 
The map \eqref{psi-def} is the unique 
morphism of combinatorial coalgebras $\psi : (V,\zeta) \to (\QSym,\zetaq)$.
If $(V,\zeta)$ is a combinatorial bialgebra, then 
$\psi$ is a morphism of graded bialgebras.
\end{theorem}

\begin{proof}
This result is only slightly more general than \cite[Theorem 4.1]{ABS} and has essentially the same proof.
We sketch the argument.
Let $\whNSym$ denote the completion of $\NSym$ with respect to the basis $\{ H_\alpha\}$.
Since $\NSym$ is a graded algebra, $\whNSym$ is a linearly compact algebra.
Write $\langle\cdot,\cdot\rangle$ for both the tautological form $V\times V^* \to \kk$
and the bilinear form $\QSym \times \whNSym \to \kk$, continuous in the second coordinate,
relative to which the pseudobasis $\{H_\alpha\}\subset \whNSym$ is dual to the basis $\{M_\alpha\} \subset \QSym$.
Both forms are nondegenerate. 
We view the dual space $V^*$ as the linearly compact algebra with unit element $\epsilon$ dual to the coalgebra $V$ via the tautological form.
The linearly compact algebra structure on $\whNSym$ is the one dual to the  
 coalgebra structure on $\QSym$.

Let $[t^n] f$ denote the coefficient of $t^n$ in $f \in \kk[t]$ and define $\zeta_n \in V^*$ by $\zeta_n(v)= [t^n] \zeta(v)$,
so that $\zeta_0 = \epsilon$. Observe that $[t^n] \zetaq(q) = \langle q, H_n\rangle$ for all $n \in \NN$ and $q \in \QSym$.
It follows that there exists a unique coalgebra morphism $\psi : V \to \QSym$ with $\zeta = \zetaq\circ \psi$
if and only if there exists a unique linearly compact algebra morphism $\phi : \whNSym \to V^*$ with $\phi(H_n) = \zeta_n$ for all $n \in \NN$,
and when this occurs, the two maps satisfy
$\langle \psi(v), w\rangle= \langle v, \phi(w)\rangle$ for all $v \in V$ and $w \in \whNSym$.

Write $V^*_{\textsf{gr}}$ for the graded algebra that is the graded dual of the graded coalgebra $V$.
Since the unit of $V^*_{\textsf{gr}}$ is $\epsilon=\zeta_0$, 
there is a unique algebra morphism $\NSym \to V^*_{\textsf{gr}}$ that sends $H_n \mapsto \zeta_n$ for each $n \in \NN$.
As $V_m \subset \ker \zeta_n$ for all $m\neq n$, this morphism is graded, 
so it extends to a unique linearly compact algebra morphism $\phi : \whNSym \to V^*$.
The resulting morphism $\phi : \whNSym \to V^*$ is evidently the unique one satisfying $\phi(H_n) = \zeta_n$ for all $n \in \NN$,
and one has $\phi(H_\alpha) = \zeta_\alpha$ with $\zeta_\alpha$ as in \eqref{psi-def}.
Hence, there exists a unique coalgebra morphism $\psi :V \to \QSym$ satisfying $\zeta = \zetaq\circ \psi$,
and for this map one has $\langle \psi(v), H_\alpha\rangle = \langle v,\phi(H_\alpha)\rangle= \langle v,\zeta_\alpha\rangle = \zeta_\alpha(v)$ for all $v \in V$ and compositions $\alpha$;
in other words, $\psi$ is the graded linear map \eqref{psi-def}.\footnote[1]{Alternatively, one can consider the graded dual of 
the algebra morphism $\NSym \to V^*_{\textsf{gr}}$ sending $H_n \mapsto \zeta_n$
to obtain a map $(V^*_{\textsf{gr}})^*_{\textsf{gr}} \to \QSym$.
The composition $V\to (V^*_{\textsf{gr}})^*_{\textsf{gr}} \to \QSym$ is then a coalgebra morphism
by \cite[Exercise 1.6.1(f)]{GrinbergReiner}, and one can check that it has the same properties as $\psi$.}

Assume $(V,\zeta)$ is a combinatorial bialgebra. 
Use the symbol $\nabla$ to also denote the products of $\kk[t]$ and $\QSym$.
Define $\xi = \nabla \circ (\zeta \otimes \zeta)$.
Then $(V\otimes V, \xi)$ is a combinatorial coalgebra
and it is easy to check that 
 $\nabla\circ (\psi \otimes \psi)$ and $\psi \circ \nabla$ are both morphisms $(V\otimes V ,\xi) \to (\QSym,\zetaq)$.
The uniqueness proved in the previous paragraph implies that $\nabla\circ (\psi \otimes \psi)= \psi \circ \nabla$.
Since $\zeta$ is an algebra morphism, we also have $\psi(1) = 1 \in \QSym$, so $\psi$ is a bialgebra morphism.
\end{proof}

The results discussed so far have linearly compact analogues.
Let
$\kk[[t]]  $ denote the algebra
of formal power series in $t$, viewed as a linearly compact space as in Example~\ref{kk[x]-ex}.
If $V \in \whkVec$ has pseudobasis $\{ v_i : i \in I\}$, 
then a linear map $\phi : V \to \kk[[t]]$ is continuous if and only if 
for each $n \in \NN$, the set of indices $i \in I$ with $[t^n] \phi(v_i) \neq 0$ is finite,
and 
$\phi\(\sum_{i \in I} c_i v_i \) = \sum_{i \in I} c_i \phi(v_i)$ for any $c_i \in \kk$.

%

\begin{definition}
Suppose $(V,\Delta,\epsilon) \in \Comon(\whkVec)$.
 If $\zeta : V \to \kk[[t]]$ is a continuous linear map with $ \zeta|_{t=0}=\epsilon$,
then  $(V,\Delta,\epsilon,\zeta)$ is a \emph{linearly compact combinatorial coalgebra}.
\end{definition}

Unlike Definition~\ref{cco-def}, this definition does not require any grading on the vector space $V$.

\begin{definition}
Suppose $(V,\nabla,\iota,\Delta,\epsilon) \in \Bimon(\whkVec)$.
If $\zeta : V \to \kk[[t]]$ is a morphism of linearly compact algebras with $\zeta|_{t=0}=\epsilon$,
then $(V,\nabla,\iota,\Delta,\epsilon,\zeta) $ is a \emph{linearly compact combinatorial bialgebra}.
\end{definition}

We often refer to just the pair $(V,\zeta)$ as a linearly compact combinatorial coalgebra or bialgebra.
The map $\zeta$ is the \emph{character} of $(V,\zeta)$.
A morphism $\phi: (V,\zeta)\to(V',\zeta')$ 
 of linearly compact combinatorial (co/bi)algebras
is a continuous (co/bi)algebra morphism 
satisfying  $\zeta' = \zeta\circ \phi$.

%
\begin{example}
Define
$
\whQSym = \prod_{n \in \NN} \QSym_n \in \whkVec
$
and
$
\whSym= \prod_{n \in \NN} \Sym_n \in \whkVec
$
to be the completions of $\QSym$
and $\Sym$ with respect to the bases $\{M_\alpha\}$ and $\{m_\lambda\}$.
 The (co)product and (co)unit maps of $\QSym$ extend to make $\whQSym$ into a linearly compact bialgebra
 and $\whSym \subset \whQSym$ into a linearly compact sub-bialgebra.
 The map $\zetaq$ 
extends to a linearly compact algebra morphism
$\whQSym \to \kk[[t]]$ and
$(\whQSym, \zetaq)$ is a linearly compact combinatorial bialgebra.
\end{example}

Suppose $(V,\Delta,\epsilon,\zeta)$ is a linearly compact combinatorial coalgebra. 
Define $\Delta^{(1)} = \Delta$ and  set 
\[\Delta^{(m-1)} = \(\Delta^{(m-2)} \htimes \id\) \circ \Delta : V \to V^{\htimes m}\]
for $m>2$.
Let $\zeta_\emptyset := \zeta|_{t=0}  = \epsilon$. 
Given $\alpha \vDash n>0$,
let $\zeta_\alpha : V \to \kk$ be
the map
whose value at $v \in V$ is the coefficient of $t^{\alpha_1} \otimes t^{\alpha_2}\otimes \cdots \otimes t^{\alpha_m}$
in the image of $v$ under \[V \xrightarrow{\Delta^{(m-1)}} V^{\htimes m} 
\xrightarrow{ \zeta^{\htimes m}} \kk[[t]]^{\htimes m}.\]
Define $\psi :V \to \whQSym$ to be the map
\be\label{hat-psi-def}
\psi(v) = \sum_\alpha \zeta_\alpha(v) M_\alpha\qquad\text{for $v \in V$}\ee
where the sum is over all compositions $\alpha$.
This is the same formula as \eqref{psi-def}, except now the sum
may have infinitely many nonzero terms.

\begin{theorem}\label{cont-abs-thm}
Let $(V,\zeta)$ be a linearly compact combinatorial coalgebra. 
The map \eqref{hat-psi-def} is the unique 
morphism of linearly compact combinatorial coalgebras $\psi : (V,\zeta) \to (\whQSym,\zetaq)$.
If $(V,\zeta)$ is a linearly compact combinatorial bialgebra, then 
$\psi$ is a morphism of linearly compact bialgebras.
\end{theorem}

\begin{proof}
The proof is similar to that of Theorem~\ref{abs-thm}.
Write $\langle\cdot,\cdot\rangle$
 for both the tautological form $V^\vee\times V \to \kk$
and the bilinear form $\NSym \times \whQSym \to \kk$, continuous in the second coordinate,
relative to which the pseudobasis $\{M_\alpha\}\subset \whQSym$ is dual to the basis $\{H_\alpha\} \subset \NSym$.
Both forms are nondegenerate.
We view 
the vector space $V^\vee$ of continuous linear maps $V \to \kk$  as the algebra with unit element $\epsilon$ dual to the linearly compact coalgebra $V$ via the tautological form.
The algebra structure on $\NSym$ is the one dual to the linearly compact coalgebra structure on $\whQSym$.

Define $\zeta_n \in V^\vee$ by $\zeta_n(v)= [t^n] \zeta(v)$,
so that $\zeta_0 = \epsilon$. 
Let $\phi$ be the unique algebra morphism $ \NSym \to V^\vee$ with $\phi(H_n) = \zeta_n$ for all $n \in \NN$.
Since the product in $V^\vee$ of a sequence of continuous linear maps  $f_1,f_2,\dots,f_m:V \to \kk$ is the map 
\[ \nabla_\kk^{(m-1)} (f_1 \htimes f_2 \htimes \dots \htimes f_m) \circ \Delta^{(m-1)} : V \to \kk,\] 
 it follows that if $\alpha=(\alpha_1,\alpha_2,\dots,\alpha_l)$ is a composition then
 $\phi(H_\alpha) = \phi(H_{\alpha_1} H_{\alpha_2} \cdots H_{\alpha_l})=\zeta_\alpha$ with $\zeta_\alpha$ as in \eqref{hat-psi-def}.
The unique map $ \psi : V \to \whQSym$
satisfying
$\langle u, \psi(v)\rangle= \langle \phi(u), v\rangle$ for all $u \in \NSym$ and $v \in V$
is therefore a linearly compact coalgebra morphism
with the formula \eqref{hat-psi-def}.
Since $\zeta = \zetaq\circ \psi$ if and only if $\langle \zeta_n, v \rangle  = \langle H_n, \psi(v)\rangle $ for all $n \in \NN$ and $v \in V$,
it follows that $\psi$ is the unique morphism of linearly compact combinatorial coalgebras $(V,\zeta) \to (\whQSym,\zetaq)$.

Assume $(V,\zeta)$ is a linearly compact combinatorial bialgebra. 
Use the symbol $\nabla$ for the products of $\kk[[t]]$ and $\whQSym$.
Define $\xi = \nabla \circ (\zeta \htimes \zeta)$.
Then
 $(V\htimes V, \xi)$ is a linearly compact combinatorial coalgebra
and the maps
 $\nabla\circ (\psi \htimes \psi)$ and $\psi \circ \nabla$ are morphisms $(V\htimes V ,\xi) \to (\whQSym,\zetaq)$, so  they must be equal.
By definition $\psi(1) = 1 \in \whQSym$, so $\psi$ is a linearly compact bialgebra morphism.
\end{proof}

The preceding result removes the requirement of a grading in Theorem~\ref{abs-thm}, at the cost of working with linearly 
compact spaces. If one needs to work with honest (bi, co)algebras, then
it is still possible to remove the requirement of a grading in  Theorem~\ref{abs-thm},
but then one must impose a technical finiteness condition
to ensure that the sum \eqref{psi-def} belongs to $\QSym$.


Theorems~\ref{abs-thm} and \ref{cont-abs-thm} also have a version for species.
Recall the definition of $\bfE$ from \eqref{nn-eq}.

\begin{definition}
Suppose $(\cV,\nabla,\iota,\Delta,\epsilon) \in \ExtBim$ and $\cZ : \cV \to \bfE(\kk[[t]])$ is a natural transformation
of functors  $\FB\to\whkVec$.
Assume, for all disjoint finite sets $S$ and $T$, that  the following conditions hold:
\ben
\item[(a)]  The tuple $(\cV[S], \Delta_S,\epsilon_S,\cZ_S)$  is a linearly compact combinatorial coalgebra.
\item[(b)] One has $\cZ_\varnothing \circ \iota_\varnothing(1) = 1 \in \kk[[t]] $.
\item[(c)] If $u \in \cV[S]$ and $v \in \cV[T]$ then  $\cZ_{S\sqcup T}\circ \nabla_{ST} (u \otimes v) = \cZ_S(u)\cZ_T(v)$.
\een
Then $(\cV,\nabla,\iota,\Delta,\epsilon,\cZ)$ is a \emph{combinatorial coalgebroid}.
\end{definition}

This definition is similar to the notion of a \emph{combinatorial Hopf monoid} given in \cite[\S5.4]{M0}.
As usual, when the other data is clear from context,
we refer to just $(\cV,\cZ)$ as a combinatorial coalgebroid.
The natural transformation $\cZ : \cV \to \bfE(\kk[[t]])$ is the \emph{character} of $(\cV,\cZ)$.
A morphism  of combinatorial coalgebroids  $ (\cV,\cZ) \to (\cV',\cZ')$
is a morphism of species coalgebroids $\phi : \cV \to \cV'$  such that $\cZ = \cZ'\circ \phi$.

If $(V,\zeta)$ is a linearly compact combinatorial bialgebra, then $(\bfE(V), \bfE(\zeta))$ is a combinatorial coalgebroid.
Let  $\EQSym = \bfE(\whQSym)$ and $\Zetaq = \bfE(\zetaq)$.
Suppose $(\cV,\cZ)$ is a combinatorial coalgebroid. 
Define $\Psi : \cV \to \EQSym$ to be the natural transformation
such that, for each set $S$, the map
$\Psi_S$ is the unique morphism $(\cV[S],\cZ_S) \to (\whQSym,\zetaq)$.
This is well-defined since if $\sigma : S \to T$ is a bijection then 
$
\zetaq \circ \Psi_T \circ \cV[\sigma]  = \cZ_T \circ \cV[\sigma] = \cZ_S
$, so the maps
$\Psi_T \circ \cV[\sigma]$ and $ \EQSym[\sigma]\circ \Psi_S $
must be equal
as both are morphisms
$(\cV[S],\cZ_S) \to (\whQSym,\zetaq)$.

\begin{corollary}\label{extended-abs-cor}
Let $(\cV,\cZ)$ be a combinatorial coalgebroid. 
Then $\Psi$ is the unique 
morphism of combinatorial coalgebroids $(\cV,\cZ) \to (\EQSym,\Zetaq)$.
\end{corollary}

\begin{proof}
By Theorem~\ref{cont-abs-thm},
$\Psi$ is the unique morphism $\cV \to \EQSym$ in the category $\Comon(\whkVec)$-$\Sp$ satisfying $\cZ  = \Zetaq \circ \Psi$.
It remains to show that $\Psi$ is a morphism of species coalgebroids.
For this, it suffices to check that $\Psi_\varnothing \circ \iota_\varnothing(1) = 1 \in \whQSym$ and $\Psi_{S\sqcup T} \circ \nabla_{ST} = \nabla_{ST} \circ (\Psi_S \htimes \Psi_T)$
for all disjoint finite sets $S$ and $T$.
The first property is evident from \eqref{hat-psi-def} since $\cZ_\varnothing \circ \iota_\varnothing(1) = 1 \in \kk[[t]]$ and $\Delta_\varnothing\circ \iota_\varnothing(1) = \iota_\varnothing(1) \otimes \iota_\varnothing(1)$.
The second property follows from Theorem~\ref{cont-abs-thm} since if $V=\cV[S]\htimes \cV[T]$ and $\xi = \nabla_{\kk[[t]]}\circ (\cZ_S \htimes \cZ_T)$
then
$(V, \xi)$ is a linearly compact combinatorial coalgebra,
and both 
$\Psi_{S\sqcup T} \circ \nabla_{ST}$ and $\nabla_{ST} \circ (\Psi_S \htimes \Psi_T)$
are morphisms 
$(V,\xi) \to (\whQSym,\zetaq)$.
\end{proof}

%
%

Suppose  $(V,\nabla,\iota,\Delta,\epsilon)$ is a graded $\kk$-bialgebra.
Let
$\XX(V)$ denote the set of graded linear maps $\zeta : V \to \kk[t]$ for which $(V,\zeta)$ is a combinatorial bialgebra.
This set is a monoid with unit element $\epsilon$ and product $\zeta\zeta' := \nabla_{\kk[t]} \circ (\zeta \otimes \zeta') \circ \Delta$
where $\nabla_{\kk[t]}$ is the product  of  $\kk[t]$.
We refer to $\XX(V)$ as the \emph{character monoid} of $V$.
If $V$ is a Hopf algebra with antipode $\antipode$, then 
$\zeta^{-1} := \zeta\circ \antipode$
is the left and right inverse of
$\zeta \in \XX(V)$, and $\XX(V)$ is a group with some notable properties \cite{ABS}.


If $(V,\nabla,\iota,\Delta,\epsilon)$ is a linearly compact $\kk$-bialgebra 
then we let
$\XX(V)$ denote the set of continuous linear maps $\zeta : V \to \kk[[t]]$ for which $(V,\zeta)$ is a linearly compact combinatorial bialgebra.
This set is again a monoid with unit element $\epsilon$ and product $\zeta\zeta' := \nabla_{\kk[[t]]} \circ (\zeta \otimes \zeta') \circ \Delta$.
In turn,
if $(\cV,\nabla,\iota,\Delta,\epsilon) \in \ExtBim$ is a species coalgebroid
then we define $\XX(\cV)$ to be the set of natural transformations $\cZ : \cV \to \bfE(\kk[[t]])$ for which $(\cV,\cZ)$ is a combinatorial coalgebroid.
This set is yet another monoid with unit element $\epsilon$, in which the product of $\cZ,\cZ' \in \XX(\cV)$
is the morphism $\cZ\cZ' : \cV \to \bfE(\kk[[t]])$ with $(\cZ\cZ')_S :=\cZ_S  \cZ'_S$ for each finite set $S$.

\section{Characters and morphisms}\label{comwor-sect}

In this section, we assume $\kk$ has characteristic zero
and view $\bfW$ as the bialgebra from Theorem~\ref{w-thm}.
Our goal here is to illustrate a variety of cases where
well-known
symmetric and quasi-symmetric functions
may be constructed 
via the morphisms in Theorems~\ref{abs-thm} and \ref{cont-abs-thm} and Corollary~\ref{extended-abs-cor}.

\subsection{Fundamental quasi-symmetric functions}

We start by examining four natural elements of $\XX(\bfW)$.
Let 
$\zeta_{\leq}: \bfW \to \kk[t]$
be the linear map  
\be
\label{same-eq}
\zeta_{\leq}([w,n]) =  \begin{cases} t^{\ell(w)}&\text{if $w$ is weakly increasing}
\\
0&\text{otherwise}
\end{cases}
\qquad\text{for $[w,n] \in \Words$.}
\ee
Define 
$ \zeta_{\geq}$, $\zeta_{<}$, $\zeta_>$ to be the linear maps $\bfW \to \kk[t]$
given by the same formula but with ``weakly increasing''
replaced by ``weakly decreasing,'' ``strictly increasing,'' and ``strictly decreasing.''

\begin{proposition}
For each $\bullet \in \{ {\leq}, {\geq}, {<}, {>}\}$,
we have $\zeta_\bullet \in \XX(\bfW)$.
\end{proposition}

\begin{proof}
This is equivalent to \cite[Proposition 5.4]{M1} and easily checked directly.
\end{proof}

For each  $\bullet \in \{ {\leq}, {\geq}, {<}, {>}\}$, we let $\psi_\bullet $ 
denote the unique morphism
$
 (\bfW,\zeta_\bullet) \to (\QSym,\zetaq)
$.
Given  $\alpha = (\alpha_1,\alpha_2,\dots,\alpha_l) \vDash n$, let
$I(\alpha) = \{\alpha_1, \alpha_1 + \alpha_2, \dots ,\alpha_1 +\alpha_2+ \dots + \alpha_{l-1}\}.$
The map $\alpha \mapsto I(\alpha)$
is a bijection from compositions of $n$ to subsets of $[n-1]$.
Write $\alpha \leq \beta$ if $\alpha,\beta\vDash n$ and $I(\alpha) \subseteq I(\beta)$.
The \emph{fundamental quasi-symmetric function} associated to $\alpha \vDash n$ is 
\[
L_\alpha = \sum_{\alpha\leq \beta} M_\beta 
=
 \sum_{\substack{i_1 \leq i_2\leq \dots\leq i_n \\ i_j < i_{j+1}\text{ if }j \in I(\alpha)}} x_{i_1}x_{i_2}\cdots x_{i_n}\in \QSym_n.
\]
The set $\{L_\alpha : \alpha \vDash n\}$ is
a second basis of $\QSym_n$.
Given $\alpha = (\alpha_1,\alpha_2,\dots,\alpha_l) \vDash n$,
let  $\beta \vDash n$ be such that $I(\beta) = [n-1] \setminus I(\alpha)$
and
define the \emph{reversal}, \emph{complement}, and \emph{transpose} of $\alpha$ to be
\[\alpha^\r = (\alpha_l,\dots,\alpha_2,\alpha_1),
\qquad
\alpha^\c =\beta,
\qquand \alpha^\t = (\alpha^\r)^\c = (\alpha^\c)^\r.
\] 
%
For a word $w=w_1w_2\cdots w_n$, define 
$w^\r = w_n\cdots w_2w_1$
and
$\Des(w) = \{ i \in [n-1] : w_i > w_{i+1}\}.$

\begin{proposition}[{\cite[Proposition 5.5]{M1}}]
\label{words-opsi-prop}
If  $[w,n] \in \Words$, $\alpha \vDash \ell(w)$, and $\Des(w) = I(\alpha)$, then
\[
\psi_{\leq}([w,n]) = L_\alpha,
\quad 
\psi_{>}([w,n]) = L_{\alpha^\c},
\quad
\psi_{\geq}([w^\r,n]) = L_{\alpha^\r},
\quand
\psi_{<}([w^\r,n]) = L_{\alpha^\t}.
\]
\end{proposition}


Suppose $(V,\zeta_V)$ is a combinatorial bialgebra. If $\iota : U \to V$ is an injective graded bialgebra morphism
then $\zeta_U := \zeta_V \circ \iota \in \XX(U)$ and $\iota$ is a morphism $(U,\zeta_U) \to (V,\zeta_V)$.
Similarly, if $\pi : V \to W$ is a surjective graded bialgebra morphism 
with $\ker \pi \subset \ker \zeta_V$
then there exists a unique character $\zeta_W \in \XX(W)$ with $\zeta_V =  \zeta_W \circ \pi$,
and  $\pi$ is a morphism $(V,\zeta_V) \to (W, \zeta_W)$.
In the first case the unique morphism $(U,\zeta_U) \to (\QSym,\zetaq)$
 factors through $(V,\zeta_V)$ 
and in the second case $(V,\zeta_V) \to (\QSym,\zetaq)$ factors through $(W,\zeta_W)$.

Fix a symbol $\bullet \in \{ {\leq}, {\geq}, {<}, {>}\}$.
The bi-ideal $\bfI_\P \subset \bfW$ is contained in $\ker \zeta_\bullet$,
so $\zeta_\bullet $ and $\psi_\bullet $  factor through the quotient map $\pi : \bfW \to \bfW_\P$.
Let $\tilde \zeta_\bullet : \bfW_\P \to \kk[t] $ and $\tilde \psi_\bullet : \bfW_\P \to \QSym$  be the unique maps with $\zeta_\bullet = \tilde\zeta_\bullet \circ \pi$
and $\psi_\bullet = \tilde \psi_\bullet \circ \pi$. Then
$\tilde\zeta_\bullet \in \XX(\bfW_\P)$ and $\tilde \psi_\bullet$ is the unique morphism $(\bfW_\P,\tilde \zeta_\bullet) \to (\QSym,\zetaq)$.
If $\sim$ is a homogeneous $\P$-algebraic word relation, so that $\bfSigma_\P^{(\sim)} \subset \bfW_\P$ is a graded Hopf sub-algebra,
then $\tilde\zeta_\bullet $ restricts to an element of  $\XX(\bfSigma_\P^{(\sim)})$
and $\tilde\psi_\bullet$ restricts to the unique morphism of combinatorial Hopf algebras $(\bfSigma_\P^{(\sim)},\tilde\zeta_\bullet) \to (\QSym,\zetaq)$.

\begin{example} Suppose $\sim$ is the commutation relation from Example~\ref{nsym-ex}.
Recall that  $\NSym$ can be realized as the 
Hopf algebra $\bfSigma_\P^{(\sim)}$ 
by identifying $H_n$ with the $n$-letter word $11\cdots1 \in \sSigma_\P^{(\sim)}$.
The character $\tilde \zeta_\leq \in \XX(\bfSigma_\P^{(\sim)})$ corresponds to the algebra morphism $\NSym \to \kk[t]$ with $H_n \mapsto t^n$,
and $\tilde \psi_\leq(H_n) = L_{(n)} = \sum_{\alpha\vDash n} M_\alpha = \sum_{\lambda \vdash n} m_\lambda =h_n$
is the $n$th homogeneous symmetric function.
Thus $\tilde \psi_\leq$ gives the natural projection $\NSym \to\Sym$
with $H_n \mapsto h_n$ for $n \in \NN$.
\end{example}

If  $\sim$ is an algebraic word relation
so that $\bfSigma_\Reduced^{(\sim)}\subset \bfW$ is a graded sub-bialgebra,
 then $\zeta_\bullet$ restricts to an element of $\XX(\bfSigma_\Reduced^{(\sim)})$
and $\psi_\bullet$ restricts to the unique morphism  $(\bfSigma_\Reduced^{(\sim)},\zeta_\bullet) \to (\QSym,\zetaq)$.

\begin{example}\label{knuth-ex2}
Suppose $\sim$ is the Knuth equivalence relation from Example~\ref{knuth-ex}.
If $\lambda \vdash n$, then
the Schur function $s_\lambda \in \Sym$
has the formula
 $
s_\lambda = \sum_{\alpha\vDash n} d_{\lambda\alpha} L_\alpha 
$ 
where 
$d_{\lambda\alpha}$ is the number of standard tableaux  of shape $\lambda$
with descent set $I(\alpha)$
\cite[Eq.\ (3.18)]{QuasiSchurBook}.
Let $T$ be a semistandard tableau of shape $\lambda$ with $\max(T) \leq n$, and set $[[T,n]] =\sum_{w \sim T} [w,n] \in \sSigma^{(\sim)}$.
The RSK correspondence gives a descent-preserving bijection between the Knuth equivalence class
of $T $
and all standard tableaux of shape $\lambda$,
so   $\psi_{\leq}([[T,n]]) = s_\lambda$.
In turn, since the linear map $\QSym\to\QSym$ with $L_\alpha \mapsto L_{\alpha^\c}$
restricts on $\Sym$ to the map sending $s_\lambda \mapsto s_{\lambda^T}$
where $\lambda^T $ is the transpose of $\lambda$ \cite[\S3.6]{QuasiSchurBook},
 it follows from Proposition~\ref{words-opsi-prop}
that $\psi_{>}([[T,n]]) = s_{\lambda^T}$.
Applying the bialgebra morphism $\psi_\leq$ to the formulas \eqref{knuth-form} for the (co)product of $\bfSigma^{(\sim)}$
gives two versions of the Littlewood-Richardson rule; see \cite[\S2]{PylPat}.
\end{example}

\subsection{Multi-fundamental quasi-symmetric functions}

Fix $\bullet \in \{ {\leq}, {\geq}, {<}, {>}\}$.
Since $\bfW_\P$ has finite graded dimension,
the character $\tilde \zeta_\bullet : \bfW_\P \to \kk[t]$
extends to a continuous linear map $ \hat\bfW_\P \to \kk[[t]]$, which we denote with the same symbol,
and it holds that $\tilde \zeta_\bullet \in \XX(\hat\bfW_\P)$.
The morphism $\tilde\psi_\bullet: (\bfW_\P,\tilde\zeta_\bullet) \to (\QSym,\zetaq)$
likewise extends to a continuous linear map $\hat\bfW_\P \to \whQSym$, which we also denote with the same symbol.
This extension is the unique morphism of linearly compact combinatorial bialgebras 
$(\hat\bfW_\P,\tilde\zeta_\bullet) \to (\whQSym,\zetaq)$.

Given finite, nonempty subsets $S,T\subset \PP$, write $S \preceq T$ if $\max(S) \leq \min(T)$ and $S \prec T$ if $\max(S) < \min(T)$,
and define $x_S = \prod_{i \in S} x_i$.
In \cite[\S5.3]{LamPyl}, Lam and Pylyavskyy define the \emph{multi-fundamental quasi-symmetric function}
of a composition $\alpha \vDash n$ to be the power series
\be\label{multi-l-eq}
\tilde L_\alpha = \sum_{
\substack{
S_1 \preceq S_2 \preceq \dots \preceq S_n \\
S_j \prec S_{j+1}\text{ if }j \in I(\alpha)}
}
x_{S_1} x_{S_2} \cdots x_{S_n}
\in \whQSym
\ee
where the sum is over finite, nonempty sets $S_1,S_2,\dots,S_n$ of positive integers.

If $f \in \kk[[x_1,x_2,\dots]]$,
then we use the shorthand 
$f(\tfrac{x}{1-x})$ to denote
the power series obtained from $f$ by substituting
 $x_i \mapsto \frac{x_i}{1-x_i} = x_i + x_i^2+x_i^3+\dots$ for each $i \in \PP$. 
It is easy to check that if $f \in \QSym$ then $f(\tfrac{x}{1-x}) \in \whQSym$.
Recall that a multi-permutation is a packed word with no adjacent repeated letters.
The functions $\tilde L_\alpha$
arise naturally as the images of the pseudobasis of the
Hopf algebra $\mMR = \hat\bfSigma_\P^{(\sim)}$
when $\sim$ is $K$-equivalence,
under the morphisms $(\mMR, \tilde \zeta_\bullet) \to (\whQSym,\zetaq)$.

\begin{proposition}
Let $\sim$ be the $K$-equivalence relation from Example~\ref{k-ex}.
Suppose $w$ is a multi-permutation
and define $[[w]] = \sum_{v \sim w} v \in \sSigma_\P^{(\sim)}$.
If $\alpha \vDash \ell(w)$ has $\Des(w) = I(\alpha)$,
then  
\[
\tilde\psi_{<}([[w]]) = \tilde L_\alpha,
\quad
 \tilde\psi_{\leq}([[w]]) = \tilde L_\alpha(\tfrac{x}{1-x}),
 \quad \tilde\psi_{>}([[w^\r]]) = \tilde L_{\alpha^\r},
 \quand
\tilde\psi_{\geq}([[w^\r]]) = \tilde L_{\alpha^\r}(\tfrac{x}{1-x}).
\]
\end{proposition}

The first identity is equivalent to \cite[Theorem 5.11]{LamPyl}.

\begin{proof}
Assume $w=w_1w_2\cdots w_n$ has $n$ letters and 
let $m_1,m_2,\dots,m_n \in \PP$.
The first identity holds since, by
Proposition~\ref{words-opsi-prop},
$\tilde\psi_<$
 applied to the word $(w_1w_1\cdots w_1)(w_2w_2\cdots w_2) \cdots (w_nw_n\cdots w_n)\sim w$
with each $w_i$ repeated $m_i$ times
gives the sum in \eqref{multi-l-eq} restricted to subsets with $|S_i| = m_i$.
It follows in a similar way that $\tilde\psi_{\leq}([[w]])$
has the same formula \eqref{multi-l-eq}
except with the sum over finite, nonempty multisets $S_1,S_2,\dots,S_n$,
which is just $\tilde L_\alpha ( \frac{x}{1-x} )$.
 The other identities are proved analogously.
\end{proof}

We shift our attention to
the species coalgebroid $(\sW, \nabla_\shuffle,\iota_\shuffle,\Delta_\odot,\epsilon_\odot)$.
For each $\bullet \in \{ {\leq}, {\geq}, {<}, {>}\}$,
let $\cZ_\bullet : \sW \to \bfE(\kk[[t]])$ be the natural transformation 
whose $S$-component for each finite set $S$ is the
continuous linear map $\sW[S] \to \kk[[t]]$ with
$[w,\lambda] \mapsto \zeta_\bullet([w,n])$ for $[w,\lambda]\in \Words_\lambda \subset \sW[S]$.
Since there are only finitely many words of a given length with all letters at most $n$, the following holds:

\begin{corollary}
For each symbol $\bullet \in \{ {\leq}, {\geq}, {<}, {>}\}$,
it holds that $\cZ_\bullet \in \XX(\sW)$.
\end{corollary}

For each $\bullet \in \{ {\leq}, {\geq}, {<}, {>}\}$,
define $\Psi_\bullet : \sW \to \EQSym$
to be the natural transformation whose $S$-component is the continuous linear map
$\sW[S] \to \whQSym$ with 
$[w,\lambda] \mapsto \psi_\bullet([w,n])$ for each 
finite set $S$ of size $n$ and each pair $[w,\lambda] \in \Words_S$.
The following is apparent from Corollary~\ref{extended-abs-cor}:

\begin{corollary}
For each $\bullet \in \{ {\leq}, {\geq}, {<}, {>}\}$,
it holds that $\Psi_\bullet$ is the
unique morphism of combinatorial coalgebroids $(\sW,\cZ_\bullet) \to (\EQSym,\Zetaq)$.
\end{corollary}


If  $\sim$ is an algebraic word relation
so that $\sS^{(\sim)}\subset \sW$ is sub-coalgebroid,
 then the natural transformation $\cZ_\bullet$ restricts to an element of $\XX(\sS^{(\sim)})$
and $\Psi_\bullet$ restricts to the unique morphism 
of combinatorial coalgebroids $(\sS^{(\sim)},\cZ_\bullet) \to (\EQSym,\Zetaq)$.

\begin{example}
Suppose $\sim$ is the Hecke equivalence relation from Example~\ref{hecke-ex}.
Given $\pi \in S_{n+1}$, 
define $[[\pi]] = \sum_w [w,n] \in \sSigma_n^{(\sim)}$ where the sum is over all Hecke words $w$ for $\pi$,
and let 
\be\label{groth-eq}
\tilde K_\pi = \Psi_>([[\pi]]),
\qquad
J_\pi  = \Psi_{\leq}([[\pi]]),
\qquand G_\pi = (-1)^{\ell(\pi)} \tilde K_\pi(-x_1,-x_2,\dots).
\ee
The  functions $G_\pi $ are the \emph{stable Grothendieck polynomials} \cite{Buch,BKSTY}.
Following \cite{LamPyl,PylPat}, we call $J_\pi$ and $\tilde K_\pi$ the \emph{weak stable Grothendieck polynomials} and \emph{signless stable Grothendieck polynomials}.
Write $\omega$ for the continuous linear involution of $\whQSym$ with $L_\alpha \mapsto L_{\alpha^\t}$.
Proposition~\ref{words-opsi-prop} implies that $J_\pi = \omega(\tilde K_\pi)$.
By \cite[Theorem 6.12]{Buch}, $J_\pi$ and $\tilde K_\pi$ are Schur positive elements of $\whSym$
and $G_\pi \in \whSym$.
\end{example}

One says that $\pi \in S_n$ is \emph{Grassmannian} if $\pi_1<\dots<\pi_p> \pi_{p+1} <\dots<\pi_n$ for some $p \in [n]$.
In this case let $\lambda(\pi)$ be the partition sorting $(\pi_1-1,\pi_2-2,\dots,\pi_p-p)$.
If $\pi$ is Grassmannian then the functions \eqref{groth-eq} depend only on $\lambda(\pi)$.
Given a partition $\lambda$, define
$
\tilde K_\lambda = \tilde K_\pi
$,
$J_\lambda = J_\pi$, 
and
$G_\lambda =G_\pi$,
where 
$\pi \in \bigsqcup_{n \in \NN} S_n$ is any Grassmannian permutation with $\lambda=\lambda(\pi)$.
By \cite[Theorem 1]{BKSTY}, each $J_\pi$ is a finite $\NN$-linear combination of $J_\lambda$'s,
and each $\tilde K_\pi$ is a finite $\NN$-linear combination of $\tilde K_\lambda$'s.

\begin{example}\label{pylpat-ex}
Let $\sim$ be $K$-Knuth equivalence
so that $\KPR = \bfSigma_\P^{(\sim)}$ is the $K$-theoretic Poirier-Reutenauer bialgebra of \cite{PylPat}.
Then $\tilde \psi_{\leq}$ is a morphism
of linearly compact Hopf algebras $\hat\bfSigma_\P^{(\sim)}\to \whSym$ by \cite[Theorem 6.23]{PylPat}.
If $w$ is a packed word and $[[w]] = \sum_{v\sim w} v \in \sSigma_\P^{(\sim)}$,
then one has
\[
\tilde \psi_{\leq}([[w]]) = J_{\lambda_1} + J_{\lambda_2} + \dots + J_{\lambda_m}
\qquand 
\tilde \psi_{>}([[w]]) = \tilde K_{\lambda_1} + \tilde K_{\lambda_2} + \dots + \tilde K_{\lambda_m}
\]
where $\lambda_1,\lambda_2,\dots,\lambda_m$ are the (not necessarily distinct) shapes
of the finite number of increasing tableaux in the $K$-Knuth equivalence class of $w$ \cite[Theorem 6.24]{PylPat}.
\end{example}

\subsection{Peak quasi-symmetric functions}

Recall the monoidal  structure on $\XX(\bfW)$:
if $\zeta,\zeta' \in \XX(\bfW)$ then $\zeta\zeta' := \nabla_{\kk[t]} \circ (\zeta \otimes \zeta') \circ \Delta_\odot \in \XX(\bfW)$.
For any symbols $\bullet,\circ \in \{ {\leq}, {\geq}, {<}, {>}\}$, we can therefore define
$\zeta_{\bullet|\circ} = \zeta_\bullet\zeta_\circ \in \XX(\bfW)$
and let $\psi_{\bullet|\circ}$ be the unique morphism $(\bfW,\zeta_{\bullet|\circ}) \to (\QSym,\zetaq)$.
For example, if $[w,n] \in \Words$ then
\be
\label{up-down-eq} \zeta_{>|\leq}([w,n]) = \begin{cases} 1 & \text{if }w=\emptyset, \\
2t^m&\text{if }w_1>\dots>w_i \leq w_{i+1} \leq \dots \leq w_m \text{ where $1 \leq i \leq m=\ell(w)$} \\
0&\text{otherwise}.
\end{cases}
\ee
Similar formulas hold for the other possibilities of $\zeta_{\bullet|\circ} $.

One calls $\alpha = (\alpha_1,\alpha_2,\dots,\alpha_l) \vDash n$
a \emph{peak composition} if $\alpha_i\geq 2$ for $1 \leq i < l$, i.e., if $1 \notin I(\alpha)$ and  $i \in I(\alpha)$ $\Rightarrow $ $i\pm 1 \notin I(\alpha)$.
The number of peak compositions of $n$ is the $n$th Fibonacci number.
The \emph{peak quasi-symmetric function} \cite[Proposition 2.2]{Stem} of a peak composition $\alpha \vDash n$ is 
\[
K_\alpha = \sum_{\substack{
\beta \vDash n
\\
I(\alpha) \subset I(\beta) \cup (I(\beta)+1)
}} 2^{\ell(\beta)} M_\beta \in \QSym_n.
\]
Such functions are a basis for a graded Hopf subalgebra of $\QSym$,
 called \emph{Stembridge's peak subalgebra} or the \emph{odd subalgebra} \cite[Proposition 6.5]{ABS},
 which we denote by 
$ \OQSym$. 

Let
$\Peak(w) = \{ i \in [2,n-1] : w_{i-1} \leq w_i > w_{i+1}\}$
and
$\Valley(w) = \{i \in [2,n-1] : w_{i-1} \geq w_i < w_{i+1}\}$ 
for a word $w=w_1w_2\cdots w_n$.
For each $\alpha \vDash n$, let $\Lambda(\alpha)\vDash n$ be the peak composition such that 
\[I(\Lambda(\alpha)) = \{ i\geq 2 : i \in I(\alpha),\ i-1\notin I(\alpha)\}.\]
If $w$ is a word and $\alpha \vDash \ell(w)$ and $\Des(w) = I(\alpha)$, then $\Peak(w) = I(\Lambda(\alpha))$.
%
Finally, given a peak composition $\alpha = (\alpha_1,\alpha_2,\dots,\alpha_l)$,
define $\alpha^{\rpk} = (\alpha_l+1, \alpha_{l-1},\dots,\alpha_2,\alpha_1-1)$.

\begin{proposition}[{\cite[Proposition 5.7]{M1}}]\label{words-kpsi-prop}
If  $[w,n] \in \Words$ and  $\alpha,\beta \vDash \ell(w)$ are compositions such that $\Peak(w) = I(\alpha)$ and $\Valley(w) = I(\beta)$, then
\[
\psi_{>|\leq}([w,n]) = K_\alpha,
\text{\ \ }
\psi_{<|\geq}([w,n]) = K_\beta,
\text{\ \ }
\psi_{\geq|<}([w^\r,n]) = K_{\alpha^{\rpk}},
\text{\ \ and\ \ }
\psi_{\leq|>}([w^\r,n]) = K_{\beta^{\rpk}}.
\]
\end{proposition}


For each $\bullet,\circ \in \{ {\leq}, {\geq}, {<}, {>}\}$, write
$\tilde\zeta_{\bullet|\circ} =\tilde\zeta_\bullet \tilde\zeta_\circ  : \bfW_\P \to \kk[t]$
and
$\tilde\psi_{\bullet|\circ} =\tilde\psi_\bullet \tilde\psi_\circ  : \bfW_\P \to \QSym$
for the maps such that $\zeta_{\bullet|\circ} = \tilde\zeta_{\bullet|\circ} \circ \pi$
and
$\psi_{\bullet|\circ} = \tilde\psi_{\bullet|\circ} \circ \pi$
where $\pi : \bfW \to \bfW_\P$ is the quotient map.

\begin{example}
Again suppose $\sim$ is the commutation relation from Example~\ref{nsym-ex},
so that we can identify 
$\NSym\cong \bfSigma_\P^{(\sim)}$ 
by setting $H_n=11\cdots1 \in \sSigma_\P^{(\sim)}$.
The character $\tilde \zeta_{>|\leq} $ corresponds to the algebra morphism $\NSym \to \kk[t]$ with 
 $H_n \mapsto 2t^n$ for $n>0$,
and we have
\[  \tilde \psi_{>|\leq}(H_n)= 
K_{(n)} = \sum_{\alpha\vDash n} 2^{\ell(\alpha)} M_\alpha = \sum_{\lambda \vdash n} 2^{\ell(\lambda)} m_\lambda =q_n\]
where $q_n \in \Sym$ is the symmetric function such that $\sum_{n\geq 0} q_n t^n = \prod_{i\geq 1} \frac{1+x_i t}{1-x_it}$ (see \cite[\S A.1]{Stem}).
Thus, in this case $\tilde\psi_{>|\leq}$ is the composition of the natural projection $\NSym \to \Sym$ with 
the algebra morphism denoted $\theta : \Sym \to \Sym$ in \cite[Remark 3.2]{Stem}.
\end{example}

Define $\OSym = \kk[q_1,q_2,q_3,\dots]$. By \cite[Theorem 3.8]{Stem},
it holds that $\OSym = \Sym \cap \OQSym$
is a graded Hopf subalgebra of
$\Sym$.
This subalgebra has a distinguished basis $\{Q_\lambda\}$ indexed by strict partitions $\lambda$,
known as the \emph{Schur $Q$-functions}; see \cite[\S A.1]{Stem} for the definition.

\begin{example}\label{peak-knuth-ex}
Suppose $\sim$ is Knuth equivalence 
and $T$ is a semistandard tableau of shape $\lambda$ with $\max(T) \leq n$.
The morphism $\psi_{>|\leq} : (\bfSigma^{(\sim)},\zeta_{>|\leq}) \to (\QSym,\zetaq)$
then has
$\psi_{>|\leq}([[T,n]]) = S_\lambda$
where 
$S_\lambda  \in \OSym$ is the \emph{Schur $S$-function} of shape $\lambda$ \cite[Chapter III, \S8, Ex.\ 7]{Macdonald}.
Each $S_\lambda$ is an
$\NN$-linear combination of Schur $Q$-functions, i.e., is \emph{Schur $Q$-positive}.
\end{example}

\begin{example}
If $\sim$ is Hecke equivalence,
then applying $\psi_{>}$ and $\psi_{>|\leq}$
to the elements of the natural basis $\sSigma_\Reduced^{(\sim)}$ of the bialgebra
of reduced classes $\bfSigma_\Reduced^{(\sim)}$
gives the \emph{Stanley symmetric functions} $F_\pi$ and $F^C_\pi$ of types A and C; 
see the discussion in \cite{M1}.
\end{example}

\subsection{Symmetric functions}

Suppose $\sim$ is a uniformly algebraic word relation.
It is natural to ask when the image of $\hat\bfSigma_\P^{(\sim)}$ under
$\tilde\psi_\bullet$ is contained in $\whSym$,
or equivalently when the image of $\sS^{(\sim)}$ under $\Psi_\bullet$ is a subspecies of $\bfE(\whSym)$.
In turn, one can ask when $\tilde\psi_\bullet(\kappa)$
is Schur positive for all elements $\kappa \in \sSigma_\P^{(\sim)}$.

%
%


\begin{theorem}\label{big-thm1}
Let $\sim$ be a uniformly algebraic word relation.
The following are equivalent: 
\ben
\item[(a)] The image of $\hat\bfSigma_\P^{(\sim)}$ under $\tilde\psi_\leq$ is contained in $\whSym$.
\item[(b)] The image of $\hat\bfSigma_\P^{(\sim)}$ under $\tilde\psi_>$ is contained in $\whSym$.
\item[(c)] The relation $\sim$ extends Knuth equivalence or $K$-Knuth equivalence.
\een
Moreover, if these conditions hold and 
$E$ is any $\sim$-equivalence class of packed words, then the symmetric functions $\tilde\psi_{\leq}(\kappa_E)$ and $\tilde\psi_{>}(\kappa_E)$ are both Schur positive.
\end{theorem}

There is a left-handed version of this result, in which the symbols $\leq$ and $>$ are replaced by $\geq$ and $<$, and Knuth equivalence in part (c) is replaced by \emph{reverse Knuth equivalence}: the relation with $v \sim w$ if and only if  $v^\r$ and $w^\r$ are Knuth equivalent. 
One can ask similar questions 
about ($\P$-)algebraic word relations, 
but such relations do not seem to have a nice classification.

\begin{proof}
The continuous linear map $\whQSym \to \whQSym$ with $L_\alpha \mapsto L_{\alpha^\c}$ 
restricts to the continuous linear involution of $\whSym$ with $s_\lambda \mapsto s_{\lambda^T}$,
so parts (a) and (b) are equivalent by Proposition~\ref{words-opsi-prop}.

Suppose (a) holds and write
$f \equiv g$ when $f,g \in \whQSym$ are such that $f-g \in \whSym$.
Consider the six words $w$ of length three involving the letters 1, 2, and 3.
By Proposition~\ref{words-opsi-prop},
we have $\tilde\psi_{\leq}(w) \equiv 0$ unless $w \in \{132,213,231,312\}$,
and
$\tilde\psi_{\leq}(132) \equiv \tilde\psi_{\leq}(231) \equiv M_{(2,1)}$
and  $\tilde \psi_{\leq}(213) \equiv \tilde \psi_{\leq}(312) \equiv M_{(1,2)}$.
To have $\tilde \psi_{\leq}\( \sum_{v \sim w} v\) \equiv 0$ for each of these words,
it must hold that  $132 \sim 213$ and $231 \sim 312$,
or $132 \sim 312$ and $231 \sim 213$. 
The former case implies the latter
since if $132\sim 213$, then $12 = 132 \cap\{1,2\} \sim 213 \cap \{1,2\} = 21$
whence $ab\sim ba$ for all $a,b \in \PP$ as $\sim$ is uniformly algebraic.
We conclude, by uniformity, that $acb \sim cab$ and $bca\sim bac$ for all positive integers $a<b<c$.

Similarly, if $w$ is one of the eight words of length three involving the letters 1 and 2,
then
 $\tilde\psi_{\leq}(w) \equiv 0$ unless $w \in \{121, 221, 211,212\}$,
and $\tilde\psi_{\leq}(121) \equiv \tilde\psi_{\leq}(221) \equiv M_{(2,1)}$
and  $\tilde \psi_{\leq}(211) \equiv \tilde \psi_{\leq}(212) \equiv M_{(1,2)}$.
To have $\tilde \psi_{\leq}\( \sum_{v \sim w} v\) \equiv 0$ for each of these words,
it must hold that $121 \sim 211$ and $212 \sim 221$,
or $121 \sim 212$ and $221 \sim 211$. 
In the first case, the relation $\sim$ extends Knuth equivalence. 
In the second case, we have $a\sim aa$ for all $a \in \PP$ since $1 = 212 \cap \{1\} \sim 121 \cap \{1\} = 11$,
so $\sim$ extends $K$-Knuth equivalence.
Thus (a) $\Rightarrow$ (c).

Examples~\ref{knuth-ex2} and \ref{pylpat-ex}
show that if (c) holds then 
$\tilde\psi_{\leq}(\kappa_E)$ and $\tilde\psi_{>}(\kappa_E)$ are both Schur positive for any $\sim$-equivalence class  $E$.
In particular, (c) $\Rightarrow$ (a).
\end{proof}


\begin{corollary}\label{big-cor1}
Assume $\sim$ is homogeneous and uniformly algebraic.
Then the image of  $\bfSigma^{(\sim)}$ under $\psi_\leq$ (equivalently, $\psi_>$) is a sub-bialgebra of $\Sym$
 if and only if $\sim$ extends Knuth equivalence. 
\end{corollary}

Our last result is an attempt to formulate a version of Theorem~\ref{big-thm1} for the morphisms 
$\tilde\psi_{\bullet|\circ}$.
Recall the definition of \emph{exotic Knuth equivalence} from Example~\ref{exotic-ex}.

\begin{proposition}\label{last-prop}
Let $\sim$ be a uniformly algebraic word relation.
The image of $\hat\bfSigma_\P^{(\sim)}$ under $\tilde\psi_{>|\leq}$ is contained in $\whSym$
 only if $\sim$ extends Knuth, $K$-Knuth, or exotic Knuth equivalence.
\end{proposition}

\begin{proof}
The argument is similar to the proof of Theorem~\ref{big-thm1}, although the calculations 
are harder to carry out by hand.
 Again write
$f \equiv g$ when $f,g \in \whQSym$ are such that $f-g \in \whSym$.
Suppose $\sim$  is a uniformly algebraic word relation
such that
$\tilde\psi_{>|\leq}\(\hat\bfSigma_\P^{(\sim)}\) \subset \whSym$.
If  $a(a+1)\sim (a+1)a$ for some positive integer $a$, then it is easy to deduce from Definition~\ref{alg-def2}
that $ab \sim ba$ for all $a<b$, in which case $\sim$ extends Knuth equivalence.
Therefore assume that $a(a+1)\not\sim (a+1)a$ for all  $a$.

Among the permutations $w \in S_4$,
the eight elements
$1324$, 
$1423$, 
$1432$, 
$2314$, 
$2413$, 
$2431$, 
$3412$, and 
$3421$
have $\tilde \psi_{>|\leq}(w)\equiv 4M_{(1,3)}$,
 the eight elements
$1243$,  
$1342$, 
$ 2143$, 
$ 2341$, 
$ 3142$, 
$ 3241$, 
$ 4132$, and
$ 4231
$
have $\tilde \psi_{>|\leq}(w)\equiv 4M_{(3,1)}$,
and the remaining elements have 
$\tilde \psi_{>|\leq}(w)\equiv 0$.
Since $12 \not\sim 21$ and $23 \not\sim 32$ and $34 \not\sim 43$,
we must have $1423 \sim 1243$ and $3421\sim 3241$.
It follows for $I=\{2,3,4\}$
that $423 = 1423 \cap I \sim 1243 \cap I = 243$
and $342 = 3421 \cap I \sim 3241 \cap I = 324$.
By the uniformity of $\sim$,
we conclude that $cab \sim acb$ and $bca \sim bac$ for all $a<b<c$.

To proceed, first suppose that $a \sim aa$ for $a \in \PP$.
We then have $3231 \sim 3213$ and $3123\sim 1323$ and 
it holds that 
$\tilde \psi_{>|\leq}( 2321) \equiv \tilde \psi_{>|\leq}(3123 + 1323) \equiv 4M_{(1,3)}
$
and
$
\tilde \psi_{>|\leq}(1232) \equiv \tilde \psi_{>|\leq}(3231 +3213) \equiv 4M_{(3,1)}
$,
while all other words of length 4 with letters in $\{1,2,3\}$ belong to $\sim$-equivalence classes $E$
with $\tilde \psi_{>|\leq}(\kappa_E)\equiv 0$.
Since $12 \not \sim 21$,
it must hold that $2321 \sim 3231 \sim 3213$
and $1232 \sim 3123\sim 1323$. Intersecting these relations with the interval $I = \{2,3\}$
shows that $232\sim 323$, which implies that $aba\sim bab$ for all integers $a<b$.
Thus, if $a \sim aa$ then $\sim$ extends $K$-Knuth equivalence.

Instead suppose that $a \not \sim aa$ for all $a \in \PP$.
Then  $3122 \sim 1322$ and
$\tilde \psi_{>|\leq}(2321) \equiv \tilde \psi_{>|\leq}(3122 + 1322) \equiv 4M_{(1,3)}$
and 
$\tilde \psi_{>|\leq}(3221) \equiv \tilde \psi_{>|\leq}(1232) \equiv 4M_{(3,1)}$,
while all other permutations of $1223$ belong to $\sim$-equivalence classes $E$
with $\tilde \psi_{>|\leq}(\kappa_E)\equiv 0$.
One of two cases must then occur:
\begin{itemize}
\item
Suppose $2321 \sim 1232$ and $3221 \sim 3122 \sim 1322$,
so that $abcb \sim bcba$ and $abb \sim bba$ for all $a<b<c$.
Then $2133 \sim 2313$ and $2321 \sim 3213$ and $3123 \sim 1323$, and
$\tilde \psi_{>|\leq}(2133 + 2313) \equiv \tilde \psi_{>|\leq}(3123 + 1323) \equiv 4M_{(1,3)}$
and $\tilde \psi_{>|\leq}(3231 + 3213) \equiv \tilde \psi_{>|\leq}(1332) \equiv 4M_{(3,1)}$,
while all other permutations of $1233$ belong to $\sim$-equivalence classes $E$
with $\tilde \psi_{>|\leq}(\kappa_E)\equiv 0$.
Since $12\not\sim 21$,
 we must have $3231\sim 3213\sim 2133 \sim 2313$  and $3123 \sim 1323 \sim 1332$.
 Intersecting these equivalences with the interval $I = \{2,3\}$ shows that $233\sim 323 \sim 332$,
 so $abb \sim bab \sim bba$ for all $a<b$.
Finally, we must have $abaa \sim aaba$ for all $a<b$ since $\tilde \psi_{>|\leq}(1211) \equiv 4 M_{(1,3)}$
and $\tilde \psi_{>|\leq}(1121) \equiv 4 M_{(3,1)}$, 
while all other words of length 4 with letters in $\{1,2\}$ belong to $\sim$-equivalence classes $E$ with  $\tilde \psi_{>|\leq}(\kappa_E)\equiv 0$.
Thus $\sim$ extends exotic Knuth equivalence.

\item Suppose $2321 \sim 3221$ and $1232 \sim 3122 \sim 1322$,
so that $aba \sim baa$ for all $a<b$.
Then $2133 \sim 2313$ and $3123\sim 1323$ and $3231 \sim 3213$,
and $\tilde \psi_{>|\leq}(2133 + 2313) \equiv \tilde \psi_{>|\leq}(3123+1323) \equiv \tilde \psi_{>|\leq}(3321) \equiv 4M_{(1,3)}$
and $\tilde \psi_{>|\leq}(3231+3213) \equiv \tilde \psi_{>|\leq}(2331) \equiv \tilde \psi_{>|\leq}(1332)\equiv 4M_{(3,1)}$,
while all other permutations of $1233$ belong to $\sim$-equivalence classes $E$
with $\tilde \psi_{>|\leq}(\kappa_E)\equiv 0$.
Since $12\not\sim 21$,
we must have $1332 \sim 3123 \sim 1323$.
Intersecting these equivalences with the interval $I=\{2,3\}$ shows that $332\sim 323$,
so $bba \sim bab$ for all $a<b$ and $\sim$ extends Knuth equivalence.
\end{itemize}
We conclude that the relation $\sim$ must extend Knuth, $K$-Knuth, or exotic Knuth equivalence.
\end{proof}

By Proposition~\ref{words-kpsi-prop}, the image $\tilde\psi_{>|\leq}\(\hat\bfSigma_\P^{(\sim)}\)$
is contained in the completion of $\OQSym$ with respect to its basis of peak quasi-symmetric functions $\{K_\alpha\}$.
By \cite[Theorem 3.8]{Stem}, the intersection of this completion with $\whSym$
is the linearly compact space of formal power series $ \kk[[q_1,q_2,q_3,\dots]]$,
which is also the completion of $\OSym$ with respect to its basis of Schur $Q$-functions.

It follows from Examples~\ref{peak-knuth-ex} and \ref{pylpat-ex} that if $\sim$ extends Knuth equivalence or $K$-Knuth equivalence then $\tilde\psi_{>|\leq}(\kappa)$ is Schur $Q$-positive
for all elements $\kappa \in \sSigma_\P^{(\sim)}$.
If we could prove the following, then we could upgrade the ``only if'' in Proposition~\ref{last-prop} to ``if and only if.''

\begin{conjecture}
If $\sim$ is exotic Knuth equivalence, then $\tilde \psi_{>|\leq}(\kappa) \in \Sym$ for $\kappa \in \sSigma_\P^{(\sim)}$.
\end{conjecture}

An even stronger property appears to be true:

\begin{conjecture}
If $\sim$ is exotic Knuth equivalence, then $ \tilde\psi_{>|\leq}(\kappa)$ is Schur positive for $\kappa \in \sSigma_\P^{(\sim)}$.
\end{conjecture} 

Curiously, $\tilde\psi_{>|\leq}(\kappa)$ is not always Schur $Q$-positive when $\kappa \in \sSigma_\P^{(\sim)}$ and $\sim$ is exotic Knuth equivalence.
We have checked the two conjectures when $\kappa=\kappa_E$ where $E$ is any exotic Knuth equivalence class of words of length at most nine. Among the 27,021 classes $E$ of packed words $w$
with $\ell(w) = 9$, only 35 are such that $\tilde\psi_{>|\leq}(\kappa_E)$ is not Schur $Q$-positive.

%
%
%

\end{document}